\documentclass[12pt]{amsart}
\usepackage{lmodern}
\usepackage{ifthen}
\usepackage{amsfonts}
\usepackage{amsxtra}
\usepackage{amssymb}
\usepackage{array}
\usepackage[margin=1.0in]{geometry}
\usepackage{xcolor}
\definecolor{cite}{rgb}{0.30,0.60,1.00}
\definecolor{url}{rgb}{0.00,0.00,0.80}
\definecolor{link}{rgb}{0.90,0.10,0.20}
\usepackage[colorlinks,linkcolor=link,urlcolor=url,citecolor=cite,pagebackref,breaklinks]{hyperref}
\usepackage{bbm}
\usepackage{mathtools}
\usepackage{mathrsfs}
\usepackage{appendix}
\usepackage[all]{xy}
\usepackage[lite,abbrev,msc-links,alphabetic]{amsrefs}
\usepackage{enumerate}

\usepackage[OT2,T1]{fontenc}
\DeclareSymbolFont{cyrletters}{OT2}{wncyr}{m}{n}
\DeclareMathSymbol{\Sha}{\mathalpha}{cyrletters}{"58}

\usepackage[utf8]{inputenc}

\numberwithin{equation}{section}

\theoremstyle{plain}
\newtheorem{proposition}{Proposition}[section]

\newtheorem{corollary}[proposition]{Corollary}
\newtheorem{lem}[proposition]{Lemma}
\newtheorem{theorem}[proposition]{Theorem}

\theoremstyle{definition}
\newtheorem{definition}[proposition]{Definition}

\theoremstyle{remark}
\newtheorem{remark}[proposition]{Remark}
\newtheorem{example}[proposition]{Example}

\renewcommand{\b}[1]{\mathbf{#1}}
\renewcommand{\c}[1]{\mathcal{#1}}
\renewcommand{\d}[1]{\mathbb{#1}}
\newcommand{\f}[1]{\mathfrak{#1}}
\renewcommand{\r}[1]{\mathrm{#1}}
\newcommand{\s}[1]{\mathscr{#1}}

\renewcommand{\(}{\left(}
\renewcommand{\)}{\right)}
\newcommand{\res}{\mathbin{|}}

\newcommand{\Sec}{\S}

\newcommand{\bA}{\b A}

\newcommand{\bC}{\b C}
\newcommand{\bD}{\b D}
\newcommand{\bE}{\b E}
\newcommand{\bF}{\b F}
\newcommand{\bG}{\b G}

\newcommand{\bP}{\b P}
\newcommand{\bQ}{\b Q}
\newcommand{\bR}{\b R}

\newcommand{\bZ}{\b Z}

\newcommand{\cC}{\c C}
\newcommand{\cD}{\c D}
\newcommand{\cE}{\c E}
\newcommand{\cF}{\c F}

\newcommand{\cH}{\c H}

\newcommand{\cL}{\c L}
\newcommand{\cM}{\c M}

\newcommand{\cO}{\c O}
\newcommand{\cP}{\c P}

\newcommand{\cS}{\c S}
\newcommand{\cT}{\c T}
\newcommand{\cU}{\c U}
\newcommand{\cV}{\c V}
\newcommand{\cW}{\c W}
\newcommand{\cX}{\c X}
\newcommand{\cY}{\c Y}
\newcommand{\cZ}{\c Z}

\newcommand{\dL}{\d L}

\newcommand{\fE}{\f E}

\newcommand{\fP}{\f P}

\newcommand{\fR}{\f R}
\newcommand{\fS}{\f S}

\newcommand{\fV}{\f V}

\newcommand{\fX}{\f X}
\newcommand{\fY}{\f Y}

\newcommand{\fc}{\f c}

\newcommand{\rR}{\r R}

\newcommand{\ra}{\r a}

\newcommand{\rd}{\r d}

\newcommand{\sA}{\s A}

\newcommand{\sD}{\s D}

\newcommand{\sF}{\s F}
\newcommand{\sG}{\s G}

\newcommand{\sK}{\s K}
\newcommand{\sL}{\s L}

\newcommand{\sO}{\s O}

\newcommand{\sT}{\s T}

\newcommand{\pres}[2]{\prescript{#1}{}{#2}}

\newcommand{\ad}{\r{ad}}
\newcommand{\an}{\r{an}}

\newcommand{\cl}{\r{cl}}

\newcommand{\cris}{\r{cris}}

\newcommand{\dr}{\r{dR}}

\newcommand{\et}{\acute{\r{e}}\r{t}}
\newcommand{\Et}{\acute{\r{E}}\r{t}}
\newcommand{\fEt}{\r{f}\acute{\r{E}}\r{t}}

\newcommand{\logcris}{\r{log}\text{-}\r{cris}}

\newcommand{\rig}{\r{rig}}

\DeclareMathOperator{\CH}{CH}

\DeclareMathOperator{\DIV}{div}

\DeclareMathOperator{\Fr}{Fr}

\DeclareMathOperator{\Hom}{Hom}
\DeclareMathOperator{\IM}{im}

\DeclareMathOperator{\Ker}{ker}

\DeclareMathOperator{\Pic}{Pic}

\DeclareMathOperator{\Spec}{Spec}
\DeclareMathOperator{\Spf}{Spf}

\DeclareMathOperator{\Tot}{Tot}
\DeclareMathOperator{\Tr}{Tr}

\DeclareMathOperator{\trop}{trop}

\begin{document}

\title[Weight decomposition and tropical cycle classes]
{Weight decomposition of de Rham cohomology sheaves and tropical cycle
classes for non-Archimedean spaces}

\author{Yifeng Liu}
\address{Department of Mathematics, Northwestern University, Evanston IL 60208, United States}
\email{liuyf@math.northwestern.edu}

\date{\today}

\begin{abstract}
We construct a functorial decomposition of de Rham cohomology sheaves,
called weight decomposition, for smooth analytic spaces over non-Archimedean
fields embeddable into $\bC_p$, which generalizes a construction of Berkovich
and solves a question raised by him. We then investigate complexes of
real tropical differential forms and currents introduced by Chambert-Loir and
Ducros, by establishing a relation with the weight decomposition and defining
tropical cycle maps with values in the corresponding Dolbeault
cohomology. As an application, we show that algebraic cycles that are
cohomologically trivial in the algebraic de Rham cohomology are
cohomologically trivial in the Dolbeault cohomology of currents as well.
\end{abstract}

\maketitle

%\renewcommand\abstractname{R\'esum\'e}
%\begin{abstract}
%[D\'ecomposition par le poids des faisceaux de cohomologie de de Rham et classes de cycle tropicales pour les espaces non-archim\'ediens]
%Nous construisons une d\'ecomposition fonctorielle des faisceaux de cohomologie de de Rham, appel\'ee d\'ecomposition par le poids, pour les %espaces analytiques lisses sur les corps non-archim\'edien plongeables dans $\bC_p$, qui g\'en\'eralise une construction de Berkovich et %r\'esout une question pos\'ee par lui. Nous \'etudions ensuite les complexes de formes diff\'erentielles et de courants tropicaux r\'eels %introduits par Chambert-Loir et Ducros, en \'etablissant une relation avec la d\'ecomposition par le poids et en d\'efinissant des applications %de cycle tropicales \`a valeurs dans la cohomologie de Dolbeault correspondante. Comme application, nous montrons que les cycles alg\'ebriques %qui sont cohomologiquement triviaux dans la cohomologie de Rham alg\'ebrique sont cohomologiquement triviaux dans la cohomologie de Dolbeault %des courants aussi.
%\end{abstract}

\setcounter{tocdepth}{1} \tableofcontents

\section{Introduction}
\label{ss:0}

Let $K$ be a complete non-Archimedean field of characteristic zero with a
nontrivial valuation. Let $X$ be a smooth $K$-analytic space in the sense of
Berkovich. Let $\cO_X$ (resp.\ $\fc_X$) be the structure sheaf (resp.\ the
sheaf of constant analytic functions \cite{Berk04}*{\Sec 8}) of $X$ in either
analytic or \'{e}tale topology. We have the following complex of
$\fc_X$-modules in either analytic or \'{e}tale topology:
\begin{align}\label{eq:derham}
\Omega^\bullet_X\colon 0\to \cO_X=\Omega^0_X \xrightarrow{\rd}
\Omega^1_X \xrightarrow{\rd}\Omega^2_X
\xrightarrow{\rd}\cdots,
\end{align}
known as the \emph{de Rham complex}, which satisfies that
$\fc_X=\Ker(\rd\colon\cO_X\to\Omega^1_X)$. It is \emph{not} exact from the term
$\Omega^1_X$ if $\dim(X)\geq 1$. The cohomology sheaves of the de Rham
complex $\Omega^{q,\cl}_X/\rd\Omega^{q-1}_X$ are called \emph{de Rham
cohomology sheaves}. For $q\geq 0$, denote by $\Upsilon^q_X$ the subsheaf of
$\Omega^{q,\cl}_X/\rd\Omega^{q-1}_X$ generated by sections of the form
\[\sum c_i\frac{\rd f_{i1}}{f_{i1}}\wedge\cdots\wedge\frac{\rd f_{iq}}{f_{iq}}\]
where the sum is finite, $c_i$ are sections of $\fc_X$, and $f_{ij}$ are
sections of $\cO^*_X$. In particular, we have $\Upsilon^0_X=\fc_X$, and that
$\Upsilon^1_X$ is simply the sheaf $\Upsilon_X$ defined in
\cite{Berk07}*{\Sec 4.3} in the case of \'{e}tale topology.

\begin{theorem}\label{th:1}
Suppose that $K$ is embeddable into $\bC_p$. Let $X$ be a smooth $K$-analytic
space. Then for every $q\geq 0$, we have a decomposition
\[\Omega^{q,\cl}_X/\rd\Omega^{q-1}_X=\bigoplus_{w\in\bZ}(\Omega^{q,\cl}_X/\rd\Omega^{q-1}_X)_w\]
of $\fc_X$-modules in either analytic or \'{e}tale topology. It satisfies
that
\begin{enumerate}[(i)]
  \item $(\Omega^{q,\cl}_X/\rd\Omega^{q-1}_X)_w=0$ unless $q\leq w\leq
      2q$;

  \item $\Upsilon^q_X\subset(\Omega^{q,\cl}_X/\rd\Omega^{q-1}_X)_{2q}$,
      and they are equal if $q=1$;

  \item $(\Omega^{1,\cl}_X/\rd\cO_X)_1$ coincides with the sheaf $\Psi_X$
      defined in \cite{Berk07}*{\Sec 4.5} in the case of \'{e}tale
      topology.
\end{enumerate}
Such decomposition is stable under base change, cup product, and functorial in $X$.
\end{theorem}

The proof of this theorem will be given at the end of Section \ref{ss:log}. We
call the decomposition in the above theorem the \emph{weight decomposition of
de Rham cohomology sheaves}.

\begin{corollary}\label{co:1}
Suppose that $K$ is embeddable into $\bC_p$. Then for every smooth $K$-analytic
space $X$, we have $\Omega^{1,\cl}_X/\rd\cO_X=\Upsilon_X\oplus\Psi_X$ in \'{e}tale topology. This answers the question in
\cite{Berk07}*{Remark 4.5.5} for such $K$.
\end{corollary}

\begin{remark}
We expect that Theorem \ref{th:1} and thus Corollary \ref{co:1} hold by only
requiring that the residue field of $K$ is algebraic over a finite field (and
$K$ is of characteristic zero).
\end{remark}

For the rest of Introduction, we work in the analytic topology only. In
particular, the de Rham complex $(\Omega^\bullet_X,\rd)$ is a complex of
sheaves on (the underlying topological space of) $X$.\\

In \cite{CLD12}, Chambert-Loir and Ducros define, for every $K$-analytic
space $X$, a bicomplex $(\sA_X^{\bullet,\bullet},\rd',\rd'')$ of sheaves of
real vector spaces on $X$ concentrated in the first quadrant. It is a
non-Archimedean analogue of the bicomplex of $(p,q)$-forms on complex
manifolds. In particular, we may define analogously the \emph{Dolbeault
cohomology} (of forms) as
\[H^{q,q'}_{\sA}(X)\coloneqq\frac{\Ker(\rd''\colon\sA_X^{q,q'}(X)\to\sA_X^{q,q'+1}(X))}
{\IM(\rd''\colon\sA_X^{q,q'-1}(X)\to\sA_X^{q,q'}(X))}.\] By \cite{CLD12} and
\cite{Jel16}, we know that for every $q\geq 0$, the complex
$(\sA_X^{q,\bullet},\rd'')$ is a fine resolution of the sheaf
$\Ker(\rd''\colon\sA_X^{q,0}\to\sA_X^{q,1})$. Thus, $H^{q,q'}_{\sA}(X)$ is
canonically isomorphic to the sheaf cohomology
$H^{q'}(X,\Ker(\rd''\colon\sA_X^{q,0}\to\sA_X^{q,1}))$. If $X$ is of
dimension $n$ and without boundary, then we may define the integration
\[\int_X\omega\]
for every top form $\omega\in\sA^{n,n}_X(X)$ with compact support. In
particular, if $X$ is moreover compact, then the integration induces a real
linear functional on $H^{n,n}_\sA(X)$.

The next theorem reveals a connection between
$\Ker(\rd''\colon\sA_X^{q,0}\to\sA_X^{q,1})$ and the algebraic de Rham
cohomology sheaves of $X$.

\begin{theorem}[Lemma \ref{le:log}, Theorem \ref{th:kernel}]
Let $K$ be a non-Archimedean field embeddable into $\bC_p$ and $X$ a smooth
$K$-analytic space. Let $\sL_X^q$ be the subsheaf of $\bQ$-vector spaces of
$\Omega^{q,\cl}_X/\rd\Omega^{q-1}_X$ generated by sections of the form
$\frac{\rd f_1}{f_1}\wedge\cdots\wedge\frac{\rd f_q}{f_q}$ where $f_j$ are
sections of $\cO^*_X$. Then
\begin{enumerate}
  \item the canonical map $\sL_X^q\otimes_\bQ\fc_X\to\Upsilon_X^q$ is an
      isomorphism;
  \item there is a canonical isomorphism $\sL_X^q\otimes_\bQ\bR\simeq
      \Ker(\rd''\colon\sA_X^{q,0}\to\sA_X^{q,1})$.
\end{enumerate}
\end{theorem}

The above theorem implies that the Dolbeault cohomology $H^{q,q'}_\sA(X)$ for
$X$ in the theorem has a canonical rational structure through the isomorphism
$H^{q,q'}_\sA(X)\simeq H^{q'}(X,\sL_X^q)\otimes_\bQ\bR$.\\

Recall that in the complex world, for a smooth complex algebraic variety
$\cX$, we have a cycle class map from $\CH^q(\cX)$ to the classical Dolbeault
cohomology $H^{q,q}_{\overline\partial}(\cX^\an)$ of the associated complex
manifold $\cX^\an$. Over a non-Archimedean field $K$, we may associate a
scheme $\cX$ of finite type over $K$ a $K$-analytic space $\cX^\an$. The
following theorem is an analogue of the above cycle class map in the
non-Archimedean world.

\begin{theorem}[Definition \ref{de:tropical}, Theorem \ref{th:cycle}, Corollary \ref{co:cycle}]
Let $K$ be a non-Archimedean field and $\cX$ a smooth scheme over $K$ of
dimension $n$. Then there is a \emph{tropical cycle class map}
\[\cl_\sA\colon\CH^q(\cX)\to H^{q,q}_\sA(\cX^\an),\]
functorial in $\cX$ and $K$, such that for every algebraic cycle $\cZ$ of
$\cX$ of codimension $q$,
\begin{align}\label{eq:integral}
\int_{\cX^\an}\cl_\sA(\cZ)\wedge\omega=\int_{\cZ^\an}\omega
\end{align}
for every $\rd''$-closed form $\omega\in\sA^{n-q,n-q}_{\cX^\an}(\cX^\an)$
with compact support.

In particular, if $\cX$ is proper and $\cZ$ is of dimension $0$, then
\[\int_{\cX^\an}\cl_\sA(\cZ)=\deg\cZ.\]
\end{theorem}

\begin{remark}
Let the situation be as in the above theorem.
\begin{enumerate}
  \item Our construction actually shows that the image of $\cl_\sA$ is in
      the canonical rational subspace $H^q(\cX^\an,\sL_{\cX^\an}^q)$.

  \item The tropical cycle class respects products on both sides. More
      precisely, for $\cZ_i\in\CH^{q_i}(\cX)$ with $i=1,2$, we have
      $\cl_\sA(\cZ_1\cdot\cZ_2)=\cl_\sA(\cZ_1)\cup\cl_\sA(\cZ_2)$.

  \item We may regard the formula \eqref{eq:integral} as a tropical
      version of Cauchy formula in multi-variable complex analysis.

  \item Even when $\cX$ is proper, one can \emph{not} use
      \eqref{eq:integral} to define $\cl_\sA(\cZ)$ as we do not know
      whether the pairing
      \[H^{q,q}_\sA(\cX^\an)\times
      H^{n-q,n-q}_\sA(\cX^\an)\xrightarrow{\cup}H^{n,n}_\sA(\cX^\an)\xrightarrow{\int_X}\bR\]
      is perfect or not at this moment.
\end{enumerate}
\end{remark}

For a proper smooth scheme $\cX$ of dimension $n$ over a general field $K$ of
characteristic zero, we have a cycle class map $\cl_\dr\colon\CH^q(\cX)\to
H^{2q}_\dr(\cX)$ with values in the algebraic de Rham cohomology. It is known
that when $K=\bC$, the kernel of $\cl_\dr$ coincides with the kernel of the
cycle class map valued in Dolbeault cohomology. In particular, if
$\cl_\dr(\cZ)=0$, then $\int_{\cZ^\an}\omega=0$ for every
$\overline\partial$-closed $(n-q,n-q)$-form $\omega$ on $\cX^\an$. In the
following theorem, we prove that the same conclusion holds in the
non-Archimedean setting as well, with mild restriction on the field $K$.

\begin{theorem}(Theorem \ref{th:trivial})\label{th:2}
Let $K\subset\bC_p$ be a finite extension of $\bQ_p$ and $\cX$ a proper
smooth scheme over $K$ of dimension $n$. Let $\cZ$ be an algebraic cycle of
$\cX$ of codimension $q$ such that $\cl_\dr(\cZ)=0$. Then
\[\int_{(\cZ\otimes_K\bC_p)^\an}\omega=0\]
for every $\rd''$-closed form
$\omega\in\sA^{n-q,n-q}((\cX\otimes_K\bC_p)^\an)$.
\end{theorem}

We emphasize again that in the above theorem, we do not know whether
$\cl_\sA(\cZ)=0$ or not. If we know the Poincar\'{e} duality for
$H^{\bullet,\bullet}_\sA(\cX^\an)$, then $\cl_\sA(\cZ)=0$. Nevertheless, we
have the following result for lower degree.

\begin{theorem}(Theorem \ref{th:line})
Let $\cX$ be a proper smooth scheme over $\bC_p$. Then
\begin{enumerate}
  \item $H^{1,1}_\sA(\cX^\an)$ is finite dimensional;

  \item for a line bundle $\cL$ on $\cX$ such that $\cl_\dr(\cL)=0$, we
      have $\cl_\sA(\cL)=0$.
\end{enumerate}
\end{theorem}

To the best of our knowledge, the first conclusion in the above theorem is
the only known case of the finiteness of $\dim H^{q,q'}_\sA(\cX^\an)$ when
both $q,q'$ are positive and $\cX$ is of general dimension. Note that in the
above theorem, we do not require that $\cX$ can be defined over a finite
extension of $\bQ_p$.

\begin{remark}
We can interpret Theorem \ref{th:2} in the following way. Let $k$ be a number
field. Let $\cX$ be a proper smooth scheme over $k$ of dimension $n$, and
$\cZ$ an algebraic cycle of $\cX$ of codimension $q$. Suppose that there
exists \emph{one} embedding $\iota_\infty\colon k\hookrightarrow\bC$ such
that
\[\int_{(\cZ\otimes_{k,\iota_\infty}\bC)^\an}\omega=0\]
for every $\overline\partial$-closed $(n-q,n-q)$-form $\omega$ on
$(\cX\otimes_{k,\iota_\infty}\bC)^\an$. Then for \emph{every} prime $p$ and
\emph{every} embedding $\iota_p\colon k\hookrightarrow\bC_p$, we have
\[\int_{(\cZ\otimes_{k,\iota_p}\bC_p)^\an}\omega=0\]
for every $\rd''$-closed $(n-q,n-q)$-form $\omega$ on
$(\cZ\otimes_{k,\iota_p}\bC_p)^\an$.\\
\end{remark}

The article is organized as follows. We review the basic theory of rigid
cohomology in Section \ref{ss:rigid}, which is one of the main tools in our
work. We construct the weight decomposition of de Rham cohomology sheaves in
the \'{e}tale topology in Section \ref{ss:weight}. In Section \ref{ss:log},
we study the behavior of logarithmic differential forms in rigid cohomology
and deduce Theorem \ref{th:1} for both topologies. We will not use \'{e}tale
topology after this point. We start Section \ref{ss:cycle} by reviewing the
theory of real forms developed by Chambert-Loir and Ducros; and then we study
its relation with de Rham cohomology sheaves. Next, we define the tropical
cycle class maps and establish their relation with integration of real forms.
In the last Section \ref{ss:triviality}, we study algebraic cycles that are
cohomologically trivial in the sense of algebraic de Rham cohomology. In
particular, we show that they are cohomologically trivial in the sense of
Dolbeault cohomology of currents (of forms if they are of codimension $1$).

\subsection*{Conventions and Notation}

\begin{itemize}
  \item Throughout the article, by a \emph{non-Archimedean field} we mean
      a complete topological field of characteristic \emph{zero} whose
      topology is induced by a nontrivial non-Archimedean valuation
      $|\;|$ of rank $1$. If the valuation is discrete, then we say that
      it is a \emph{discrete non-Archimedean field} by abuse of
      terminology.

  \item Let $K$ be a non-Archimedean field. Put
      \[K^\circ=\{x\in K\res
      |x|\leq 1\},\quad K^{\circ\circ}=\{x\in K\res |x|<1\},\quad
      \widetilde{K}=K^\circ/K^{\circ\circ}.\] Denote by $K^\ra$ the
      algebraic closure of $K$ and $\widehat{K^\ra}$ its completion. A
      \emph{residually algebraic} extension of $K$ is an extension $K'/K$
      of non-Archimedean fields such that the induced extension
      $\widetilde{K'}/\widetilde{K}$ is algebraic. In the text, discrete
      non-Archimedean fields are usually denoted by lower-case letters
      like $k,k'$, etc. And $\varpi$ will always be a uniformizer of a
      discrete non-Archimedean field, though we will still remind readers
      of this.

  \item Let $K$ be a non-Archimedean field, and $A$ an affinoid
      $K$-algebra. We then have the $K$-analytic space $\cM(A)$. Denote
      by $A^\circ$ the subring of power-bounded elements, which is a
      $K^\circ$-algebra. Put
      $\widetilde{A}=A^\circ\otimes_{K^\circ}\widetilde{K}$. We say that
      $A$ is \emph{integrally smooth} if $A$ is strictly $K$-affinoid and
      $\Spf A^\circ$ is a smooth formal $K^\circ$-scheme.

  \item Let $K$ be a non-Archimedean field. For a real number $r>0$, we
      denote by $D(0;r)$ the open disc over $K$ with center at zero of
      radius $r$. For real numbers $R>r>0$, we denote by $B(0;r,R)$ the
      open annulus over $K$ with center at zero of inner radius $r$ and
      outer radius $R$. An \emph{open poly-disc (of dimension $n$)} over
      $K$ is the product of finitely many open discs $D(0;r_i)$ (of
      number $n$).

  \item For a non-Archimedean field $K$, all $K$-analytic (Berkovich)
      spaces are assumed to be Hausdorff and strictly $K$-analytic
      \cite{Berk93}*{1.2.15}. Suppose that $K'/K$ is an extension of
      non-Archimedean fields. For a $K$-analytic space $X$ and a
      $K'$-analytic space $Y$, we put
      \[X\widehat\otimes_KK'=X\times_{\cM(K)}\cM(K'),\quad
      Y\times_KX=Y\times_{\cM(K')}(X\widehat\otimes_KK');\] and for a formal
      $K^\circ$-scheme $\fX$ and a formal $K'^\circ$-scheme $\fY$, we put
      \[\fX\widehat\otimes_{K^\circ}K'^\circ=\fX\times_{\Spf
      K^\circ}\Spf K'^\circ,\quad
      \fY\times_{K^\circ}\fX=\fY\times_{\Spf
      K'^\circ}(\fX\widehat\otimes_{K^\circ}K'^\circ).\]

  \item If $k$ is a discrete non-Archimedean field and $\fX$ is a special
      formal $k^\circ$-scheme in the sense of \cite{Berk96}, then we have
      the notion $\fX_\eta$, the generic fiber of $\fX$, which is a
      $k$-analytic space; and $\fX_s$, the closed fiber of $\fX$, which
      is a scheme locally of finite type over $\widetilde{k}$; and a
      reduction map $\pi\colon\fX_\eta\to\fX_s$. For a general
      non-Archimedean field $K$, we say a formal $K^\circ$-scheme $\fX$
      is special if there exist a discrete non-Archimedean field
      $k\subset K$ and a special formal $k^\circ$-scheme $\fX'$ such that
      $\fX\simeq\fX'\widehat\otimes_{k^\circ}K^\circ$. For a special
      formal $K^\circ$-scheme, we have similar notion
      $\pi\colon\fX_\eta\to\fX_s$ which is canonically defined. In this
      article, all formal $K^\circ$-schemes will be special. Note that if
      $\cZ$ is a subscheme of $\fX_s$, then $\pi^{-1}\cZ$ is usually
      denoted as $]\cZ[_{\fX_\eta}$ in rigid analytic geometry.

  \item If $\cX$ is a scheme over an affine scheme $\Spec A$ and $B$ is
      an $A$-algebra, then we put $\cX_B=\cX\times_{\Spec A}\Spec B$.
      Such abbreviation will be applied only to schemes, neither formal
      schemes nor analytic spaces. If $\cX$ is a scheme over $\Spec
      K^\circ$ for a non-Archimedean field $K$, then we write $\cX_s$ for
      $\cX_{\widetilde{K}}$.

  \item Let $K$ be a non-Archimedean field and $X$ a $K$-analytic space.
      For a point $x\in X$, one may associate nonnegative integers
      $s_K(x),t_K(x)$ as in \cite{Berk90}*{\Sec 9.1}. For readers'
      convenience, we recall the definition. The number $s_K(x)$ is equal
      to the transcendence degree of $\widetilde{\cH(x)}$ over
      $\widetilde{K}$, and the number $t_K(x)$ is equal to to the
      dimension of the $\bQ$-vector space
      $\sqrt{|\cH(x)^*|}/\sqrt{|K^*|}$, where $\cH(x)$ is the completed
      residue field of $x$. In the text, the field $K$ will always be
      clear so will be suppressed in the notation $s_K(x),t_K(x)$.

  \item Let $X$ be a site. Whenever we have a suitable notion of de Rham
      complex $(\Omega^\bullet_X,\rd)$ on $X$, we denote by
      $H^\bullet_\dr(X)\coloneqq H^\bullet(X,\Omega^\bullet_X)$ the
      corresponding de Rham cohomology of $X$, as the hypercohomology of
      the de Rham complex.
\end{itemize}

\subsection*{Acknowledgements}

The author is partially supported by NSF grant DMS--1602149. He thanks Weizhe~Zheng for helpful discussions.

\section{Review of rigid cohomology}
\label{ss:rigid}

In this section, we review the theory of rigid cohomology developed in, for
example, \cite{Bert97} and \cite{LS07}.

Let $\fR$ be the category of triples $(K,X,Z)$ where $K$ is a non-Archimedean
field; $X$ is a scheme of finite type over $\widetilde{K}$; and $Z$ is a
Zariski closed subset of $X$. A morphism from $(K',X',Z')$ to $(K,X,Z)$
consists of a field extension $K'/K$ and a morphism $X'\to
X\otimes_{\widetilde{K}}\widetilde{K'}$ whose restriction to $Z'$ factors
through $Z\otimes_{\widetilde{K}}\widetilde{K'}$. Let $\fV$ be the category
of pairs $(K,V^\bullet)$ where $K$ is a non-Archimedean field and $V^\bullet$
is a graded $K$-vector space. A morphism from $(K,V^\bullet)$ to
$(K',V'^\bullet)$ consists of a field extension $K'/K$ and a graded linear
map $V^\bullet\otimes_KK'\to V'^\bullet$.

We have a functor of \emph{rigid cohomology with support}:
$\fR^{\r{opp}}\to\fV$ sending $(K,X,Z)$ to $H^\bullet_{Z,\rig}(X/K)$. Put
$H^\bullet_\rig(X/K)=H^\bullet_{X,\rig}(X/K)$ for simplicity. We list the
following properties which will be used extensively in this article:
\begin{itemize}
  \item Suppose that we have a morphism $(K',X',Z')\to (K,X,Z)$ with
      $X'\simeq X\otimes_{\widetilde{K}}\widetilde{K'}$ and $Z'\simeq
      Z\otimes_{\widetilde{K}}\widetilde{K'}$. Then the induced map
      $H^\bullet_{Z,\rig}(X/K)\otimes_KK'\to H^\bullet_{Z',\rig}(X'/K')$
      is an isomorphism of finite dimensional graded $K'$-vector spaces
      (\cite{GK02}*{Corollary 3.8} and \cite{Berk07}*{Corollary 5.5.2}).

  \item For $Y=X\backslash Z$, we have a long exact sequence:
      \begin{align}\label{eq:support}
      \cdots\to H^i_{Z,\rig}(X/K) \to H^i_\rig(X/K) \to H^i_\rig(Y/K) \to H^{i+1}_{Z,\rig}(X/K)\to \cdots.
      \end{align}

  \item If both $X$, $Z$ are smooth, and $Z$ is of codimension $r$ in
      $X$, then we have a Gysin isomorphism $H^i_{Z,\rig}(X/K)\simeq
      H^{i-2r}_\rig(Z/K)$.

  \item Suppose that $K$ is residually algebraic over $\bQ_p$ (in other
      words, $\widetilde{K}$ is a finite extension of $\bF_p$). Then the
      sequence \eqref{eq:support} is equipped with a Frobenius action of
      sufficiently large degree. In particular, each item $V$ in
      \eqref{eq:support} admits a direct sum decomposition
      $V=\bigoplus_{w\in\bZ}V_w$ where $V_w$ consists of vectors of
      generalized weight $w$ (\cite{Chi98}*{\Sec 1 \& \Sec 2}).

  \item Suppose that $X$ is smooth and $Z$ is of codimension $r$, then
      $H^i_{Z,\rig}(X/K)_w=0$ unless $i\leq w\leq 2(i-r)$
      (\cite{Chi98}*{Theorem 2.3}).\\
\end{itemize}

We will extensively use the notion of $K$-analytic germs (\cite{Berk07}*{\Sec
5.1}), rather than $K$-dagger spaces. Roughly speaking, a $K$-analytic germ
is a pair $(X,S)$ where $X$ is a $K$-analytic space and $S\subset X$ is a
subset. We say that $(X,S)$ is a strictly $K$-affinoid germ if $S$ is a
strictly affinoid domain. We say that $(X,S)$ is smooth if $X$ is smooth in
an open neighborhood of $S$. We have the structure sheaf $\cO_{(X,S)}$, and
the de Rham complex $\Omega_{(X,S)}^\bullet$ when $(X,S)$ is smooth. (See
\cite{Berk07}*{\Sec 5.2} for details.) In particular, we have the de Rham
cohomology $H^\bullet_\dr(X,S)$ when $(X,S)$ is smooth. For a smooth
$K$-analytic germ $(X,S)$ where $S=\cM(A)$ for an integrally smooth
$K$-affinoid algebra $A$, we have a canonical functorial isomorphism
$H^\bullet_\dr(X,S)\simeq H^\bullet_\rig(\Spec\widetilde{A}/K)$ (see
\cite{Bert97}*{Proposition 1.10}, whose proof actually works for general
$K$).

The following lemma generalizes the construction in \cite{GK02}*{Lemma 2}.

\begin{lem}\label{le:dagger}
Let $(X_1, Y_1)$ and $(X_2,Y_2)$ be two smooth strictly $K$-affinoid germs.
Then for a morphism $\phi\colon Y_2\to Y_1$ of strictly $K$-affinoid domains,
there is a canonical restriction map $\phi^*\colon H^\bullet_\dr(X_1,Y_1)\to
H^\bullet_\dr(X_2,Y_2)$. It satisfies the following conditions:
\begin{enumerate}[(i)]
  \item if $\phi$ extends to a morphism $(X_2,Y_2)\to (X_1,Y_1)$ of
      germs, then $\phi^*$ coincides with the usual pullback;

  \item for a finite extension $K'$ of $K$, if we write $X'_i$ (resp.\
      $Y'_i$) for $X_i\widehat\otimes_KK'$ (resp.\
      $Y_i\widehat\otimes_KK'$) for $i=1,2$ and $\phi'$ for
      $\phi\widehat\otimes_KK'$, then $\phi'^*$ coincides with the scalar
      extension of $\phi^*$, in which we identify
      $H^\bullet_\dr(X'_i,Y'_i)$ with $H^\bullet_\dr(X_i,Y_i)\otimes_KK'$
      for $i=1,2$;

  \item if $Y_1=\cM(A_1)$ and $Y_2=\cM(A_2)$ for some integrally smooth
      $K$-affinoid algebras $A_1$ and $A_2$, then $\phi^*$ coincides with
      $\widetilde\phi^*\colon H^\bullet_\rig(\Spec \widetilde{A_1}/K)\to
      H^\bullet_\rig(\Spec \widetilde{A_2}/K)$ under the canonical
      isomorphism $H^\bullet_\dr(X_i,Y_i)\simeq
      H^\bullet_\rig(\Spec\widetilde{A_i}/K)$ for $i=1,2$, where
      $\widetilde\phi\colon\Spec\widetilde{A_2}\to\Spec\widetilde{A_1}$
      is the induced morphism;

  \item if $(X_3,Y_3)$ is another smooth strictly $K$-affinoid germ with
      a morphism $\psi\colon Y_3\to Y_2$, then
      $(\phi\circ\psi)^*=\psi^*\circ\phi^*$.
\end{enumerate}
\end{lem}

\begin{proof}
Put $X=X_1\times_K X_2$, $Y=Y_1\times_K Y_2$, and $\Delta\subseteq Y$ the
graph of $\phi$, which is isomorphic to $Y_2$ via the projection to the
second factor. Denote by $a_i\colon X\to X_i$ the projection morphism. We
have maps
\[H^\bullet_\dr(X_1,Y_1)\xrightarrow{a_1^*}\varinjlim_V H^\bullet_\dr(V)
\xleftarrow{a_2^*}H^\bullet_\dr(X_2,Y_2),\]
where $V$ runs through open neighborhoods of $\Delta$ in $X$. We show that
$a^*_2$ is an isomorphism. Then we define $\phi^*$ as $(a_2^*)^{-1}\circ
a_1^*$.

The proof is similar to that of \cite{GK02}*{Lemma 2}. To show that $a^*_2$
is an isomorphism is a local problem. Thus we may assume that there are
elements $t_1,\dots,t_m\in\sO_{X_1}(X_1)$ such that $\rd t_1,\dots,\rd t_m$
form a basis of $\Omega^1(X_1,Y_1)$ over $\sO(X_1,Y_1)$, and there exist a
strictly $K$-affinoid neighborhood $U_\epsilon\subset X$ of $\Delta$ with an
element $\epsilon\in|K^\times|$, and an isomorphism
\[ U_\epsilon\cap Y\xrightarrow{\sim} \cM(K\langle\epsilon^{-1}Z_1,\dots\epsilon^{-1}Z_m\rangle)\times_K\Delta,\]
in which $\epsilon^{-1}Z_i$ is sent to
$\epsilon^{-1}(t_i\otimes1-1\otimes\phi^*(t_i))$. Note that
$K\langle\epsilon^{-1}Z_1,\dots\epsilon^{-1}Z_m\rangle$ is an integrally
smooth $K$-affinoid algebra, and
$\Spec\widetilde{K\langle\epsilon^{-1}Z_1,\dots\epsilon^{-1}Z_m\rangle}$ is
canonically isomorphic to $\bA^m_{\widetilde{K}}$. Thus by \cite{GK02}*{Lemma
2}, the restriction map $H^\bullet_\dr(X_2,Y_2)\to
H^\bullet_\dr(X,U_\epsilon\cap Y)$ is an isomorphism. We may choose a
sequence of such $U_\epsilon$ with $\bigcap_\epsilon U_\epsilon=\Delta$. Then
$\varinjlim_\epsilon H^\bullet_\dr(X,U_\epsilon\cap Y)\simeq \varinjlim_V
H^\bullet_\dr(V)$ and thus $a_2^*$ is an isomorphism.

Properties (i) and (ii) follow easily from the construction. Property (iv) is
straightforward but tedious to check; we will leave it to readers. We now
check Property (iii), as it is important for our later argument. The induced
projection morphism
\[\cM(K\langle\epsilon^{-1}Z_1,\dots\epsilon^{-1}Z_m\rangle)\times_K\Delta\simeq
U_\epsilon\cap Y\to Y_i\] extends canonically to a morphism of formal
$K^\circ$-schemes
\[\Spf(K\langle\epsilon^{-1}Z_1,\dots\epsilon^{-1}Z_m\rangle\widehat\otimes_KA_\Delta)^\circ\to\Spf A_i^\circ,\]
where $A_\Delta$ is the coordinate $K$-affinoid algebra of $\Delta$ which is
isomorphic to $A_2$. Therefore, the restriction map
$H^\bullet_\dr(X_i,Y_i)\to H^\bullet_\dr(X,U_\epsilon\cap Y)$ coincides with
the map
\[\widetilde{a_i}^*\colon H^\bullet_\rig(\Spec \widetilde{A_i}/K)\to
H^\bullet_\rig(\Spec\widetilde{K\langle\epsilon^{-1}Z_1,\dots\epsilon^{-1}Z_m\rangle\widehat\otimes_KA_\Delta}/K)\]
induced from the homomorphism
$\widetilde{A_i}\to\widetilde{K\langle\epsilon^{-1}Z_1,\dots\epsilon^{-1}Z_m\rangle\widehat\otimes_KA_\Delta}$
of $\widetilde{K}$-algebras. Note that $\widetilde{a_2}^*$ is an isomorphism,
and that $(\widetilde{a_2}^*)^{-1}$ coincides with the restriction map
\[H^\bullet_\rig(\Spec\widetilde{K\langle\epsilon^{-1}Z_1,\dots\epsilon^{-1}Z_m\rangle\widehat\otimes_KA_\Delta}/K)
\to H^\bullet_\rig(\Spec \widetilde{A_2}/K)\] induced from the homomorphism
$\widetilde{K\langle\epsilon^{-1}Z_1,\dots\epsilon^{-1}Z_m\rangle\widehat\otimes_KA_\Delta}\to\widetilde{A_2}$
sending $\epsilon^{-1}Z_i$ to $0$ for all $i$. Property (iii) follows.
\end{proof}

The following example will be used in the computation later.

\begin{example}\label{ex:torus}
Let $K$ be a non-Archimedean field. For an integer $t\geq 0$ and an element
$\varpi\in K$, define the formal $K^\circ$-scheme
\[\fE^t_\varpi=\Spf K^\circ[[T_0,\dots,T_t]]/(T_0\cdots T_t-\varpi)\]
and let $\bE^t_\varpi$ be its generic fiber. Let $E^t_\varpi$ be the
$K$-affinoid algebra
\[K\langle |\varpi|^{-\frac{1}{t+1}}T_0,\cdots,|\varpi|^{-\frac{1}{t+1}}T_t,
|\varpi|^{\frac{1}{t+1}}T_0^{-1},\cdots,|\varpi|^{\frac{1}{t+1}}T_t^{-1}\rangle/(T_0\cdots
T_t-\pi),\] which is integrally smooth. Moreover, $\cM(E^t_\varpi)$ is
canonically a strictly $K$-affinoid domain in $\bE^t_\varpi$, and the
restriction map
\[H^\bullet_\dr(\bE^t_\varpi)\to H^\bullet_\dr(\bE^t_\varpi,\cM(E^t_\varpi))\simeq H^\bullet_\rig(\Spec\widetilde{E^t_\varpi}/K)\]
is an isomorphism by \cite{GK02}*{Lemma 3}. If $K$ is residually algebraic
over $\bQ_p$, then we have
$H^q_\rig(\Spec\widetilde{E^t_\varpi}/K)=H^q_\rig(\Spec\widetilde{E^t_\varpi}/K)_{2q}$.
\end{example}

\section{Weight decomposition in \'{e}tale topology}
\label{ss:weight}

In this section, we construct the weight decomposition of de Rham cohomology
sheaves in the \'{e}tale topology. Therefore, in this section, sheaves like
$\cO_X$, $\fc_X$, and the de Rham complex $(\Omega^\bullet_X,\rd)$ are
understood in the \'{e}tale topology.

\begin{definition}[Marked pair]\label{de:marked}
Let $k$ be a discrete non-Archimedean field.
\begin{enumerate}
  \item We say that a scheme $\cX$ over $k^\circ$ is \emph{strictly
      semi-stable of dimension $n$} if $\cX$ is locally of finite
      presentation, Zariski locally \'{e}tale over $\Spec
      K^\circ[T_0,\dots,T_n]/(T_0\cdots T_t-\varpi)$ for some uniformizer
      $\varpi$ of $k$, and $\cX_k$ is smooth over $k$. For every $0\leq
      t\leq n$, denote by $\cX_s^{[t]}$ the union of intersection of
      $t+1$ distinct irreducible components of $\cX_s$. It is a closed
      subscheme of $\cX_s$ with each irreducible component smooth.

  \item A \emph{marked $k$-pair $(\cX,\cD)$ of dimension $n$ and depth
      $t$} consists of an affine strictly semi-stable scheme $\cX$ over
      $k^\circ$ of dimension $n$, and an irreducible component $\cD$ of
      $\cX_s^{[t]}$ that is geometrically irreducible.
\end{enumerate}
\end{definition}

We start from the following lemma, which generalizes \cite{Berk07}*{Lemma
2.1.2}.

\begin{lem}\label{le:212}
Suppose that $K$ is embeddable into $\widehat{k^\ra}$ for some discrete
non-Archimedean field $k$. Let $X$ be a smooth $K$-analytic space, and $x$ a
point of $X$ with $s(x)+t(x)=\dim_x(X)$. Given a morphism of strictly
$K$-analytic spaces $X\to\fY_\eta$, where $\fY$ is a special formal
$K^\circ$-scheme, there exist
\begin{itemize}
  \item a finite extension $K'$ of $K$, a finite extension $k'$ of $k$
      contained in $K'$,

  \item a marked $k'$-pair $(\cX,\cD)$ of dimension $\dim_x(X)$ and depth
      $t(x)$,

  \item an open neighborhood $U$ of
      $(\widehat\cX_{/\cD})_\eta\widehat\otimes_{k'}K'$ in
      $\cX_{K'}^\an$,

  \item a point $x'\in(\widehat\cX_{/\cD})_\eta\widehat\otimes_{k'}K'$,

  \item a morphism of $K$-analytic spaces $\varphi\colon U\to X$, and

  \item a morphism of formal $K^\circ$-schemes
      $\widehat\cX_{/\cD}\widehat\otimes_{k'^\circ}K'^\circ\to\fY$,
\end{itemize}
such that the following are true:
\begin{enumerate}[(i)]
  \item $\varphi$ is \'{e}tale and $\varphi(x')=x$;

  \item the induced morphism
      $(\widehat\cX_{/\cD})_\eta\widehat\otimes_{k'}K'\to\fY_\eta$
      coincides with the composition
      \[(\widehat\cX_{/\cD})_\eta\widehat\otimes_{k'}K'\hookrightarrow
      U\xrightarrow{\varphi}X\to\fY_\eta.\]
\end{enumerate}
\end{lem}

\begin{proof}
Put $t=t(x)$, $s=s(x)$, and $n=t+s$. By \cite{Berk07}*{Proposition 2.3.1}, by
possibly taking finite extensions of $k$ (and $K$), we may replace $X$ by
$(B\times_k Y)\widehat\otimes_kK$, where $B=\prod_{j=1}^t B(0;r_j,R_j)$ for
some $0<r_j<R_j$ and $Y$ is a smooth $k$-analytic space of dimension $s$, and
$x$ projects to $b\in B$ with $t(b)=t$ and $y\in Y$ with $s(y)=s$. Denote by
$\cP$ the $k^\circ$-scheme $\bP^1_{k^\circ}$ with the point $0$ on the
special fiber blown up, and by $\fP$ the formal completion of $\cP$ along the
open subscheme $\cP_s\backslash\{\pi(0),\pi(\infty)\}$, which is isomorphic
to $\Spf k^\circ\langle X,Y\rangle/(XY-\varpi)$ for some uniformizer $\varpi$
of $k$. By taking further finite extensions of $k$ (and $K$), we may assume
that there is an open immersion $\prod_t\fP_\eta\subset B$ containing $b$
such that $\pi(b)=\b{0}$, where $\b{0}$ is the closed point in $\prod_t\cP_s$
that is nodal in every component.

For $Y$, we proceed exactly as in the Step 1 of the proof of
\cite{Berk07}*{Lemma 2.1.2}. We obtain two strictly $k$-affinoid domains
$Z'\subset W'\subset Y$. As in the beginning of Step 3 of the proof of
\cite{Berk07}*{Lemma 2.1.2}, we also get an integral scheme $\cY'$ proper and
flat over $k^\circ$ with an embedding $Y\subset\cY^{\prime\an}_\eta$, open
subschemes $\cZ'\subset\cW'\subset\cY'_s$, such that $Z'=\pi^{-1}\cZ'$,
$W'=\pi^{-1}\cW'$.

Now we put two parts together. Define $\cY=\prod_t\cP\times\cY'$ where the
fiber product is taken over $k^\circ$, and $\cW=\prod_t\fP_s\times\cW'$ where
the fiber product is taken over $\widetilde{k}$. Then
$W\coloneqq\prod_t\fP_\eta\times W'$ coincides with $\pi^{-1}\cW$ in
$\cY_k^\an$. Moreover, $W_K$ is an open neighborhood of $x$ where $W_K$
denotes the inverse image of $W$ in $X=(B\times_k Y)\widehat\otimes_kK$.
Write $W=\bigcup_{i=1}^l W_i$ as in Step 2 of the proof of
\cite{Berk07}*{Lemma 2.1.2}. By \cite{Berk07}*{Lemma 2.1.3 (ii)}, we may
assume that $W_i$ are all $k$-affinoid by taking finite extensions of $k$
(and $K$). Making a finite number of additional blow-ups, we may also assume
that there are open subschemes $\cW_i\subset\cW$ with $W_i=\pi^{-1}\cW_i$ and
$\cW=\bigcup_{i=1}^l\cW_i$.

Now we proceed as in Step 4 of the proof of \cite{Berk07}*{Lemma 2.1.2}. Take
an alteration $\phi\colon\cX'\to\cY$ after further finite extensions of $k$
(and $K$), and a point $x'\in\cX'^\an_K$ such that $\phi(x')=x$. By a similar
argument, one can show that
$\pi(x')\in\cX'_s\otimes_{\widetilde{k}}\widetilde{K}$ has dimension at least
$s$. On the other hand, we have $s(x')\geq s$ and $t(x')\geq t$. Thus,
$s(x')=s$ and $t(x')=t$. Denote by $\cC$ the Zariski closure of $\pi(x')$ in
$\cX'_s$, equipped with the reduced induced scheme structure. Suppose that
$\cC$ is contained in $t'$ distinct irreducible components of $\cX'_s$. Then
$t'\leq t$. We take an open subscheme $\cU'$ of $\cX'$ satisfying:
$\cD'\coloneqq\cU'\cap\cC$ is open dense in $\cC$; $\phi(\cD')$ is contained
in $\cW$; $\cU'$ is \'{e}tale over $\Spec k^\circ[T_0,\dots,T_n]/(T_0\cdots
T_{t'}-\pi)$ for some uniformizer $\varpi$ of $k$ such that $\cD'$ is the
zero locus of the ideal generated by $(T_0,\dots,T_t,\pi)$. Now we blow up
the closed ideal generated by $(T_{t'+1},\varpi)$, and then the strict
transform of the closed ideal generated by $(T_{t'+2},\varpi)$, and continue
to obtain an affine strictly semi-stable scheme $\cX$ over $k^\circ$ such
that the strict transform $\cD$ of $\cD'$ is an irreducible component of
$\cX^{[t]}_s$. After further finite extensions of $k$ (and $K$) and replacing
$\cX$ by an affine open subscheme such that $\cX_s\cap\cD$ is dense in $\cD$,
we obtain a marked $k$-pair $(\cX,\cD)$ of dimension $n$ and depth $t$ such
that $\phi\colon\cX\to\cY$ is \'{e}tale on the generic fiber. Note that
$(\phi_K^\an)^{-1}W_K$ is a neighborhood of $x'$ containing $\pi^{-1}\cD$ as
$\phi(\cD)\subseteq\cW$. Here, $x'\in\cX^\an_K$ is an arbitrary preimage of
the original $x'\in\cX'^\an_K$, which exists by construction.

We take $U$ to be an arbitrary open neighborhood of $\pi^{-1}\cD$ contained
in $(\phi_K^\an)^{-1}W_K$, and $\varphi$ to be $\phi_K^\an\res_U$. By the
same argument in Step 5 of the proof of \cite{Berk07}*{Lemma 2.1.2}, $\phi$
induces a morphism of $K^\circ$-formal schemes
$\widehat\cX_{/\phi^{-1}\cW}\widehat\otimes_{k^\circ}K^\circ\to\fY$ and thus
a morphism $\widehat\cX_{/\cD}\widehat\otimes_{k^\circ}K^\circ\to\fY$. The
conclusions of the lemma are all satisfied by the construction.
\end{proof}

From now on, we assume that $K$ is a residually algebraic extension of
$\bQ_p$.

\begin{definition}[Fundamental chart]\label{de:chart}
Let $X$ be a $K$-analytic space and $x\in X$ a point. A \emph{fundamental
chart} of $(X;x)$ consists of data $(\bD,(\cY,\cD),(D,\delta),W,\alpha;y)$
where
\begin{itemize}
        \item $(\cY,\cD)$ is a marked $k$-pair of dimension $t(x)+s(x)$
            and depth $t(x)$, where $k$ is a finite extension of $\bQ_p$,

        \item $\bD$ is an open poly-disc over $L$ of dimension
            $\dim_x(X)-t(x)-s(x)$, where $L$ is simultaneously a finite
            extension of $K$ and a (residually algebraic) extension of
            $k$,

        \item $D$ is an integrally smooth affinoid $k$-algebra, and
            \begin{align}\label{eq:splitting}
            \delta\colon\Spf D^\circ[[T_0,\dots,T_t]]/(T_0\cdots
            T_t-\varpi)\xrightarrow{\sim}\widehat\cX_{/\cD}
            \end{align}
            is an isomorphism of formal $k^\circ$-schemes, where $\varpi$
            is a uniformizer of $k$,

        \item $W$ is an open neighborhood of
            $(\cY_{/\cD})_\eta\widehat\otimes_kL=\pi^{-1}\cD_{\widetilde{L}}$
            in $\cY_L^\an$,

        \item $y$ is a point in $\bD\times_L W$ such that it projects to
            $0$ in $\bD$ and a point in $W$ whose reduction is the
            generic point of $\cD_{\widetilde{L}}$,

        \item $\alpha\colon\bD\times_L W\to X$ is an \'{e}tale morphism
            of $K$-analytic spaces.
\end{itemize}
Note that the fields $k$ and $L$ will be implicit from the notation (as they
are not important).
\end{definition}

The isomorphism \eqref{eq:splitting} induces an isomorphism
$\Spec\widetilde{D}\simeq\cD$ of $\widetilde{k}$-schemes, and an isomorphism
\begin{align}\label{eq:decomposition}
\delta^*\colon H^q_\dr(\bD\times_L(W,\pi^{-1}\cD_{\widetilde{L}}))\xrightarrow{\sim}
\bigoplus_{j=0}^q H^j_\rig(\cD/k)\otimes_k H^{q-j}_\dr(\bE^t_\varpi)\otimes_kL
\end{align}
of $L$-vector spaces \cite{GK02}*{Lemmas 2 \& 3} and \cite{Berk07}*{Corollary
5.5.2}. Here, $\bE^t_\varpi$ is the $k$-analytic space defined in Example
\ref{ex:torus}. Denote by $H^q_w(\bD,(\cY,\cD),(D,\delta),W)$ the subspace of
the left-hand side of \eqref{eq:decomposition} corresponding to the subspace
\[\bigoplus_{j=0}^q H^j_\rig(\cD/k)_{w-2(q-j)}\otimes_k H^{q-j}_\dr(\bE^t_\varpi)\otimes_kL\]
on the right-hand side. In particular, all elements in
$H^{q-j}_\dr(\bE^t_\varpi)$ are regarded to be of weight $2(q-j)$. Then we
have a direct sum decomposition
\begin{align}\label{eq:decomp}
H^q_\dr(\bD\times_L(W,\pi^{-1}\cD_{\widetilde{L}}))=\bigoplus_{w\in\bZ}H^q_w(\bD,(\cY,\cD),(D,\delta),W).
\end{align}
Finally, we denote by $H^q_{(w)}(\bD,(\cY,\cD),(D,\delta),W)$ the subspace of
$H^q_\dr(\bD\times_LW)$ as the inverse image of
$H^q_w(\bD,(\cY,\cD),(D,\delta),W)$ under the restriction map
\[H^q_\dr(\bD\times_LW)\to
H^q_\dr(\bD\times_L(W,\pi^{-1}\cD_{\widetilde{L}})).\] In what follows, if $\bD$ is of
dimension $0$, then we suppress it from all notations.

\begin{remark}\label{re:range}
Note that $H^q_w(\bD,(\cY,\cD),(D,\delta),W)=0$ unless $q\leq w\leq 2q$, and
the decomposition \eqref{eq:decomp} is stable under base change along a
residually algebraic extension of $K$ (and $L$ accordingly). We warn that the
decomposition \eqref{eq:decomp} depends on all of the data
$(\bD,(\cY,\cD),(D,\delta),W)$, not just the $L$-analytic germ
$\bD\times_L(W,\pi^{-1}\cD_{\widetilde{L}})$. (However, the dependence on
$\bD$ and $W$ is very weak.)
\end{remark}

\begin{definition}
Let $X$ be a $K$-analytic space and $x\in X$ a point.
\begin{enumerate}
  \item Let $\fEt(X;x)$ be the category whose objects are fundamental
      charts of $(X;x)$, and a morphism
      \[\phi\colon(\bD_2,(\cY_2,\cD_2),(D_2,\delta_2),W_2,\alpha_2;y_2)\to
      (\bD_1,(\cY_1,\cD_1),(D_1,\delta_1),W_1,\alpha_1;y_1)\] consists
      implicitly extensions of related fields $K\subset L_1\subset L_2$
      such that $k_1\subset k_2$, and a morphism $\Phi(\phi)\colon
      \bD_2\times_{L_2}W_2\to\bD_1\times_{L_1}W_1$ of $L_1$-analytic
      spaces sending $y_2$ to $y_1$, and such that
      \[\Phi(\phi)^*H^q_{(w)}(\bD_1,(\cY_1,\cD_1),(D_1,\delta_1),W_1)\subset
      H^q_{(w)}(\bD_2,(\cY_2,\cD_2),(D_2,\delta_2),W_2)\] for all $q,w\in
      \bZ$. Note that $\Phi(\phi)$ needs \emph{not} to respect each
      factors.

  \item Let $\Et(X;x)$ be the category of \'{e}tale neighborhoods of
      $(X;x)$. Recall that its objects are triples $(Y,\alpha;y)$ where
      $\alpha\colon Y\to X$ is an \'{e}tale morphism sending $y\in Y$ to
      $x$, and morphisms are defined in the obvious way. In the notation
      $(Y,\alpha;y)$, the morphism $\alpha$ will be suppressed if it is
      not relevant. For a presheaf $\sF$ on $X_{\et}$, the stalk of $\sF$
      at $x$ is defined to be $\sF_x\coloneqq \varinjlim_{(Y,\alpha;y)}
      \sF(Y)$ where the colimit is taken over the category $\Et(X;x)$.

  \item We have a functor $\Phi\colon\fEt(X;x)\to\Et(X;x)$ sending an
      object $(\bD,(\cY,\cD),(D,\delta),W,\alpha;y)$ of $\fEt(X;x)$ to
      $(\bD\times_LW,\alpha;y)$, and a morphism $\phi$ to $\Phi(\phi)$.
\end{enumerate}
\end{definition}

The following lemma generalizes \cite{Berk07}*{Proposition 2.1.1}.

\begin{lem}\label{le:initial}
Suppose that $K$ is embeddable into $\bC_p$ and $X$ is a smooth $K$-analytic
space. Fix an arbitrary point $x\in X$ and let $(Y,\alpha_0;y_0)$ be an
object of $\Et(X;x)$. Then
\begin{enumerate}
  \item there exists an object
      $(\bD,(\cY,\cD),(D,\delta),W,\alpha;y)\in\fEt(X;x)$ such that its
      image under $\Phi$ admits a morphism to $(Y,\alpha_0;y_0)$;

  \item given two morphisms
      $\beta_i\colon\Phi(\bD_i,(\cY_i,\cD_i),(D_i,\delta_i),W_i,\alpha_i;y_i)
      \to(Y,\alpha_0;y_0)$ in $\Et(X;x)$ for $i=1,2$, there exists an
      object $(\bD,(\cY,\cD),(D,\delta),W,\alpha;y)\in\fEt(X;x)$ together
      with morphisms $\phi_i$ to
      $(\bD_i,(\cY_i,\cD_i),(D_i,\delta_i),W_i,\alpha_i;y_i)$ in
      $\fEt(X;x)$ for $i=1,2$ such that the following diagram
      \[\xymatrix{
      &  \Phi(\bD_1,(\cY_1,\cD_1),(D_1,\delta_1),W_1,\alpha_1;y_1) \ar[rd]^-{\beta_1} \\
      \Phi(\bD,(\cY,\cD),(D,\delta),W,\alpha;y) \ar[ru]^-{\Phi(\phi_1)}\ar[rd]_-{\Phi(\phi_2)} && (Y,\alpha_0;y_0) \\
      &  \Phi(\bD_2,(\cY_2,\cD_2),(D_2,\delta_2),W_2,\alpha_2;y_2) \ar[ru]_-{\beta_2}
      }\] commutes.
\end{enumerate}
In particular, the functor $\Phi\colon\fEt(X;x)\to\Et(X;x)$ is initial.
\end{lem}

\begin{proof}
We may assume that $X$ is of dimension $n$. Put $t=t(x)$ and $s=s(x)$.

For (1), by \cite{Berk07}*{Proposition 2.3.1}, after taking a finite
extension of $K$, we may assume that $Y=\bD\times_K X'$ and $y_0=(0,x')$
(which makes sense) for a point $x'\in X'$ with $t(x')=t$ and $s(x')=s$,
where $X'$ is a smooth $K$-analytic space of dimension $s+t$. Now we only
need to apply Lemma \ref{le:212} to $\fY=\Spf K^\circ$, the pair $(X';x')$,
and the structure morphism $X'\to\fY_\eta=\cM(K)$. The existence of
$(D,\delta)$ is due to the argument in Part (iv) of the proof of
\cite{GK02}*{Theorem 2.3}.

For (2), we may assume that $K=L_1=L_2$. For $i=1,2$, we choose a relative
compactification $\cY_i\hookrightarrow\overline{\cY_i}$ over $k_i^\circ$,
where $\overline{\cY_i}$ is proper. Then $W_i$ is open in $\fY_{i\eta}$,
where $\fY_i=\widehat{\overline{\cY_i}}\widehat\otimes_{k_i^\circ}K^\circ$.
Consider the \'{e}tale morphism
\[\alpha'_0\colon Y'\coloneqq(\bD_1\times_KW_1)\times_Y(\bD_2\times_KW_2)
\to Y,\] and a point $y'_0\in Y'$ projecting to $y_1$ (resp.\ $y_2$) in the
first (resp.\ second) factor. Again by \cite{Berk07}*{Proposition 2.3.1}, we
may find an object of the form $(\bD\times_K X',\alpha';(0,x'))$ in
$\Et(X;x)$ as in (1) with a morphism to $(Y',\alpha'_0;y'_0)$. Now we apply
Lemma \ref{le:212} to $X'$, the point $x'$, $\fY=\fY_1\times_{K^\circ}\fY_2$,
the morphism
\[X'\xrightarrow{(\beta_1,\beta_2)} W_1\times_K W_2 \subset \fY_\eta,\]
where $\beta_i$ equals the composition
\[X'\simeq\{0\}\times_KX'\subset\bD\times_K X'\to\bD_i\times_K
W_i\to W_i \quad(i=1,2)\] with the last arrow being the projection. We obtain
a marked $k$-pair $(\cY,\cD)$ of dimension $s+t$ and depth $t$, for some
discrete non-Archimedean field $k$ containing $k_1,k_2$ and contained in
(possibly a finite extension of) $K$; an open neighborhood $W$ of
$(\widehat\cY_{/\cD})_\eta\widehat\otimes_kK$ in $\cY_K^\an$, a point $y'\in
W$, a morphism of $K$-analytic spaces $\varphi\colon W\to X'$ such that
$\varphi(y')=x'$, and a morphism of formal $K^\circ$-schemes
$\psi=(\psi_1,\psi_2)\colon\widehat\cY_{/\cD}\widehat\otimes_{k^\circ}K^\circ\to\fY_1\times_{K^\circ}\fY_2$
compatible with $\varphi$. As $\psi_i$ maps the generic point of
$\cD_{\widetilde{K}}$ to the generic point of $(\cD_i)_{\widetilde{K}}$, we
may replace $(\cY,\cD)$ by an affine open such that
$\psi_i(\cD_{\widetilde{K}})\subset(\cD_i)_{\widetilde{K}}$ for $i=1,2$. In
particular, we have morphisms
$\psi_i\colon\widehat\cY_{/\cD}\widehat\otimes_{k^\circ}K^\circ\to\widehat{\cY_i}_{/\cD_i}\widehat\otimes_{k_i^\circ}K^\circ$.
Note that $\psi_i$ does not necessarily descent to any finite extension of
$k$. By the proof of \cite{GK02}*{Theorem 2.3}, there is an integrally smooth
$k$-affinoid algebra $D$ and an isomorphism $\delta$ as in
\eqref{eq:splitting}.

Now the object $(\bD,(\cY,\cD),(D,\delta),W,\alpha;y)$ has been constructed
with $y=(0,y')$ and the obvious $\alpha$. Let $\Phi(\phi_i)$ be the composite
morphism $\bD\times_KW\to\bD\times_KX'\to\bD_i\times_K W_i$ for $i=1,2$. It
remains to show that
\begin{enumerate}[(i)]
  \item For $i=1,2$, every $q$, every $w$, and an element $\omega\in
      H^q_\rig(\cD_i/k_i)_w$, we have
      \[(\beta_i\circ\varphi)^*(\delta_i^*)^{-1}\omega\in H^q_{(w)}((\cY,\cD),(D,\delta),W).\]

  \item For $i=1,2$ and an arbitrary coordinate $T$ of $\bE^t_{\varpi_i}$
      (where $\varpi_i$ is a uniformizer of $k_i$), we have
      \[(\beta_i\circ\varphi)^*(\delta_i^*)^{-1}\frac{\rd T}{T}\in H^1_{(2)}((\cY,\cD),(D,\delta),W).\]
\end{enumerate}
Note that the composite morphism of formal $k^\circ$-schemes
\[\Spf((E^t_\varpi)^\circ\widehat\otimes_{k^\circ}D^\circ)\to\Spf D^\circ[[T_1,\dots,T_t]]/(T_1\cdots
T_t-\varpi)\xrightarrow{\delta}\widehat\cY_{/\cD}\] induces an isomorphism
\[H^q_\dr(W,\pi^{-1}\cD_{\widetilde{K}})\simeq H^q_\rig(\Spec\widetilde{E^t_\varpi}\times_{\widetilde{k}}\cD_{\widetilde{K}}/K)\]
under which
\[H^q_w((\cY,\cD),(D,\delta),W)=H^q_\rig(\Spec\widetilde{E^t_\varpi}\times_{\widetilde{k}}\cD_{\widetilde{K}}/K)_w\]
for every $q$ and every $w$.

For (i), as we have morphisms of formal $K^\circ$-schemes
\[\Spf((E^t_\varpi)^\circ\widehat\otimes_{k^\circ}D^\circ\widehat\otimes_{k^\circ}K^\circ)
\to\widehat{\cY}_{/\cD}\widehat\otimes_{k^\circ}K^\circ\xrightarrow{\psi_i}
\widehat{\cY_i}_{/\cD_i}\widehat\otimes_{k_i^\circ}K^\circ\to\Spf
D^\circ_i\widehat\otimes_{k_i^\circ}K^\circ,\] Lemma \ref{le:dagger} implies
that $(\beta_i\circ\varphi)^*(\delta_i^*)^{-1}\omega$ coincides with
$\varphi_i^*\omega$ in $H^q_\dr(W,\pi^{-1}\cD_{\widetilde{K}})$, where
\[\varphi_i\colon\Spec\widetilde{E^t_\varpi}\times_{\widetilde{k}}\cD_{\widetilde{K}}
\to(\cD_i)_{\widetilde{K}}\]
is the induced morphism of (affine smooth) $\widetilde{K}$-schemes.

For (ii), we may assume that $\Spf D^\circ_i$ has a $K^\circ$-point by
replacing $K$ by a finite extension (at the very beginning). Thus we have
morphisms of formal $K^\circ$-schemes
\[\Spf((E^t_\varpi)^\circ\widehat\otimes_{k^\circ}D^\circ\widehat\otimes_{k^\circ}K^\circ)
\to\widehat{\cY}_{/\cD}\widehat\otimes_{k^\circ}K^\circ\xrightarrow{\psi_i}
\widehat{\cY_i}_{/\cD_i}\widehat\otimes_{k_i^\circ}K^\circ\to\fE^t_{\varpi_i}\widehat\otimes_{k_i^\circ}K^\circ
\xrightarrow{T}\Spf K^\circ[[T]].\] On the generic fiber, the image of the
induced morphism $\cM(E^t_\varpi\widehat\otimes_kD\widehat\otimes_kK)\to
D(0;1)$ does not contain $0$, which implies that it factors through a
morphism $\cM(E^t_\varpi\widehat\otimes_kD\widehat\otimes_kK)\to\cM(K\langle
r^{-1}T,rT^{-1}\rangle)$ for a unique $r<1$ in $\sqrt{|K^\times|}$ as
$(E^t_\varpi)^\circ\widehat\otimes_{k^\circ}D^\circ$ is smooth over
$k^\circ$. By taking a finite extension of $K$, we may assume that
$r\in|K^{\circ\circ}|$. Then $K\langle r^{-1}T,rT^{-1}\rangle$ is integrally
smooth, and we have $\Spec\widetilde{K\langle
r^{-1}T,rT^{-1}\rangle}\simeq(\bG_\r{m})_{\widetilde{K}}$. Moreover,
\[H^1_\rig(\bG_\r{m}/K)_2=H^1_\rig(\bG_\r{m}/K)\simeq H^1_\dr(D(0,1),\cM(K\langle
r^{-1}T,rT^{-1}\rangle))=K\{\frac{\rd T}{T}\}.\] Thus Lemma \ref{le:dagger} implies (ii).
\end{proof}

\begin{remark}\label{re:decomp}
The above lemma with its proof implies the following: For part of the data
$(\bD,(\cY,\cD),(D,\delta),W)$ from Definition \ref{de:chart} and
$f\in\cO^*(\bD\times_LW)$, we have
\[\frac{\rd f}{f}\in H^1_{(2)}(\bD,(\cY,\cD),(D,\delta),W).\]
Here, we regard $\frac{\rd f}{f}$, a priori a closed $1$-form on
$\bD\times_LW$, as an element in $H^1_\dr(\bD\times_LW)$.
\end{remark}

Now we are ready to define the desired direct summand
$(\Omega^{q,\cl}_X/\rd\Omega^{q-1}_X)_w$ in the weight decomposition of de
Rham cohomology sheaves.

\begin{definition}[De Rham cohomology sheaves with weights]\label{de:weight}
Suppose that $K$ is residually algebraic over $\bQ_p$ and $X$ is a smooth
$K$-analytic space.

For every object $U$ of $X_{\et}$, define
$(\Omega^{q,\cl}_X/\rd\Omega^{q-1}_X)(U)^\r{pre}_w\subset(\Omega^{q,\cl}_X/\rd\Omega^{q-1}_X)(U)$
to be the image of elements $\omega\in\Omega^{q,\cl}_X(U)$ such that for
every point $u\in U$, there exists a fundamental chart
$(\bD,(\cY,\cD),(D,\delta),W,\alpha;y)$ of $(U;u)$ such that
$\alpha^*\omega$, regarded as an element in $H^q_\dr(\bD\times_L W)$, belongs
to $H^q_{(w)}(\bD,(\cY,\cD),(D,\delta),W)$. The assignment
$U\mapsto(\Omega^{q,\cl}_X/\rd\Omega^{q-1}_X)(U)^\r{pre}_w$ defines a
sub-presheaf $(\Omega^{q,\cl}_X/\rd\Omega^{q-1}_X)^\r{pre}_w$ of
$\Omega^{q,\cl}_X/\rd\Omega^{q-1}_X$.

We define $(\Omega^{q,\cl}_X/\rd\Omega^{q-1}_X)_w$ to be the sheafification
of $(\Omega^{q,\cl}_X/\rd\Omega^{q-1}_X)^\r{pre}_w$, which is canonically a
subsheaf of $\Omega^{q,\cl}_X/\rd\Omega^{q-1}_X$.
\end{definition}

The following lemma can be proved by the same way as for
\cite{Berk07}*{Corollary 5.5.3}.

\begin{lem}\label{le:base_change}
Let $K'/K$ be an extension such that $K'$ is embeddable into $\bC_p$. Let $X$
be a smooth $K$-analytic space and $\varsigma\colon X'\coloneqq
X\widehat\otimes_KK'\to X$ the canonical projection. Then the canonical map
of sheaves on $X'_{\et}$
\[\varsigma^{-1}(\Omega^{q,\cl}_X/\rd\Omega^{q-1}_X)\otimes_LK'\to\Omega^{q,\cl}_{X'}/\rd\Omega^{q-1}_{X'}\]
is an isomorphism, where $L$ is the algebraic closure of $K$ in $K'$.
\end{lem}

The following theorem establishes the functorial weight decomposition of de Rham cohomology sheaves in Theorem \ref{th:1} in the case
of \'{e}tale topology.

\begin{theorem}\label{th:weight}
If $K$ is embeddable into $\bC_p$ and $X$ is a smooth $K$-analytic space,
then we have that
\begin{enumerate}
  \item under the situation of Lemma \ref{le:base_change},
      \[\varsigma^{-1}(\Omega^{q,\cl}_X/\rd\Omega^{q-1}_X)_w\otimes_LK'=(\Omega^{q,\cl}_{X'}/\rd\Omega^{q-1}_{X'})_w,\]
      for every $w\in\bZ$;

  \item the image of the composite map
      \[(\Omega^{q_1,\cl}_X/\rd\Omega^{q_1-1}_X)_{w_1}\otimes(\Omega^{q_2,\cl}_X/\rd\Omega^{q_2-1}_X)_{w_2}
      \to\Omega^{q_1,\cl}_X/\rd\Omega^{q_1-1}_X\otimes\Omega^{q_2,\cl}_X/\rd\Omega^{q_2-1}_X
      \xrightarrow{\wedge}\Omega^{q_1+q_2,\cl}_X/\rd\Omega^{q_1+q_2-1}_X\]
      is contained in the subsheaf
      $(\Omega^{q_1+q_2,\cl}_X/\rd\Omega^{q_1+q_2-1}_X)_{w_1+w_2}$;

  \item the sheaf $(\Omega^{q,\cl}_X/\rd\Omega^{q-1}_X)_w$ is zero unless
      $q\leq w\leq 2q$;

  \item the canonical map
      \[\bigoplus_{w\in\bZ}(\Omega^{q,\cl}_X/\rd\Omega^{q-1}_X)_w\to\Omega^{q,\cl}_X/\rd\Omega^{q-1}_X\]
      is an isomorphism;

  \item for every morphism $f\colon Y\to X$ of smooth $K$-analytic
      spaces, we have
      \[f^\#(f^{-1}(\Omega^{q,\cl}_X/\rd\Omega^{q-1}_X)_w)\subset(\Omega^{q,\cl}_Y/\rd\Omega^{q-1}_Y)_w\]
      for every $w\in\bZ$. Here, $f^\#$ denotes the canonical map
      $f^{-1}\Omega_X^\bullet\to\Omega_Y^\bullet$ and induced maps of
      cohomology sheaves.
\end{enumerate}
\end{theorem}

\begin{proof}
Part (1) follows from the definition and Remark \ref{re:range}. Part (2)
follows from definition and Lemma \ref{le:initial} (2).

For the remaining parts, it suffices to work on stalks. Thus we fix a point
$x\in X$ with $t=t(x)$ and $s=s(x)$.

For (3), take an element $[\omega]$ in the stalk of
$(\Omega^{q,\cl}_X/\rd\Omega^{q-1}_X)_w$ at $x$ for some $w<q$ or $w>2q$. We
may assume that it has a representative $\omega\in\Omega^{q,\cl}_X(U)$ for
some \'{e}tale neighborhoods $(U;u)$ of $(X;x)$. By definition, we have a
fundamental chart $(\bD,(\cY,\cD),(D,\delta),W,\alpha;y)$ of $(U;u)$ such
that $\alpha^*\omega=0$ in
$H^q_\dr(\bD\times_L(W,\pi^{-1}\cD_{\widetilde{L}}))$ by Remark
\ref{re:range}. Then there exists an open neighborhood $W^-$ of
$\pi^{-1}\cD_{\widetilde{L}}$ in $W$, such that $\alpha^*\omega=0$ in
$H^q_\dr(\bD\times_LW)$. In other words, $[\omega]=0$ in the stalk of
$(\Omega^{q,\cl}_X/\rd\Omega^{q-1}_X)_w$ at $x$.

For (4), we first show that the map is injective. Let $[\omega]$ be an
element in the stalk $\Omega^{q,\cl}_{X,x}/\rd\Omega^{q-1}_{X,x}$. Suppose
that we have $[\omega]=\sum[\omega]^1_w=\sum[\omega]^2_w$ in which both
$[\omega]^1_w$ and $[\omega]^2_w$ are in the stalk of
$(\Omega^{q,\cl}_X/\rd\Omega^{q-1}_X)_w$ at $x$. We may choose an object
$(U;u)\in\Et(X,x)$ such that $[\omega]^i_w$ has a representative
$\omega^i_w\in(\Omega^{q,\cl}_X/\rd\Omega^{q-1}_X)(U)^\r{pre}_w$ for $i=1,2$
and every $w\in\bZ$, and $\sum\omega^1_w=\sum\omega^2_w$. In particular,
$[\omega]$ has a representative
$\omega\coloneqq\sum\omega^1_w=\sum\omega^2_w$ on $(U;u)$. Fix a weight
$w\in\bZ$. It suffices to show that $[\omega^1_w]=[\omega^2_w]$ in the stalk
at $x$. By Definition \ref{de:weight}, there exist two fundamental charts
$(\bD_i,(\cY_i,\cD_i),(D_i,\delta_i),W_i,\alpha_i;y_i)$ of $(U;u)$ such that
$\alpha_i^*\omega^i_w$ belongs to
$H^q_{(w)}(\bD_i,(\cY_i,\cD_i),(D_i,\delta_i),W_i)$ for $i=1,2$. By Lemma
\ref{le:initial}, we may find another fundamental chart
$(\bD,(\cY,\cD),(D,\delta),W,\alpha;y)\in\fEt(U;u)$ as in that lemma. Then we
have $\Phi(\phi_i)^*\omega^i_w\in H^q_{(w)}(\bD,(\cY,\cD),(D,\delta),W)$ for
both $i=1,2$. However, $\Phi(\phi_1)^*\omega^1_w$ and
$\Phi(\phi_2)^*\omega^2_w$, after restriction to
$H^q_\dr(\bD\times_L(W,\pi^{-1}\cD_{\widetilde{L}}))$, must be equal, as they
are both the weight $w$ component of $\alpha^*\omega$ in
$H^q_\dr(\bD\times_L(W,\pi^{-1}\cD_{\widetilde{L}}))$ under the decomposition
\eqref{eq:decomp}. As the map $H^q_\dr(\bD\times_LW)\to
(\Omega^{q,\cl}_X/\rd\Omega^{q-1}_X)_x$ factors through
$H^q_\dr(\bD\times_L(W,\pi^{-1}\cD_{\widetilde{L}}))$, we have
$[\omega]^1_w=[\omega]^2_w$. Finally, Lemma \ref{le:stalk} below implies that
the map in (4) is surjective as well.

For (5), we take a point $y\in Y$ such that $f(y)=x$. We may take a
fundamental chart $(\bD,(\cY,\cD),(D,\delta),W,\alpha;y)$ of $(X;x)$ and
replace $X$ by $\bD\times_LW$ and $x$ by a point $(0,x)$ where $x\in W$ with
$t(x)=t$ and $s(x)=s$ such that $\dim W=s+t$. By the same proof of Lemma
\ref{le:initial} (2), we may find a fundamental chart
$(\bD',(\cY',\cD'),(D',\delta'),W',\alpha';y')$ of $(Y;y)$ such that
$(f\circ\alpha')^*H^q_{(w)}(\bD,(\cY,\cD),(D,\delta),W)\subset
H^q_{(w)}(\bD',(\cY',\cD'),(D',\delta'),W')$. This confirms Part (5) since
$H^q_{(w)}(\bD,(\cY,\cD),(D,\delta),W)$ (resp.\
$H^q_{(w)}(\bD',(\cY',\cD'),(D',\delta'),W')$) restricts to the weight $w$
part in the stalk of $\Omega^{q,\cl}_X/\rd\Omega^{q-1}_X$ (resp.\
$\Omega^{q,\cl}_Y/\rd\Omega^{q-1}_Y$) at $x$ (resp.\ $y$), by Lemma
\ref{le:stalk} below.
\end{proof}

The following lemma is the most crucial and difficult part in the proof of
the weight decomposition.

\begin{lem}\label{le:stalk}
Let the assumptions be as in Theorem \ref{th:weight}. We take a point $x\in
X$. For any fixed weight $w$, an object
$(\bD,(\cY,\cD),(D,\delta),W,\alpha;y)\in\fEt(X,x)$, and an element
$\omega\in H^q_{(w)}(\bD,(\cY,\cD),(D,\delta),W)$, the induced class
$[\omega]\in\Omega^{q,\cl}_{X,x}/\rd\Omega^{q-1}_{X,x}$ belongs to the stalk
of $(\Omega^{q,\cl}_X/\rd\Omega^{q-1}_X)_w$ at $x$.
\end{lem}

\begin{proof}
We may assume $L=K$ where $L$ is the finite extension of $K$ implicitly
contained in the data of the fundamental chart. To simplify notation, we
denote by $V$ the strictly $K$-affinoid domain $\pi^{-1}\cD_{\widetilde{K}}$
in $W$. As $\bD$ will be irrelevant in the discussion, we will regard $y$ as
a point in $V$. Moreover, by possibly shrinking $(\cY,\cD)$, the
decomposition \eqref{eq:decomposition} and Remark \ref{re:decomp}, we may
assume that the image of $\omega$ in $H^q_\dr(\bD\times_K(W,V))$ is in
$H^q_\rig(\cD/K)_w$.

\textbf{Step 1.} We choose a smooth $k^\circ$-algebra $D^\natural$ (of
dimension $s$) such that its $\varpi$-adic completion is $D^\circ$, where we
recall that $\varpi$ is a uniformizer of the discrete non-Archimedean field
$k\subset K$. In particular, we may identify $(\Spec D^\natural)_s$ with
$\cD$, and $\cM(D)$ with a strictly $k$-affinoid domain in $(\Spec
D^\natural)_k^\an$. As in Lemma \ref{le:dagger}, we have germs $(W,V)$ and
$((\Spec D^\natural)_k^\an,\cM(D))$ and a morphism
$V\to\cM(D)\widehat\otimes_kK$ induced from $\delta$. We choose a
neighborhood $U_\epsilon$ of the graph of the previous morphism as in the
proof of Lemma \ref{le:dagger}, such that the induced map
\[H^\bullet_\dr(W,V)\to H^\bullet_\dr(W\times_k(\Spec
D^\natural)_k^\an,U_\epsilon\cap(V\times_k\cM(D)))\] is an isomorphism. By a
similar argument in the proof of \cite{GK02}*{Lemma 2}, we may replace $W$ by
a smaller open neighborhood of $V$ such that there is a morphism $W\to
U_\epsilon$ sending $V$ into $U_\epsilon\cap(V\times_k\cM(D))$ whose induced
map
\[H^\bullet_\dr(W\times_k(\Spec
D^\natural)_k^\an,U_\epsilon\cap(V\times_k\cM(D)))\to H^\bullet_\dr(W,V)\] is
the inverse of the previous isomorphism. In other words, we have a morphism
$\delta'\colon W\to(\Spec D^\natural)_K^\an$ sending $V$ into
$\cM(D)\widehat\otimes_kK$ such that, although $\delta'\res_{V}$ might not
coincide with the original morphism $V\to\cM(D)\widehat\otimes_kK$ induced
from $\delta$, we still have that the induced map
\[H^\bullet_\rig(\cD/K)\simeq H^\bullet_\dr((\Spec D^\natural)_k^\an,\cM(D))\otimes_kK
\xrightarrow{\delta'^*} H^\bullet_\dr(W,V)\] coincides with the map induced from
the K\"{u}nneth decomposition \eqref{eq:decomposition} (where $\bD$ is
trivial).

\textbf{Step 2.} We choose a compactification $(\Spec
D^\natural)_k\hookrightarrow\overline{\cS}_k$ over $k$, and define
$\overline{\cS}$ to be the $k^\circ$-scheme $\overline{\cS}_k\coprod_{(\Spec
D^\natural)_k}\Spec D^\natural$. Apply \cite{dJ96}*{Theorem 8.2} to the
$k^\circ$-variety $\overline{\cS}$ and $Z=\emptyset$. We obtain a finite
extension $k'/k$, an alteration $\cS^\natural\to\overline{\cS}_{k'^\circ}$
and a $k'^\circ$-compactification $\cS^\natural\hookrightarrow\cS$ where
$\cS$ is a projective strictly semi-stable scheme over $k'^\circ$ such that
$\cS\backslash\cS^\natural$ is a strict normal crossing divisor of $\cS$
(concentrated on the special fiber). We may further assume that all
irreducible components of $\cS_s$ are geometrically irreducible. To ease
notation, we replace $k$ by $k'$ and possibly $K$ by a finite extension. We
may fix an irreducible component $\cE$ of $\cS_s$ such that its generic point
belongs to $\cS^\natural_s$ and maps to the generic point of
$\overline{\cS}_s\simeq\cD$. Note that the complement of
$\cE^\natural\coloneqq\cE\cap\cS^\natural_s$ in $\cE$ is exactly
$\cS_s^{[1]}\cap\cE$. Denote by $\sigma_\cE$ the unique point in $\cS_K^\an$
who reduction is the geometric point of $\cE_{\widetilde{K}}$. Then
$\pi^{-1}\cE_{\widetilde{K}}$ is an open neighborhood of $\sigma_\cE$.

Define $W^\natural$ via the pullback square in the following diagram:
\[\xymatrix{
W^\natural \ar[r]\ar[d]\ar@/^2pc/[rr]^-{\delta^\natural}  &  (\cS^\natural)_K^\an \ar@{^(->}[r]\ar[d]  & \cS_K^\an \\
W  \ar[r]^-{\delta'}  &  (\overline{\cS})_K^\an, }\] and
$\delta^\natural\colon W^\natural\to\cS_K^\an$ as the composition in the
upper row. We may choose a point $y^\natural\in W^\natural$ which lifts $y$
and maps to $\sigma_\cE$ in $\cS_K^\an$. The image of the form $\omega$ in
$H^q_\rig(\cD/K)_w$ induces a class $[\omega^\natural]\in
H^q_\rig(\cE^\natural/K)_w$ via restriction along the alteration. By taking a
finite unramified extension of $k$ (and possibly a finite extension of $K$),
we may assume that $(\Fr^*-p^{fw})^N[\omega^\natural]=0$ for some integer
$N\geq 1$, where $\#\widetilde{k}=p^{2f}$ and $\Fr$ denotes the relative
Frobenius of $\cE^\natural/\widetilde{k}$. Put $\fS=\widehat\cS_{/\cE}$. We
fix an open neighborhood $U$ of $\pi^{-1}\cE^\natural$ in $\fS_\eta$ such
that $[\omega^\natural]$ has a representative $\omega^\natural\in
H^q_\dr(U\widehat\otimes_kK)$.

\textbf{Step 3.} We are now going to shrink $U$ such that $\omega^\natural$
has controlled behavior on $U\backslash\pi^{-1}\cE^\natural$. We may cover
$\fS$ by finitely many special open formal $k^\circ$-subschemes $\fS_i$
satisfying the following conditions: Each $\fS_i$ is \'{e}tale over
\[\Spf k^\circ[[t_0]]\langle t_1,\dots, t_r,t_{r+1},t_{r+1}^{-1},\dots,t_s,t_s^{-1}\rangle/(t_0\cdots t_r-\varpi)\]
for some $0\leq r=r_i\leq s$; $\cE_i\coloneqq\cE\cap\fS_s$ is affine; and if
we write $f_{i,j}$ for the image of $t_j$ in $\fS_i$, then
$\cE^\natural_i\coloneqq\cE^\natural\cap\fS_s$ is defined by the equations
$f_{i,0}=0$ and $f_{i,1}\cdots f_{i,r}\neq 0$. Define the formal
$k^\circ$-scheme $\fS_i^\natural$ via the following pullback diagram
\[\xymatrix{
\fS_i^\natural \ar[r]\ar[d]  &  \fS_i \ar[d] \\
\Spf k^\circ\langle t_1,t_1^{-1},\dots,t_s,t_s^{-1}\rangle \ar[r] &\Spf
k^\circ[[t_0]]\langle t_1,\dots,
t_r,t_{r+1},t_{r+1}^{-1},\dots,t_s,t_s^{-1}\rangle/(t_0\cdots t_r-\varpi).
}\] Then $(\fS_i^\natural)_\eta=\pi^{-1}\cE_i^\natural$ in $\fS_{i\eta}$. For
$0<\epsilon<1$, denote by $\fS_{i\eta}(\epsilon)$ the open subset of
$\fS_{i\eta}$ defined by the inequality $|f_{i,0}|<|\varpi|^{1-\epsilon}$.
Then $\fS_{i\eta}(\epsilon)$ form a fundamental system of open neighborhoods
of $(\fS_i^\natural)_\eta$ in $\fS_{i\eta}$. In fact, the open subset
$\fS_{i\eta}(\epsilon)$ does not depend on the choice of the \'{e}tale
coordinates as above. We choose an open neighborhood $U_i$ of
$(\fS_i^\natural)_\eta$ in $\fS_\eta$ contained in $U$, together with an
absolute Frobenius lifting $\phi_i\colon U_i\to U$ satisfying properties
\begin{enumerate}[(a)]
  \item $\phi_i^*f_{i,j}=f_{i,j}^{p^{2f}}$ for $j=1,\dots,r$ (as in
      \cite{Chi98}*{Lemma 3.1.1});

  \item $|(\phi_i^*g-g^{p^{2f}})(x)|<1$ for all regular functions $g$ on
      $\fS_i^\natural$ and all $x\in U_i$ at which both $g$ and
      $\phi_i^*g$ are defined (as \cite{Berk07}*{Lemma 6.1.1});

  \item $(\phi_i^*-p^{fw})^M\omega^\natural=0$ in
      $H^q_\dr(U_i\widehat\otimes_kK)$ for some integer $M\geq 1$.
\end{enumerate}

Since $U_i\cap\fS_{i\eta}$ is an open neighborhood of $(\fS_i^\natural)_\eta$
in $\fS_{i\eta}$, there exists some $\epsilon_i>0$ such that
$\fS_{i\eta}(\epsilon_i)\subset U_i$. Take $\epsilon=\min\{\epsilon\}>0$, and
replace $U$ by the union
$\fS_\eta(\epsilon)\coloneqq\bigcup_i\fS_{i\eta}(\epsilon)$ in $\fS_\eta$,
which is an intrinsically defined open neighborhood of $\pi^{-1}\cE^\natural$
in $\fS_\eta$. We suppose that $\epsilon$ is very close to $0$ in terms of
$p^f,s,|\varpi|$. Now we replace $W^\natural$ by
$W^\natural\cap(\delta^\natural)^{-1}(\fS_\eta(\epsilon)\widehat\otimes_kK)$.
By construction, we may remove a Zariski closed subset of $W^\natural$ of
dimension at most $s+t-1$ such that the resulting morphism $\bD\times_K
W^\natural\to \bD\times_K W\xrightarrow{\alpha} X$ is \'{e}tale. In
particular, $(\bD\times_K W^\natural;(0,y^\natural))$ is an object of
$\Et(X;x)$.

\textbf{Step 4.} It remains to show the following claim:

For every point $u\in W^\natural$, there exists a fundamental chart
$(\bD',(\cY',\cD'),(D',\delta'),W',\alpha';y')$ of $(W^\natural;u)$ such that
$\alpha'^*(\delta^\natural)^*\omega^\natural$ belongs to
$H^q_{(w)}(\bD',(\cY',\cD'),(D',\delta'),W')$.

We start similarly as in the proof of Lemma \ref{le:initial}. By
\cite{Berk07}*{Proposition 2.3.1}, after taking a finite extension of $K$, we
may assume that $W^\natural=\bD\times_K X'$ and $u=(0,x')$ for a point $x'\in
X'$ with $t(x')=t(u)$ and $s(x')=s(u)$, where $X'$ is a smooth $K$-analytic
space of dimension $s(u)+t(u)$. Thus we have a morphism
$\delta^\natural\colon X'\simeq\{0\}\times_KX'\to
\fS_\eta(\epsilon)\widehat\otimes_kK$. If $\delta^\natural(u)$ belongs to
$(\pi^{-1}\cE^\natural)\widehat\otimes_kK$, the our claim follows in the same
way as Claim (i) in the proof of Lemma \ref{le:initial} (2). In general,
$\delta^\natural(u)$ belongs to $\fS_{i\eta}(\epsilon)\widehat\otimes_kK$ for
some $i$, and we assume that its reduction $\pi(\delta^\natural(u))$ belongs
to $(\cS_s^{[r']}\backslash\cS_s^{[r'+1]})\cap\cE_i$ for a unique $0\leq
r'\leq r_i$. (If $r'=0$, then we are back to the previous special case.)
Without lost of generality, we may assume that $r'=r_i=r$.

Let $\cF\subset\cE$ be the irreducible component of
$\cS_s^{[r]}\backslash\cS_s^{[r+1]}$ where $\pi(\delta^\natural(u))$ belongs
to. By shrinking $\fS_i$, we may assume that $\cF$ is defined by the
equations $f_{i,0}=\cdots=f_{i,r}=0$, and there exists an integrally smooth
$k$-affinoid algebra $F$ together with an isomorphism
\begin{align}\label{eq:splitting_temp}
\Spf F^\circ[[t_{i,0},\dots,t_{i,r}]]\simeq\widehat{\fS_i}_{/\cF}
\end{align}
of formal $k^\circ$-schemes, also sending $t_{i,j}$ to $f_{i,j}$. Therefore,
we have an isomorphism of graded $k$-algebras
\begin{align}\label{eq:decomp_temp}
H^\bullet_\dr(\fS_\eta,\pi^{-1}\cF)
\simeq
H^\bullet_\rig(\Spec\widetilde{F}/k)\otimes_kH^\bullet_\dr(\bE^r_\varpi),
\end{align}
where $\bE^r_\varpi$ is the $k$-analytic space in Example \ref{ex:torus}. By
\cite{GK02}*{Lemma 3} and the above isomorphism, the restriction map
\begin{align}\label{eq:decomp_temp_1}
H^\bullet_\dr(\fS_\eta,\pi^{-1}\cF)\to
H^\bullet_\dr(\fS_\eta(\epsilon),\fS_\eta(\epsilon)\cap\pi^{-1}\cF
\end{align}
is an isomorphism. Now it suffices to show that the class of
$\omega^\natural$ in
\[H^\bullet_\dr(\fS_\eta(\epsilon)\widehat\otimes_kK,\fS_\eta(\epsilon)\widehat\otimes_kK\cap\pi^{-1}\cF_{\widetilde{K}})
\simeq
H^\bullet_\dr(\fS_\eta(\epsilon),\fS_\eta(\epsilon)\cap\pi^{-1}\cF)\otimes_kK\]
is of weight $w$ with respect to the decomposition \eqref{eq:decomp_temp} and
the isomorphism \eqref{eq:decomp_temp_1}. Then our claim follows in the same
way as in the proof of Lemma \ref{le:initial} (2). Without lost of
generality, we now assume that $\omega^\natural$ is an element in
$H^\bullet_\dr(\fS_\eta(\epsilon),\fS_\eta(\epsilon)\cap\pi^{-1}\cF)$.

\textbf{Step 5.} To compute the weight, we use the Frobenius lifting
$\phi_i\colon U_i\to\fS_\eta(\epsilon)$ where $U_i\subset\fS_\eta(\epsilon)$
is an open neighborhood of $(\fS_i^\natural)_\eta$ in $\fS_\eta$, which might
be smaller than the one we start with. Assume that $U_i\cap\fS_{i\eta}$
contains $\fS_{i\eta}(\epsilon')$ for some $0<\epsilon'<\epsilon$. We
introduce more notations as follows: We fix a positive integer $N$ such that
$0<1/N<p^{-2f}\epsilon'$. Replacing $K$ by a finite extension, we may assume
that there exists a totally ramified extension $k_+/k$ contained in $K$ with
an element $\varpi_+\in k_+^\circ$ such that $\varpi_+^{rN}=\varpi$. We
consider the following $k_+$-affinoid algebras
\begin{align*}
F_0&=F\widehat\otimes_kk_+\langle\tau_1,\tau_1^{-1},\dots,\tau_r,\tau_r^{-1}\rangle, \\
F_1&=F\widehat\otimes_kk_+\left\langle\frac{t_{i,0}}{\varpi_+^{rN-r}},\frac{\varpi_+^{rN-r}}{t_{i,0}},
\frac{t_{i,1}}{\varpi_+},\frac{\varpi_+}{t_{i,1}},\dots,\frac{t_{i,r}}{\varpi_+},\frac{\varpi_+}{t_{i,r}}
\right\rangle/(t_{i,0}\cdots t_{i,r}-\varpi),\\
F_2&=F\widehat\otimes_kk_+\left\langle\frac{t_{i,0}}{\varpi_+^{rN-rp^{2f}}},\frac{\varpi_+^{rN-rp^{2f}}}{t_{i,0}},
\frac{t_{i,1}}{\varpi_+^{p^{2f}}},\frac{\varpi_+^{p^{2f}}}{t_{i,1}},\dots,\frac{t_{i,r}}{\varpi_+^{p^{2f}}},\frac{\varpi_+^{p^{2f}}}{t_{i,r}}
\right\rangle/(t_{i,0}\cdots t_{i,r}-\varpi).
\end{align*}
Note that $F_0$ is integrally smooth. We have natural isomorphisms
\begin{align*}
\rho_1&\colon F_1\xrightarrow{\sim} F_0,
\quad t_{i,j}\mapsto\varpi_+\tau_j, 1\leq j\leq r,\; t_{i,0}\mapsto\varpi_+^{rN-r}\prod_{j=1}^r\tau_j^{-1};\\
\rho_2&\colon F_2\xrightarrow{\sim} F_0,
\quad t_{i,j}\mapsto\varpi_+^{p^{2f}}\tau_j, 1\leq j\leq r,\; t_{i,0}\mapsto\varpi_+^{rN-rp^{2f}}\prod_{j=1}^r\tau_j^{-1}.
\end{align*}
For $\alpha=1,2$, we define a formal $k_+^\circ$-scheme
$\fS_i\langle\alpha\rangle$ via the following pullback diagram
\[\xymatrix{
\fS_i\langle\alpha\rangle \ar[r]\ar[d] & \widehat{\fS_i}_{/\cF}\widehat\otimes_{k^\circ}k_+^\circ \ar[d]^-{\eqref{eq:splitting_temp}}\\
\Spf F_\alpha^\circ \ar[r] & \Spf
F^\circ[[t_{i,0},\dots,t_{i,r}]]\otimes_{k^\circ}k_+^\circ },\] so that
$\fS_i\langle\alpha\rangle_\eta$ is canonically a strictly $k_+$-affinoid
domain in $U_i\widehat\otimes_kk_+$ by our choice of $N$. Moreover,
$\rho_\alpha$ induces an isomorphism, denoted again by $\rho_\alpha$,
\[\rho_\alpha\colon \Spf F_0^\circ\xrightarrow{\sim}\fS_i\langle\alpha\rangle\]
of formal $k_+^\circ$-schemes. Properties (a) and (b) of the Frobenius
lifting $\phi_i$ implies that it induces by restriction a morphism
$\phi_i\colon\fS_i\langle1\rangle_\eta\to\fS_i\langle2\rangle_\eta$, and the
composition
$\rho_{2\eta}^{-1}\circ\phi_i\circ\rho_{1\eta}\colon\cM(F_0)\to\cM(F_0)$ is a
Frobenius lifting. We fix a smooth $k_+$-affinoid germ $(V,\cM(F_0))$.

Note that for $\alpha=1,2$, we have isomorphisms
\[H^\bullet_\dr(\fS_\eta\widehat\otimes_kk_+,\pi^{-1}\cF)
\xrightarrow{\sim}
H^\bullet_\dr(\fS_\eta(\epsilon)\widehat\otimes_kk_+,\fS_\eta(\epsilon)\widehat\otimes_kk_+\cap\pi^{-1}\cF)
\xrightarrow{\sim}
H^\bullet_\dr(U_i\widehat\otimes_kk_+,\fS_i\langle\alpha\rangle_\eta)\] by
\cite{GK02}*{Lemma 3}. In particular, we may equip
$H^\bullet_\dr(U_i\widehat\otimes_kk_+,\fS_i\langle\alpha\rangle_\eta)$ with
a weight decomposition inherited from \eqref{eq:decomp_temp}. By construction
and \cite{Bos81}*{Corollary 1}, we have
\begin{itemize}
  \item a morphism
      $\rho_1^\dag\colon(V,\cM(F_0))\to(U_i\widehat\otimes_kk_+,\fS_i\langle1\rangle_\eta)$
      such that $\rho_1^\dag\res_{\cM(F_0)}$ is very close to
      $\rho_{1\eta}$ which induces the same morphism on the special
      fiber, and moreover the induced restriction map
      \[(\rho_1^\dag)^*\colon H^\bullet_\dr(U_i\widehat\otimes_kk_+,\fS_i\langle1\rangle_\eta)\to
      H^\bullet_\rig(\Spec\widetilde{F_0}/k_+)\] is an isomorphism
      respecting weights,

  \item a morphism
      $\rho_2^\dag\colon(U_i\widehat\otimes_kk_+,\fS_i\langle2\rangle_\eta)\to(V,\cM(F_0))$
      such that $\rho_2^\dag\res_{\fS_i\langle2\rangle_\eta}$ is very
      close to $\rho_{2\eta}^{-1}$ (not $\rho_{2\eta}$ !) which induces
      the same morphism on the special fiber, and moreover the induced
      restriction map
      \[(\rho_2^\dag)^*\colon H^\bullet_\rig(\Spec\widetilde{F_0}/k_+)\to
      H^\bullet_\dr(U_i\widehat\otimes_kk_+,\fS_i\langle2\rangle_\eta)\] is
      an isomorphism respecting weights.
\end{itemize}
In summary, we have weight preserving isomorphisms
\[\resizebox{16cm}{!}{\xymatrix{
& H^\bullet_\dr(\fS_\eta(\epsilon)\widehat\otimes_kk_+,\fS_\eta(\epsilon)\widehat\otimes_kk_+\cap\pi^{-1}\cF)\ar[ld]\ar[rd] \\
H^\bullet_\dr(U_i\widehat\otimes_kk_+,\fS_i\langle1\rangle_\eta) \ar[rd]_-{(\rho_1^\dag)^*}  &&
H^\bullet_\dr(U_i\widehat\otimes_kk_+,\fS_i\langle2\rangle_\eta) \\
& H^\bullet_\rig(\Spec\widetilde{F_0}/k_+). \ar[ur]_-{(\rho_2^\dag)^*} }}\]
We will identify the top three objects in the above commutative diagram.
Recall that we regard $\omega^\natural$ as an element in
$H^\bullet_\dr(\fS_\eta(\epsilon)\widehat\otimes_kk_+,\fS_\eta(\epsilon)\widehat\otimes_kk_+\cap\pi^{-1}\cF)$.
Let $\omega_0$ be the element in $H^q_\rig(\Spec\widetilde{F_0}/k_+)$ such
that $(\rho_2^\dag)^*\omega_0=\omega^\natural$. By Property (c) of the
Frobenius lifting $\phi_i$, we have that
$((\rho_1^\dag)^*\circ\phi_i^*\circ(\rho_2^\dag)^*-p^{fw})^M\omega_0=0$.
However, $\rho_2^\dag\circ\phi_i\circ\rho_1^\dag\colon
(V,\cM(F_0))\to(V,\cM(F_0))$ is a Frobenius lifting of the Frobenius
endomorphism of $\Spec\widetilde{F_0}$ over $\widetilde{k_+}=\widetilde{k}$.
Therefore, $\omega_0$ and hence $\omega^\natural$ are of weight $w$. The
lemma is finally proved!
\end{proof}

\begin{remark}
From the proof of Theorem \ref{th:weight}, we know that the support of
$(\Omega^{q,\cl}_X/\rd\Omega^{q-1}_X)_w$ is contained in the subset $\{x\in
X\res s(x)\geq 2q-w, s(x)+t(x)\geq q\}$.
\end{remark}

\section{Logarithmic differential forms}
\label{ss:log}

In this section, we study the behavior of logarithmic differential forms in
the rigid cohomology. Based on this and Theorem \ref{th:weight}, we finish
the proof of Theorem \ref{th:1} for both topologies.

Let $k$ be a finite extension of $\bQ_p$. Let $\cS$ be a proper strictly
semi-stable scheme over $k^\circ$ of dimension $s$. We fix an irreducible
component $\cE$ of $\cS_s$ and let $\cE_1,\dots,\cE_M$ be all other
irreducible components that intersect $\cE$. For a subset
$I\subset\{1,\dots,M\}$, put $\cE_I=(\bigcap_{i\in I}\cE_i)\cap\cE$ (in
particular, $\cE_\emptyset=\cE$) and
$\cE_I^\heartsuit=\cE_I\backslash\cS_s^{[|I|+1]}$. For two subsets $I,J$ of
$\{1,\dots,M\}$, we write $I\prec J$ if $I\subset J$ and numbers in
$J\backslash I$ are all greater than those in $I$.

For $I\subset\{1,\dots,M\}$, we have the open immersion
$\cE_I^\heartsuit\subset\cE_I\backslash\cS_s^{[|I|+2]}$, whose compliment is
$\coprod_{I\subset J,|J|=|I|+1}\cE_J^\heartsuit$. Thus we have maps
\[H^\bullet_\rig(\cE_I^\heartsuit/k)\to \bigoplus_{I\subset J,|J|=|I|+1}H^{\bullet+1}_{\cE_J^\heartsuit,\rig}(\cE_I\backslash\cS_s^{[|I|+2]}/k)
\xrightarrow{\sim}\bigoplus_{I\subset
J,|J|=|I|+1}H^{\bullet-1}_\rig(\cE_J^\heartsuit/k),\] where the second map is
the Gysin isomorphism. In the above composite map, denote by $\xi^I_J$ the
induced map from $H^\bullet_\rig(\cE_I^\heartsuit/k)$ to the component
$H^{\bullet-1}_\rig(\cE_J^\heartsuit/k)$ if $I\prec J$, and the zero map if
not.

In general, for $I\prec J$, there is a unique strictly increasing sequence
$I=I_0\prec I_1\prec\cdots\prec I_{|J\backslash I|}=J$ and we define
\[\xi^I_J\coloneqq\xi^{I_{|J\backslash I|-1}}_J\circ\dots\circ\xi^I_{I_1}\colon
H^\bullet_\rig(\cE_I^\heartsuit/k)\to H^{\bullet-|J\backslash
I|}_\rig(\cE_J^\heartsuit/k),\] and $\xi^I_J=0$ if $I\prec J$ does not
hold. Together, for $i\leq j$, they induce a map
\begin{align}\label{eq:gamma}
\xi^i_j\colon\bigoplus_{|I|=i}H^\bullet_\rig(\cE_I^\heartsuit/k)\to\bigoplus_{|J|=j}H^{\bullet+i-j}_\rig(\cE_J^\heartsuit/k),
\end{align}
such that $\xi^i_j\res_{H^\bullet_\rig(\cE_I^\heartsuit/k)}$ is the direct
sum of $\xi^I_J$ for all $J$ with $|J|=j$. First, we have the following
lemma.

\begin{lem}\label{le:highest}
Let notation be as above. For every $0\leq q\leq s$, the restriction of
\[\xi^0_q\colon H^q_\rig(\cE^\heartsuit/k)\to\bigoplus_{|J|=q}H^0_\rig(\cE_J^\heartsuit/k)\]
to $H^q_\rig(\cE^\heartsuit/k)_{2q}$ is injective.
\end{lem}

\begin{proof}
By the long exact sequence of cohomology with support \eqref{eq:support}, the
kernel of the map $\xi^0_q$ is a weight preserving extension of $k$-vector
spaces $H^{q+|I|}_{\cE_I,\rig}(\cE/k)$ for $|I|<q$. Therefore, the lemma
follows since $H^{q+|I|}_{\cE_I,\rig}(\cE/k)$ is pure of weight $q+|I|<2q$ by
\cite{Tsu99}*{Theorems 5.2.1 \& 6.2.5} (with constant coefficients).
\end{proof}

Denote by $Z^i(\cE)^\heartsuit$ the abelian group generated by $\cE_I$ with
$|I|=i$, modulo the subgroup generated by $\cE_I$ with $\cE_I=\emptyset$. Put
$Z(\cE)^\heartsuit=\bigoplus_{i=0}^M Z^i(\cE)^\heartsuit$. The image of
$\cE_I$ in $Z(\cE)^\heartsuit$ will be denoted by $[\cE_I]$. We define a
wedge product
\[\wedge\colon Z(\cE)^\heartsuit\otimes Z(\cE)^\heartsuit\to Z(\cE)^\heartsuit,\]
which is group homomorphism uniquely determined by the following conditions:
\begin{itemize}
  \item $Z_1\wedge Z_2=(-1)^{ij}Z_2\wedge Z_1$, if $Z_1\in
      Z^i(\cE)^\heartsuit$ and $Z_2\in Z^j(\cE)^\heartsuit$;
  \item $[\cE_I]\wedge[\cE_J]=0$ if $I\cap J\neq\emptyset$;
  \item $[\cE_I]\wedge[\cE_J]=[\cE_{I\cup J}]$ if $I\cap J=\emptyset$ and
      $I\prec I\cup J$.
\end{itemize}
It is easy to see that $\wedge$ is associative and maps
$Z^i(\cE)^\heartsuit\otimes Z^j(\cE)^\heartsuit$ into
$Z^{i+j}(\cE)^\heartsuit$. We have an (injective) class map
\[\cl^\heartsuit\colon Z(\cE)^\heartsuit\to \bigoplus_{I}H^0_\rig(\cE_I^\heartsuit/k)\simeq
\bigoplus_{I}k^{\oplus\pi_0(\cE_I^\heartsuit)}\]
sending $[\cE_I]$ to the canonical generator on (each connected component of)
$\cE_I^\heartsuit$.

For an element $f\in\cO^*(\cS_k^\an,\pi^{-1}\cE^\heartsuit)$, that is, an
invertible function on some open neighborhood of $\pi^{-1}\cE^\heartsuit$ in
$\cS_k^\an$, we can associate canonically an element $\DIV(f)\in
Z^1(\cE)^\heartsuit$. In fact, there exists an element $c\in k^\times$ such
that $|cf|=1$ on $\pi^{-1}\cE^\heartsuit$. Thus the reduction
$\widetilde{cf}$ is an element in $\cO^*_\cE(\cE^\heartsuit)$, and we define
$\DIV(f)$ to be the associated divisor of $\widetilde{cf}$, which is an
element in $Z^1(\cE)^\heartsuit$. Obviously, it does not depend on the choice
of $c$. Finally, note that by the definition of rigid cohomology, we have a
canonical isomorphism $H^\bullet_\dr(\cS_k^\an,\pi^{-1}\cE^\heartsuit)\simeq
H^\bullet_\rig(\cE^\heartsuit/k)$.

\begin{proposition}\label{pr:log}
Let notation be as above. Given
$f_1,\dots,f_q\in\cO^*(\cS_k^\an,\pi^{-1}\cE^\heartsuit)$, if we regard
$\frac{\rd{f_1}}{f_1}\wedge\cdots\wedge\frac{\rd{f_q}}{f_q}$ as an element in
$H^q_\dr(\cS_k^\an,\pi^{-1}\cE^\heartsuit)\simeq H^q_\rig(\cE^\heartsuit/k)$,
then we have
\begin{align}\label{eq:log}
\xi^0_q\(\frac{\rd{f_1}}{f_1}\wedge\cdots\wedge\frac{\rd{f_q}}{f_q}\)
=\cl^\heartsuit\(\DIV(f_1)\wedge\cdots\wedge\DIV(f_q)\).
\end{align}
\end{proposition}

\begin{proof}
The question is local around the generic point of every connected component
of $\cE_I$ with $|I|=q$. Thus, we may assume that $\cS$ is affine and admits
a smooth morphism
\[f\colon\cS\to\Spec k^\circ[T_0,\dots,
T_q]/(T_0\cdots T_q-\varpi)\] where $\varpi$ is a uniformizer of $k$, such
that
\begin{itemize}
  \item $\cE=\cE_0$ and $\cE_i$ ($i=1,\dots,q$) are all the irreducible
      components of $\cS_s$ that intersect $\cE$, where $\cE_i$ is
      defined by the ideal $(f^*T_i,\varpi)$;
  \item $\cE_I$ is irreducible and nonempty for $I\subset\{1,\dots,q\}$.
\end{itemize}
Since $\frac{\rd{(cf)}}{cf}=\frac{\rd{f}}{f}$; both sides of \eqref{eq:log}
are multi-linear in
$f_1,\dots,f_q\in\cO^*(\cS_k^\an,\pi^{-1}\cE^\heartsuit)$; and $\frac{\rd
f}{f}=\frac{\rd f'}{f'}$ in $H^1_\rig(\cE^\heartsuit/k)$ if $|f|=|f'|=1$ on
$\pi^{-1}\cE^\heartsuit$ and $\widetilde{f}=\widetilde{f'}$, we may assume
that $f_i=f^*T_i$. Then as both sides of \eqref{eq:log} are functorial in $f$
under pullback, we may assume that $\cS=\Spec k^\circ[T_0,\dots,
T_q]/(T_0\cdots T_q-\varpi)$ and $f_i=T_i$.

Put $\cS'=\Spec k^\circ[T_1,\dots,T_q]$ and let $g\colon\cS\to\cS'$ be the
morphism sending $T_i$ to $T_i$ ($1\leq i\leq q$). For
$I\subset\{1,\dots,q\}$, let $\cE'_I$ be the closed subscheme of $\cS'_s$
defined by the ideal $(\varpi,T_i\res i\in I)$. Then $g$ induces an
isomorphism $\cE_I\simeq\cE'_I$. Similarly, we have maps
\[{\xi'}^I_J\colon H^\bullet_\rig(\cE'^\heartsuit_I/k)\to H^{\bullet+i-j}_\rig(\cE'^\heartsuit_J/k)\]
for $I\subset J$ and ${\xi'}^i_j$ for $i\leq j$, where
$\cE'^\heartsuit_I=g(\cE^\heartsuit_I)$. It is easy to see that
${\xi'}^I_J=\xi^I_J$ if we identify $H^\bullet_\rig(\cE'^\heartsuit_I/k)$
with $H^\bullet_\rig(\cE^\heartsuit_I/k)$ through $g^*$. Therefore, it
suffices to show the equality \eqref{eq:log} for $\cS'$, that is,
\[{\xi'}^0_q\(\frac{\rd{T_1}}{T_1}\wedge\cdots\wedge\frac{\rd{T_q}}{T_q}\)
=1\in H^0_\rig(\cE'_{\{1,\dots,q\}}/k)\simeq k.\] However, $\cS'$, which is
isomorphic to $\bA^q_{k^\circ}$ can be canonically embedded into the proper
smooth scheme $\bP^q_{k^\circ}$ over $k^\circ$. Thus, the rigid cohomology
$H^\bullet_\rig(\cE'^\heartsuit_I/k)$ and the map ${\xi'}^i_j$ can be
computed on $(\bP^q_k)^\an$. On the generic fiber $\cS'_k$, we similarly
define $\cT_I$ to be the closed subscheme $\Spec k[T_1,\dots,T_q]/(T_i\res
i\in I)$ of $\cS'_k$ for $I\subset\{1,\dots,q\}$, and
$\cT_I^\heartsuit=\cT_I\backslash\bigcup_{I\varsubsetneq J}\cT_J$. We may
similarly define maps $\alpha^I_J\colon H^\bullet_\dr(\cT_I^\heartsuit)\to
H^{\bullet-|J\backslash I|}_\dr(\cT_J^\heartsuit)$ and $\alpha^i_j$ via
algebraic de Rham cohomology theory. Then we have canonical vertical
isomorphisms rendering the diagram
\[\xymatrix{
H^\bullet_\dr(\cT_I^\heartsuit) \ar[r]^-{\alpha^I_J}\ar[d]_-{\simeq} &  H^{\bullet-|J\backslash I|}_\dr(\cT_J^\heartsuit) \ar[d]^-{\simeq} \\
H^\bullet_\rig(\cE'^\heartsuit_I/k) \ar[r]^-{{\xi'}^I_J} &  H^{\bullet-|J\backslash I|}_\rig(\cE'^\heartsuit_J/k)
}\] commutative. From the standard computation in algebraic de Rham cohomology, we have
\[\alpha^0_q\(\frac{\rd{T_1}}{T_1}\wedge\cdots\wedge\frac{\rd{T_q}}{T_q}\)
=1\in H^0_\dr(\cT_{\{1,\dots,q\}}),\] where $\cT_{\{1,\dots,q\}}$ is just the
point of origin. Thus, the proposition is proved.
\end{proof}

Now we are ready to prove Theorem \ref{th:1}. We begin with the case of
\'{e}tale cohomology and then the case of analytic topology.

\begin{proof}[Proof of Theorem \ref{th:1} in \'{e}tale topology]

As in the previous section, sheaves like $\cO_X$, $\fc_X$, and the de Rham
complex $(\Omega^\bullet_X,\rd)$ are understood in the \'{e}tale topology.

The direct sum decomposition has been proved in Theorem \ref{th:weight} (4).
Property (i) follows from Theorem \ref{th:weight} (3).

For Property (ii), the inclusion
$\Upsilon^q_X\subset(\Omega^{q,\cl}_X/\rd\Omega^{q-1}_X)_{2q}$ follows from
Theorem \ref{th:weight} (2) and Remark \ref{re:decomp}. Now we show that
$(\Omega^{1,\cl}_X/\rd\cO_X)_2\subset\Upsilon^1_X$. We check the inclusion on
stalks. Take a point $x\in X$ with $s=s(x)$ and $t=t(x)$. For every class
$[\omega]$ in the stalk of $(\Omega^{1,\cl}_X/\rd\cO_X)_2$ at $x$, we may
find a fundamental chart $(\bD,(\cY,\cD),(D,\delta),W,\cZ,\alpha;y)$ of
$(X;x)$ such that $[\omega]$ has a representative $\omega\in
H^1_{(2)}(\bD,(\cY,\cD),(D,\delta),W)$. Note that the decomposition
\eqref{eq:decomposition} specializes to the decomposition
\[H^1_\dr(\bD\times_L(W,\pi^{-1}\cD_{\widetilde{L}}))=H^1_\rig(\cD/L)\oplus H^1_\dr(\bE^t_\varpi\widehat\otimes_kL).\]
If the restriction of $\omega$ to
$H^1_\dr(\bD\times_L(W,\pi^{-1}\cD_{\widetilde{L}}))$ belongs to
$H^1_\dr(\bE^t_\varpi\widehat\otimes_kL)$, then we are done. Otherwise,
$\omega$ restricts to $H^1_\rig(\cD/L)_2$. It suffices to show that classes
in $H^1_\rig(\cD/L)_2$ can be represented by logarithmic differential of
invertible functions \'{e}tale locally, up to a constant multiple.

We repeat certain process in Step 2 of the proof of Lemma \ref{le:stalk}.
Choose a smooth $k^\circ$-algebra $D^\natural$ (of dimension $s$) such that
its $\varpi$-adic completion is $D^\circ$, a compactification $(\Spec
D^\natural)_k\hookrightarrow\overline{\cS}_k$ over $k$, and define
$\overline{\cS}$ to be the $k^\circ$-scheme $\overline{\cS}_k\coprod_{(\Spec
D^\natural)_k}\Spec D^\natural$. Then we obtain a finite extension $k'/k$, an
alteration $\cS^\natural\to\overline{\cS}_{k'^\circ}$ and a
$k'^\circ$-compactification $\cS^\natural\hookrightarrow\cS$ where $\cS$ is a
projective strictly semi-stable scheme over $k'^\circ$ such that
$\cS\backslash\cS^\natural$ is a strict normal crossing divisor of $\cS$. We
may further assume that all irreducible components of $\cS_s$ are
geometrically irreducible. To ease notation, we replace $k$ by $k'$ and
possibly $L$ by a finite extension. We may fix an irreducible component $\cE$
of $\cS_s$ such that its generic point belongs to $\cS^\natural_s$ and maps
to the generic point of $\overline{\cS}_s\simeq\cD$. Thus there is a unique
point $\sigma_\cE\in(\widehat{\cS^\natural})_\eta$ such that
$\pi(\sigma_\cE)$ is the generic point of $\cE$.

Now we apply the setup in the beginning of this section to $\cS$ and $\cE$.
Note that $\cE\cap\cS^\natural_s$ coincides with $\cE^\heartsuit$. It
suffices to show that every class in $H^1_\rig(\cE^\heartsuit/k)_2$ can be
represented by the logarithmic differential of an invertible function on some
\'{e}tale neighborhood of $\sigma_\cE$. Put $\cE^{[i]}=\cE\cap\cS_s^{[i]}$
for $i\geq 1$. We have
$\cE^{[1]}\backslash\cE^{[2]}=\coprod_{i=1}^M\cE_{\{i\}}^\heartsuit$. There
are exact sequences
\begin{align}\label{eq:gysin}
H^1_\rig(\cE/k)\to H^1_\rig(\cE^\heartsuit/k)\to H^2_{\cE^{[1]},\rig}(\cE/k) \to H^2_\rig(\cE/k),
\end{align}
and
\begin{align*}
H^2_{\cE^{[2]},\rig}(\cE/k)\to H^2_{\cE^{[1]},\rig}(\cE/k) \to H^2_{\cE^{[1]}\backslash\cE^{[2]},\rig}(\cE\backslash\cE^{[2]}/k)
\to H^3_{\cE^{[2]},\rig}(\cE/k).
\end{align*}
As the codimension of $\cE^{[2]}$ in $\cE$ is at least $2$, we have
$H^2_{\cE^{[2]},\rig}(\cE/k)=H^3_{\cE^{[2]},\rig}(\cE/k)=0$. Thus,
\[H^2_{\cE^{[1]},\rig}(\cE/k)\simeq H^2_{\cE^{[1]}\backslash\cE^{[2]},\rig}(\cE\backslash\cE^{[2]}/k)
\simeq\bigoplus_{i=1}^M H^0_\rig(\cE^\heartsuit_{\{i\}}/k).\] Since the composition
\[\bigoplus_{i=1}^M H^0_\rig(\cE_{\{i\}}^\heartsuit/k)\simeq\bigoplus_{i=1}^M H^2_{\cE_{\{i\}},\rig}(\cE/k)\to
H^2_{\cE^{[1]},\rig}(\cE/k)\to\bigoplus_{i=1}^M
H^0_\rig(\cE_{\{i\}}^\heartsuit/k)\] is an isomorphism, we may replace the
term $H^2_{\cE^{[1]},\rig}(\cE/k)$ in \eqref{eq:gysin} by $\bigoplus_{i=1}^M
H^0_\rig(\cE_{\{i\}}/k)$, which is isomorphic to the $k$-vector space
$Z^1(\cE)^\heartsuit\otimes k$, and the boundary map
\[Z^1(\cE)^\heartsuit\otimes k\to H^2_\rig(\cE/k)\]
becomes the cycle class map in rigid cohomology. As $H^1_\rig(\cE/k)$ is of
pure weight $1$, we have the isomorphism
\begin{align}\label{eq:gysin1}
H^1_\rig(\cE^\heartsuit/k)_2\xrightarrow{\sim}\Ker(Z^1(\cE)^\heartsuit\otimes k\to H^2_\rig(\cE/k)).
\end{align}

Now take a divisor $D=\sum_{i=1}^Mc_i[\cE_{\{i\}}]$ with $c_i\in\bZ$ such
that its cycle class in $H^2_\rig(\cE/k)\simeq H^2_\r{\cris}(\cE/k)$ is
trivial. Then there exists some integer $\mu>0$ such that $\mu D$ is
algebraically equivalent to zero, and in particular $\cO_{\cE}(\mu D)$ is an
element in $\Pic^0_{\cE/\widetilde{k}}(\widetilde{k})$. Since
$\Pic^0_{\cE/\widetilde{k}}$ is a projective scheme over the finite field
$\widetilde{k}$, one may replace $\mu$ by some multiple such that
$\cO_\cE(\mu D)$ is a trivial line bundle. Therefore, there exists a function
$\tilde{f}\in\cO^*_\cE(\cE^\heartsuit)$ with $\DIV(\tilde{f})=\mu D$. We may
assume that that $\tilde{f}$ lifts to a function
$f\in\cO^*(\cS_k^\an,\pi^{-1}\cE^\heartsuit)$. (Otherwise, we may take an
affine open subscheme $\Spec D'$ of $\cS$ such that $(\Spec D')_s$ is densely
contained in $\cE^\heartsuit$ and $\tilde{f}\res_{(\Spec D')_s}$ lifts to a
function $f\in\cO(\Spec D'_k)$, and repeat the above process to $\Spec D'$.)
By Proposition \ref{pr:log}, $\frac{\rd{f}}{f}$ has image $D$ under the map
$H^1_\rig(\cE^\heartsuit/k)\to H^2_{\cE^{[1]},\rig}(\cE/k)\simeq
Z^1(\cE)^\heartsuit\otimes k$. Thus, (ii) is proved.

For Property (iii), when $X$ has dimension $1$, it follows from (the proof
of) \cite{Berk07}*{Theorem 4.3.1}. In general, it suffices to show that
$(\Omega^{1,\cl}_X/\rd\Omega^0_X)_1\subset\Psi_X$ by \cite{Berk07}*{Theorem
4.5.1 (i)} and Theorem \ref{th:weight} (4). However, this follows from the
definition of $\Psi_X$, Theorem \ref{th:weight} (5), and the case of curves.
\end{proof}

\begin{proof}[Proof of Theorem \ref{th:1} in analytic topology]

Now sheaves like $\cO_X$, $\fc_X$, $\Upsilon_X^q$, and the de Rham complex
$(\Omega^\bullet_X,\rd)$ are understood in the analytic topology. The
corresponding objects in the \'{e}tale topology will be denoted by
$\cO_{X_{\et}}$, $\fc_{X_{\et}}$, $\Upsilon_{X_{\et}}^q$, and
$(\Omega^\bullet_{X_{\et}},\rd)$.

Note that we have a canonical morphism $\nu\colon X_{\et}\to X$ of sites, and
$\cO_X=\nu_*\cO_{X_{\et}}$, $\fc_X=\nu_*\fc_{X_{\et}}$,
$\Omega^q_X=\nu_*\Omega^q_{X_{\et}}$ for every $q$. We claim that the
canonical map
$\Omega^{q,\cl}_X/\rd\Omega^{q-1}_X\to\nu_*(\Omega^{q,\cl}_{X_{\et}}/\rd\Omega^{q-1}_{X_{\et}})$
is an isomorphism. It will follow from: (a)
$\Omega^{q,\cl}_X=\nu_*\Omega^{q,\cl}_{X_{\et}}$ as subsheaves of
$\Omega^q_X$; (b) $\rd\Omega^{q-1}_X=\nu_*\rd\Omega^{q-1}_{X_{\et}}$ as
subsheaves of $\Omega^q_X$; (c) $\rR^1\nu_*\rd\Omega^{q-1}_{X_{\et}}=0$.
Assertion (a) is obvious. Both (b) and (c) will follow from the general fact
that $\rR^i\nu_*\sF=0$ for $i>0$ and any sheaf of $\bQ$-vector spaces $\sF$
on $X_{\et}$. In fact for every $x\in X$, we have
$(\rR^i\nu_*\sF)_x=H^i(\cH(x),i_x^{-1}\sF)$, where $\cH(x)$ is the completed
residue field of $x$ and $i_x\colon\cM(\cH(x))\to X$ is the canonical
morphism, and we know that the profinite Galois cohomology
$H^i(\cH(x),i_x^{-1}\sF)$ is torsion hence trivial for $i>0$.

Now for $w\in\bZ$, we define
$(\Omega^{q,\cl}_X/\rd\Omega^{q-1}_X)_w=\nu_*(\Omega^{q,\cl}_{X_{\et}}/\rd\Omega^{q-1}_{X_{\et}})_w$.
Then we have a decomposition
\[\Omega^{q,\cl}_X/\rd\Omega^{q-1}_X=\bigoplus_{w\in\bZ}(\Omega^{q,\cl}_X/\rd\Omega^{q-1}_X)_w,\]
stable under base change and functorial in $X$ and satisfying Property (i).

For Property (ii), we have the inclusion
$\Upsilon^q_X\subset\nu_*\Upsilon^q_{X_{\et}}$ as subsheaves of
$\Omega^{q,\cl}_X/\rd\Omega^{q-1}_X$, which is canonically isomorphic to
$\nu_*(\Omega^{q,\cl}_{X_{\et}}/\rd\Omega^{q-1}_{X_{\et}})$. Thus, we have
the inclusion of sheaves
$\Upsilon^q_X\subset(\Omega^{q,\cl}_X/\rd\Omega^{q-1}_X)_{2q}$. When $q=1$,
we have to show that $\nu_*\Upsilon^1_{X_{\et}}\subset\Upsilon^1_X$. We check
this on the stalk at an arbitrary point $x\in X$. Take an element $[\omega]$
in $(\nu_*\Upsilon^1_{X_{\et}})_x$. We may assume that it has a
representative $\omega\in\Omega^1_X(U)$ for some open neighborhood $U$ of $x$
satisfies $\alpha^*\omega=\frac{\rd f'}{f'}+\rd g'$ for some finite \'{e}tale
surjective morphism $\alpha\colon U'\to U$ and $f'\in\cO^*(U')$,
$g'\in\cO(U')$. Then $\omega=\deg(\alpha)^{-1}\frac{\rd f}{f}+\rd g$ where
$f$ (resp.\ $g$) is the multiplicative (resp.\ additive) trace of $f'$
(resp.\ $g'$) along $\alpha$.
\end{proof}

\section{Tropical cycle class map}
\label{ss:cycle}

In this section, we study the sheaf
$\Ker(\rd''\colon\sA_X^{q,0}\to\sA_X^{q,1})$ and its relation with de Rham
cohomology sheaves. We construct tropical cycle class maps and show their
compatibility with integration. In this section, sheaves like $\cO_X$,
$\fc_X$, and the de Rham complex $(\Omega^\bullet_X,\rd)$ are understood in
the analytic topology.

\begin{definition}[Sheaf of rational Milnor $K$-theory]\label{de:milnor}
Let $(X,\cO_X)$ be a ringed site. We define the \emph{$q$-th sheaf of
rational Milnor $K$-theory} $\sK^q_X$ for $(X,\cO_X)$ to be the sheaf
associated to the presheaf assigning an open $U$ in $X$ to
$K^M_q(\cO_X(U))\otimes\bQ$ (\cite{Sou85}*{\Sec 6.1}). Here,
$K^M_q(\cO_X(U))$ is the abelian group generated by the symbols
$\{f_1,\dots,f_q\}$ where $f_1,\dots,f_q\in\cO_X^*(U)$, modulo the relations
\begin{itemize}
  \item
      $\{f_1,\dots,f_if'_i,\dots,f_q\}=\{f_1,\dots,f_i,\dots,f_q\}+\{f_1,\dots,f'_i,\dots,f_q\}$,

  \item $\{f_1,\dots,f,\dots,1-f,\dots,f_q\}=0$.
\end{itemize}
\end{definition}

\begin{example}\label{ex:scheme}
Let $\cX$ be a smooth scheme of finite type over an arbitrary field $k$ of
dimension $n$. Then by \cite{Sou85}*{\Sec 6.1, Remarque}, we have an
isomorphism
\begin{align}\label{eq:class_milnor}
\cl_\sK\colon\CH^q(\cX)_\bQ\coloneqq\CH^q(\cX)\otimes\bQ\xrightarrow{\sim}H^q(\cX,\sK^q_\cX)
\end{align}
for every integer $q$. It can be viewed as a universal cycle class map.

If $\cZ$ is an irreducible closed subscheme of $\cX$ of codimension $q$ that
is a locally complete intersection, then $\cl_\sK(\cZ)$ has an explicit
description as follows: Choose a finite affine open covering $\cU_i$ of $\cX$
and $f_{i1},\dots,f_{iq}\in\cO_\cX(\cU_i)$ such that $\cZ\cap\cU_i$ is
defined by the ideal $(f_{i1},\dots,f_{iq})$. Let $\cU_{ij}$ be the
nonvanishing locus of $f_{ij}$. Then $\{\cU_{ij}\res j=1,\dots,q\}$ is an
open covering of $\cU_i\backslash\cZ$. Thus the element
$\{f_{i1},\dots,f_{iq}\}\in K_q^M(\cO_\cX(\bigcap_{j=1}^q\cU_{ij}))$ gives
rise to an element in $H^{q-1}(\cU_i\backslash\cZ,\sK_\cX^q)$ and hence in
$H^q_{\cZ\cap\cU_i}(\cU_i,\sK_\cX^q)$. One can show that the image in
$H^q_{\cZ\cap\cU_i}(\cU_i,\sK_\cX^q)$ does not depend on the choice of
$\{f_{i1},\dots,f_{iq}\}$. Therefore, we obtain a class $c(\cZ)$ in
$H^0(\cX,\underline{H}^q_\cZ(\cX,\sK_\cX^q))$. By
\cite{Sou85}*{Th\'{e}or\`{e}me 5}, we know that the map
$\underline{H}^i(\cX,\sK_\cX^q)\to
\underline{H}^i(\cX\backslash\cZ,\sK_\cX^q)$ is a bijection (resp.\ an
injection) if $i\leq q-2$ (resp.\ $i=q-1$), and thus
$\underline{H}^i_\cZ(\cX,\sK_\cX^q)=0$ for $i\leq q-1$. Thus, the local to
global spectral sequence induces an isomorphism $H^q_\cZ(\cX,\sK_\cX^q)\simeq
H^0(\cX,\underline{H}^q_\cZ(\cX,\sK_\cX^q))$. Then $\cl_\sK(\cZ)$ is the
image of $c(\cZ)$ under the map
$H^0(\cX,\underline{H}^q_\cZ(\cX,\sK_\cX^q))\simeq H^q_\cZ(\cX,\sK_\cX^q)\to
H^q(\cX,\sK_\cX^q)$.
\end{example}

We recall some facts from the theory of real forms on non-Archimedean
analytic spaces developed by Chambert-Loir and Ducros in \cite{CLD12}. (See
also \cite{Gub13} for a slightly different formulation.) Let $X$ be a
$K$-analytic space. There is a bicomplex
$(\sA_X^{\bullet,\bullet},\rd',\rd'')$ of sheaves of real vector spaces on
(the underlying topological space of) $X$, where $\sA_X^{q,q'}$ is the
\emph{sheaf of $(q,q')$-forms} (\cite{CLD12}*{\Sec 3.1}). Moreover, they
define another bicomplex $(\sD_X^{\bullet,\bullet},\rd',\rd'')$ of sheaves of
real vector spaces on $X$, where $\sD_X^{q,q'}$ is the \emph{sheaf of
$(q,q')$-currents}, together with a canonical map
\[\kappa_X\colon(\sA_X^{\bullet,\bullet},\rd',\rd'')\to (\sD_X^{\bullet,\bullet},\rd',\rd'')\]
of bicomplexes given by integration (\cite{CLD12}*{\Sec 4.2 \& \Sec 4.3}). It
is known that $\sA_X^{q,q'}=\sD_X^{q,q'}=0$ unless $0\leq q,q'\leq\dim(X)$.

\begin{definition}[Dolbeault cohomology]
Let $X$ be a $K$-analytic space. We define the \emph{Dolbeault cohomology}
(of forms) to be
\[H^{q,q'}_{\sA}(X)\coloneqq\frac{\Ker(\rd''\colon\sA_X^{q,q'}(X)\to\sA_X^{q,q'+1}(X))}
{\IM(\rd''\colon\sA_X^{q,q'-1}(X)\to\sA_X^{q,q'}(X))},\] and the
\emph{Dolbeault cohomology} (of currents) to be
\[H^{q,q'}_{\sD}(X)\coloneqq\frac{\Ker(\rd''\colon\sD_X^{q,q'}(X)\to\sD_X^{q,q'+1}(X))}
{\IM(\rd''\colon\sD_X^{q,q'-1}(X)\to\sD_X^{q,q'}(X))},\] together with an
induced map $\kappa_X\colon H^{q,q'}_{\sA}(X)\to H^{q,q'}_{\sD}(X)$.
\end{definition}

By \cite{Jel16}*{Corollary 4.6} and \cite{CLD12}*{Corollaire 3.3.7}, the
complex $(\sA_X^{q,\bullet},\rd'')$ is a fine resolution of
$\Ker(\rd''\colon\sA^{q,0}_X\to\sA^{q,1}_X)$. In particular, we have a
canonical isomorphism
\[H^\bullet(X,\Ker(\rd''\colon\sA^{q,0}_X\to\sA^{q,1}_X))\simeq H^{q,\bullet}_\sA(X).\]
Suppose that $X$ is of dimension $n$. By definition, we have a bilinear
pairing
\[\sD^{q,q'}_X(U)\times \sA^{n-q,n-q'}_X(U)_c\to\bR,\]
for every open $U\subset X$, where $\sA^{n-q,n-q'}_X(U)_c\subset
\sA^{n-q,n-q'}_X(U)$ is the subset of forms whose support is compact and
disjoint from the boundary of $X$. In particular, if $X$ is compact and
without boundary, then we have an induced pairing
\[\langle\;,\;\rangle_X\colon H^{q,q'}_\sA(X)\times H^{n-q,n-q'}_\sA(X)\to\bR.\]

\begin{definition}\label{de:kernel}
Let $X$ be a $K$-analytic space. We have the sheaf of rational Milnor
$K$-theory $\sK^\bullet_X$ for the ringed topological space $(X,\cO_X)$
(Definition \ref{de:milnor}).
\begin{enumerate}
  \item We define a map of sheaves
      \[\tau^q_X\colon\sK_X^q\to\Ker(\rd''\colon\sA^{q,0}_X\to\sA^{q,1}_X)\]
      as follows. For a symbol $\{f_1,\dots,f_q\}\in\sK_X^q(U)$ with
      $f_1,\dots,f_q\in\cO_X^*(U)$, we have the induced moment morphism
      $(f_1,\dots,f_q)\colon U\to(\bG_{\r{m},K}^\an)^q$. Composing with
      the evaluation map $-\log|\;|\colon(\bG_{\r{m},K}^\an)^q\to\bR^q$,
      we obtain a continuous map
      \[\trop_{\{f_1,\dots,f_q\}}\colon U\to\bR^q.\]
      If we endow the target with coordinates $x_1,\dots,x_q$ where
      $x_i=-\log|f_i|$, then we define
      \[\tau^q_X(\{f_1,\dots,f_q\})=\rd x_1\wedge\cdots\wedge\rd x_q\in\Ker(\rd''\colon\sA^{q,0}_X(U)\to\sA^{q,1}_X(U)).\]
      It is easy to see that $\tau^q_X$ factors through the relations of
      Milnor $K$-theories, and thus induces a map of corresponding
      sheaves.

  \item If $X$ is moreover smooth, then we define another map of sheaves
      \[\lambda^q_X\colon\sK_X^q\to\Omega^{q,\cl}_X/\rd\Omega^{q-1}_X\]
      as follows. For a symbol $\{f_1,\dots,f_q\}\in\sK_X^q(U)$ with
      $f_1,\dots,f_q\in\cO_X^*(U)$, we put
      \[\lambda^q_X(\{f_1,\dots,f_q\})=\frac{\rd f_1}{f_1}\wedge\cdots\wedge\frac{\rd f_q}{f_q},\]
      where the right-hand side is regarded as an element in
      $\Omega^{q,\cl}_X(U)$ and hence in
      $(\Omega^{q,\cl}_X/\rd\Omega^{q-1}_X)(U)$. It is easy to see that
      $\lambda^q_X$ factors through the relations of Milnor $K$-theories,
      and thus induces a map of corresponding sheaves.

  \item We introduce the following quotient sheaves:
      \[\sT_X^q=\sK_X^q/\Ker\tau^q_X,\quad\sL_X^q=\sK_X^q/\Ker\lambda^q_X\]
      whenever the maps are defined.
\end{enumerate}
\end{definition}

\begin{proposition}\label{pr:kernel}
Let $K$ be a non-Archimedean field and $X$ a smooth $K$-analytic space. Then
$\tau^q_X$ induces an isomorphism
\[\sT_X^q\otimes_\bQ\bR\simeq\Ker(\rd''\colon\sA^{q,0}_X\to\sA^{q,1}_X).\]
\end{proposition}

\begin{proof}
If suffices to show the isomorphism on stalks. We fix a point $x\in X$ with
$s=s(x)$ and $t=t(x)$. We first describe a typical section of
$\Ker(\rd''\colon\sA^{q,0}_X\to\sA^{q,1}_X)$ around $x$. We say a collection
of data $(U;f_1,\dots,f_N)$ where $U$ is an open neighborhood of $x$ and
$f_1,\dots,f_N\in\cO^*_X(U)$ is \emph{basic at $x$} if, under the induced
tropicalization map
\[\trop_U\colon
U\xrightarrow{f_1,\dots,f_N}(\bG_{\r{m},K}^\an)^N\xrightarrow{-\log|\;|}
T_N\otimes_\bZ\bR\simeq\bR^N\] where $T_N$ is the cocharacter lattice of
$\bG_{\r{m}}^N$, there exists a rational polyhedral complex $\cC$ of
dimension $s+t$ with a unique minimal polyhedron $\sigma_U$, which is of
dimension $t$, such that $\trop_U(U)$ is an open subset of $\cC$ and
$\trop_U(x)$ is contained in $\sigma_U$. For every polyhedron $\tau$ of
$\cC$, we denote by $\dL(\tau)$ the underlying linear $\bQ$-subspace of
$T_{N,\bQ}\coloneqq T_N\otimes_\bZ\bQ$. Then we have an inclusion
\[\sum_{\sigma_U\prec\tau\in\cC}\wedge^q\dL(\tau)\subset\wedge^qT_{N,\bQ}\]
of $\bQ$-vector spaces, and thus a map
\[\Hom_\bQ(\wedge^qT_{N,\bQ},\bR)\to\Hom_\bQ(\sum_{\sigma_U\prec\tau\in\cC}\wedge^q\dL(\tau),\bR).\]
By \cite{JSS15}*{Proposition 3.16}, the canonical map
\[\Hom_\bQ(\wedge^qT_{N,\bQ},\bR)\to\Ker(\rd''\colon\sA^{q,0}_X(U)\to\sA^{q,1}_X(U))\]
factors through
$\Hom_\bQ(\sum_{\sigma_U\prec\tau\in\cC}\wedge^q\dL(\tau),\bR)$, and moreover
every element in the stalk
$\Ker(\rd''\colon\sA^{q,0}_{X,x}\to\sA^{q,1}_{X,x})$ has a representative in
$\Hom_\bQ(\sum_{\sigma_U\prec\tau\in\cC}\wedge^q\dL(\tau),\bR)$ for some
basic data $(U;f_1,\dots,f_N)$. This implies that the induced map
$\sT_X^q\otimes_\bQ\bR\simeq\Ker(\rd''\colon\sA^{q,0}_X\to\sA^{q,1}_X)$ is
injective, as well as surjective since elements in
$\Hom_\bQ(\wedge^qT_{N,\bQ},\bQ)$ are in the image of $\tau_X^q$.
\end{proof}

\begin{remark}
Proposition \ref{pr:kernel} implies that for all $q,q'\geq 0$, we have a
canonical isomorphism
\[H^{q'}(X,\sT_X^q\otimes_\bQ\bR)\simeq H^{q,q'}_\sA(X).\]
In particular, the real vector space $H^{q,q'}_\sA(X)$ has a canonical
rational structure coming from the isomorphism
$H^{q'}(X,\sT_X^q\otimes_\bQ\bR)\simeq H^{q'}(X,\sT_X^q)\otimes_\bQ\bR$.
\end{remark}

\begin{definition}[Tropical cycle class map]\label{de:tropical}
Let $K$ be a non-Archimedean field and $\cX$ a smooth scheme over $K$.
\begin{enumerate}
  \item The \emph{tropical cycle class map (in forms)} $\cl_\sA$ is
      defined to be the
      composition
      \begin{multline*}
      \cl_\sA\colon\CH^q(\cX)_\bQ\xrightarrow{\cl_\sK}H^q(\cX,\sK_\cX^q)\to
      H^q(\cX^\an,\sK_{\cX^\an}^q)\\
      \xrightarrow{H^q(\cX^\an,\tau^q_{\cX^\an})}
      H^q(\cX^\an,\sT_{\cX^\an}^q)\to
      H^q(\cX^\an,\sT_{\cX^\an}^q\otimes_\bQ\bR)\xrightarrow{\sim}
      H^{q,q}_\sA(\cX^\an),
      \end{multline*}
      which can be regarded as a cycle class map valued in Dolbeault
      cohomology of forms.

  \item The \emph{tropical cycle class map (in currents)} $\cl_\sD$ is
      defined to be the further composition
      \[\cl_\sD\colon H^q(\cX,\sK_\cX^q)\xrightarrow{\cl_\sA}H^{q,q}_\sA(\cX^\an)
      \xrightarrow{\kappa_{\cX^\an}}H^{q,q}_\sD(\cX^\an),\] which can be
      regarded as a cycle class map valued in Dolbeault cohomology of
      currents.
\end{enumerate}
It is clear that both $\cl_\sA$ and $\cl_\sD$ are homomorphisms of graded
$\bQ$-algebras.
\end{definition}

The following theorem establishes the compatibility of tropical cycle class
maps and integration, which can be viewed as a tropical version of Cauchy
formula in multi-variable complex analysis.

\begin{theorem}\label{th:cycle}
Let $K$ be a non-Archimedean field and $\cX$ a smooth scheme over $K$ of
dimension $n$. Then for every algebraic cycle $\cZ$ of $\cX$ of codimension
$q$, we have the equality
\[\langle\cl_\sD(\cZ),\omega\rangle_{\cX^\an}=\int_{\cZ^\an}\omega\]
for every $\rd''$-closed form $\omega\in \sA^{n-q,n-q}_{\cX^\an}(\cX^\an)_c$
with compact support.
\end{theorem}

\begin{proof}
We may assume that $\cZ$ is prime, that is, a reduced irreducible closed
subscheme of $\cX$ of codimension $q$. Let $\cZ_{\r{sing}}\subset\cZ$ be the
singular locus, which is a closed subscheme of $\cX$ of codimension $>q$. Put
$\cU=\cX\backslash\cZ_{\r{sing}}$,
$\cZ_{\r{sm}}=\cZ\backslash\cZ_{\r{sing}}$, $X=\cX^\an$, $U=\cU^\an$, and
$Z=\cZ_{\r{sm}}^\an$. In particular, $Z$ is a Zariski closed subset of $U$.
To ease notation, we put
\[\sA^{q,q',\cl}_X=\Ker(\rd''\colon\sA_X^{q,q'}(X)\to\sA_X^{q,q'+1}(X)),\quad
\sD^{q,q',\cl}_X=\Ker(\rd''\colon\sD_X^{q,q'}(X)\to\sD_X^{q,q'+1}(X)).\] We
fix a form $\omega\in\sA^{q,q,\cl}_X(X)_c$. By \cite{CLD12}*{Lemme 3.2.5},
$\omega$ belongs to $\sA^{q,q,\cl}_X(U)_c$.

\textbf{Step 1.} Using Example \ref{ex:scheme}, we describe explicitly the
class $\cl_\sD(\cZ_{\r{sm}})$. We choose a finite affine open covering
$\cU_i$ of $\cU$ and $f_{i1},\dots,f_{iq}\in\cO_\cU(\cU_i)$ such that
$\cZ_{\r{sm}}\cap\cU_i$ is defined by the ideal $(f_{i1},\dots,f_{iq})$. Let
$\cU_{ij}$ be the nonvanishing locus of $f_{ij}$. Put $U_i=\cU_i^\an$ and
$U_{ij}=\cU_{ij}^\an$. Then $\{ U_{ij}\res j=1,\dots,q\}$ is an open covering
of $U_i\backslash Z$. Thus the element $\tau^q_U(\{f_{i1},\dots,f_{iq}\})$
gives rise to an element in $H^{q-1}(U_i\backslash Z,\sA^{q,0,\cl}_U)\simeq
H^{q-1}(U_i\backslash Z,\sA^{q,\bullet}_U)$, and we denote its image under
the composite map
\[H^{q-1}(U_i\backslash
Z,\sA^{q,\bullet}_U)\to H^{q-1}(U_i\backslash Z,\sD^{q,\bullet}_U) \to
H^q_{Z\cap U_i}(U_i,\sD^{q,\bullet}_U)\] by $c(Z)_i$. It is easy to see that
$c(Z)_i$ does not depend on the choice of $f_{i1},\dots,f_{iq}$. Therefore,
$\{c(Z)_i\}$ gives rise to an element $c(Z)\in
H^0(U,\underline{H}^q_Z(\sD^{q,\bullet}_U))$. Again by \cite{CLD12}*{Lemme
3.2.5}, $\underline{H}^i_Z(\sD^{q,\bullet}_U)=0$ for $i<q$, and we have an
isomorphism $H^q_Z(U,\sD^{q,\bullet}_U)\simeq
H^0(U,\underline{H}^q_Z(\sD^{q,\bullet}_U))$. The image of $c(Z)$ in
$H^q(U,\sD^{q,\bullet}_U)\simeq H^{q,q}_\sD(U)$ coincides with
$\cl_\sD(\cZ_{\r{sm}})$.

\textbf{Step 2.} We study $H^q_{U_i\cap Z}(U_i,\sD^{q,\bullet}_U)$ in more
details. Put
\[\sD^{q,\bullet}_{U,Z}(U_i)=\Ker(\sD^{q,\bullet}_U(U_i)\to\sD^{q,\bullet}_U(U_i\backslash Z)),\]
with the induced differential $\rd''$, and put
\[H^{q,q'}_{Z,\sD}(U_i)=\frac{\Ker(\rd''\colon\sD^{q,q'}_{U,Z}(U_i)\to\sD^{q,q'+1}_{U,Z}(U_i))}
{\IM(\rd''\colon\sD^{q,q'-1}_{U,Z}(U_i)\to\sD^{q,q'}_{U,Z}(U_i))}.\] As
$\sD^{q,\bullet}_U$ is a complex of flasque sheaves, we have the following
commutative diagram
\[\xymatrix{
H^{q,q'-1}_\sD(U_i\backslash Z) \ar[r]^-{\delta''}\ar[d]_-{\simeq} & H^{q,q'}_{Z,\sD}(U_i) \ar[r]\ar[d]^-{\simeq}
& H^{q,q'}_\sD(U_i) \ar[d]^-{\simeq} \ar[r]\ar[d]^-{\simeq} & H^{q,q'}_\sD(U_i\backslash Z) \ar[d]^-{\simeq} \\
H^{q'-1}(U_i\backslash Z,\sD^{q,\bullet}_U) \ar[r] & H^{q'}_{U_i\cap
Z}(U_i,\sD^{q,\bullet}_U) \ar[r] & H^{q'}(U_i,\sD^{q,\bullet}_U) \ar[r] &
H^{q'}(U_i\backslash Z,\sD^{q,\bullet}_U). }\] In particular, when $q'=q$ we have
\[H^q_{U_i\cap Z}(U_i,\sD^{q,\bullet}_U)\simeq
H^{q,q}_{Z,\sD}(U_i)\simeq
\Ker(\sD^{q,q,\cl}_U(U_i)\to\sD^{q,q}_U(U_i\backslash
Z))\subset\sD^{q,q,\cl}_U(U_i).\] Let
$\theta_i\in\sA^{q,q-1,\cl}_U(U_i\backslash Z)$ be a Dolbeault representative
of $\tau^q_U(\{f_{i1},\dots,f_{iq}\})$ as a cohomology class in
$H^{q-1}(U_i\backslash Z,\sA^{q,0,\cl}_U)$, with induced class $[\theta_i]\in
H^{q,q-1}_\sD(U_i\backslash Z)$. By partition of unity, we may write
$\omega=\sum_i\omega_i$ with $\omega_i\in\sA^{n-q,n-q}_U(U_i)_c$. Note that
\[\langle\cl_\sD(\cZ),\omega\rangle_{\cX^\an}=\langle\cl_\sD(\cZ_{\r{sm}}),\omega\rangle_U
=\langle\delta''([\theta]),\omega\rangle_U=\sum_i\langle\delta''([\theta]),\omega_i\rangle_U
=\sum_i\langle\delta''([\theta_i]),\omega_i\rangle_U\] and
\[\quad \int_{\cZ^\an}\omega=\sum_i\int_{U_i\cap Z}\omega_i.\] To prove the
theorem, it suffices to show that
\[\langle\delta''([\theta_i]),\omega_i\rangle_U=\int_{U_i\cap Z}\omega_i\]
for every $i$.

\textbf{Step 3.} In what follows, we suppress the subscript $i$. We summarize
our data as follows:
\begin{itemize}
  \item an affine smooth scheme $\cU$ over $K$ of dimension $n$, with
      $U=\cU^\an$,

  \item a smooth irreducible closed subscheme $\cZ$ of codimension $q$
      defined by the ideal $(f_1,\dots,f_q)$ where
      $f_1,\dots,f_q\in\cO_\cU(\cU)$, with $Z=\cZ^\an$,

  \item $\theta\in\sA^{q,q-1,\cl}_U(U\backslash Z)$ a Dolbeault
      representative of $\tau^q_U(\{f_1,\dots,f_q\})$ as a cohomology
      class in $H^{q-1}(U\backslash Z,\sA^{q,\bullet}_U)$, and

  \item $\omega\in\sA^{n-q,n-q}_U(U)_c$.
\end{itemize}
Our goal is to show that
\[\langle\delta''([\theta]),\omega\rangle_U=\int_Z\omega.\]
Here we recall that $[\theta]\in H^{q,q-1}_\sD(U\backslash Z)$ is the class
induced by $\theta$, and $\delta''\colon H^{q,q-1}_\sD(U\backslash Z)\to
H^{q,q}_{Z,\sD}(U)$ is the coboundary map, in which the target
$H^{q,q}_{Z,\sD}(U)$ is a subspace of $\sD^{q,q,\cl}_U(U)$. As $Z$ is a
closed Zariski subset of $U$ of codimension $q$, the image of
$\sA^{n-q,n-q}_U(U)_c$ under $\rd''$ is in $\sA^{n-q,n-q+1,\cl}_U(U\backslash
Z)_c$. By definition, the following diagram
\[\xymatrix{
H^{q,q-1}_\sD(U\backslash Z)\!\!\!\!\!\!\!\!\!\! \ar[d]_-{\delta''} & \times &
\!\!\!\!\!\!\!\!\!\!\sA^{n-q,n-q+1,\cl}_U(U\backslash Z)_c \ar[rr]^-{\langle\;,\;\rangle_U} && \bR \ar@{=}[d] \\
\sD^{q,q,\cl}_U(U)\!\!\!\!\!\!\!\!\!\! & \times & \!\!\!\!\!\!\!\!\!\!\sA^{n-q,n-q}_U(U)_c
\ar[u]_-{\rd''}\ar[rr]^-{\langle\;,\;\rangle_U} && \bR }\] is
commutative. Therefore, we have
\[\langle\delta''([\theta]),\omega\rangle_U=\int_{U\backslash Z}\theta\wedge\rd''\omega.\]
Thus it suffices to show that
\[\int_{U\backslash Z}\theta\wedge\rd''\omega=\int_Z\omega.\]
Obviously, the equality does not depend on the choice of the Dolbeault
representative.

\textbf{Step 4.} Let $\cU_i\subset\cU$ be the nonvanishing locus of $f_i$.
Then we have an open covering $\underline{U}=\{U_i\}$ of $U\backslash Z$,
where $U_i=\cU_i^\an$. For $I\subset\{1,\dots,q\}$, put $U_I=\bigcap_{i\in
I}U_i$.

Let us recall the construction of a Dolbeault representative $\theta$. We
inductively construct elements $\theta_i\in H^{q-i-1}(U\backslash
Z,\sA_U^{q,i,\cl})$ represented by an (alternative) closed \v{C}ech cocycle
\[\theta_i=\{\theta_{i,I}\in\sA^{q,i,\cl}_U(U_I)\res |I|=q-i\}\] for
$i=0,\dots,q-1$. The class $\theta_0$ is simply
\[\{\theta_{0,\{1,\dots,q\}}=\tau^q_U(\{f_1,\dots,f_q\})\in\sA^{q,0,\cl}_U(U_{\{1,\dots,q\}})\}.\]
Suppose that we have $\theta_{i-1}$ for some $1\leq i\leq q-1$. By
Poincar\'{e} lemma, we have an exact sequence
\[0\to\sA^{q,i-1,\cl}_U\to\sA^{q,i-1}_U\to\sA^{q,i,\cl}_U\to 0.\]
As $\sA^{q,i-1}_U$ is a fine sheaf, the \v{C}ech cohomology
$H^{q-i}(\underline{U},\sA^{q,i-1}_U)$ is trivial. Thus there exists
$\vartheta_i=\{\vartheta_{i,J}\in\sA^{q,i-1}_U(U_J)\res|J|=q-i\}$ with
$\delta_{\underline{U}}\vartheta_i=\theta_{i-1}$, where
$\delta_{\underline{U}}$ denotes the \v{C}ech differential for the covering
$\underline{U}$. Now we set
$\theta_i=\rd''\vartheta_i\coloneqq\{\rd''\vartheta_{i,J}\in\sA^{q,i,\cl}_U(U_J)\res|J|=q-i\}$.
The last closed \v{C}ech cocycle
$\theta_{q-1}=\{\theta_{q-1,\{i\}}\in\sA^{q,q-1,\cl}_U(U_i)\res
i=1,\dots,q\}$ is simply a Dolbeault representative of
$\tau^q_U(\{f_1,\dots,f_q\})\in H^{q-1}(U\backslash Z,\sA^{q,\bullet}_U)$.

For $\epsilon>0$ and $I\subset\{1,\dots,q\}$, put
\[V_\epsilon^I=\{x\in U\res f_i(x)\in \partial\overline{D(0,\epsilon)},i\in I;\; f_j(x)\in \overline{D(0,\epsilon)},j\not\in I\},\]
and $U_\epsilon=\overline{U\backslash V^\emptyset_\epsilon}$. Here,
$\overline{D(0,\epsilon)}$ is the closed disc of radius $\epsilon$ with
center at zero, and $\overline{U\backslash V^\emptyset_\epsilon}$ is the
closure of $U\backslash V^\emptyset_\epsilon$ in $U$. As
$\rd''\omega\in\sA^{n-q,n-q+1}_U(U\backslash Z)_c$, there is a real number
$\epsilon_0>0$ such that $\rd''\omega=0$ on $V_{\epsilon_0}^\emptyset$. Thus
for every $0<\epsilon<\epsilon_0$, we have
\begin{align}\label{eq:cycle_1}
\int_{U\backslash Z}\theta\wedge\rd''\omega=\int_{\overline{U\backslash V^\emptyset_\epsilon}}\theta\wedge\rd''\omega=
-\int_{\overline{U\backslash V^\emptyset_\epsilon}}\rd''(\theta\wedge\omega)=-\int_{U_\epsilon}\rd''(\theta_{q-1}\wedge\omega).
\end{align}
Since $U_\epsilon$ is a closed subset of $U$, the forms $\omega$ and hence
$\theta\wedge\omega$ have compact support on $U_\epsilon$.

\textbf{Step 5.} Now we have to use integration on boundaries $V^I_\epsilon$
and the corresponding Stokes' formula. We use the formulation of boundary
integration through contraction as in \cite{Gub13}*{\Sec 2}. We consider
first a tropical chart $\trop_W\colon
W\to(\bG_{\r{m},K}^\an)^N\xrightarrow{-\log|\;|}\bR^N$, where $W$ is an open
subset of $U_\epsilon$. Since $V^I_\epsilon$ is a $K^I_\epsilon$-analytic
space of dimension $n-|I|$ for some extension $K^I_\epsilon/K$ of
non-Archimedean fields, the image $\sigma_I\coloneqq\trop_W(W\cap
V^I_\epsilon)$ consists of closed faces of codimension $|I|$ of $\trop_W(W)$.
For every $i\in I$, we choose a tangent vector $\omega_i$ for the closed face
$\sigma_{\{i\}}$ of $\sigma_\emptyset$ of codimension $1$, as defined in
\cite{Gub13}*{2.8}.

Suppose that $I=\{m_1,\dots,m_j\}$ where $1\leq m_1\leq\cdots\leq m_j\leq q$.
If $\alpha$ is an $(n,n-i)$-superform on $W$ with compact support, then we
define
\[\int_{\sigma_I}\alpha\coloneqq
\int_{\sigma_I}\langle\alpha;-\omega_{m_1},\dots,-\omega_{m_j}\rangle_{\{1,\dots,j\}}.\]
It is easy too see that the above integral does not depend on the choice of
$\omega_i$; however, it does depend on the order. We may patch the above
integral to define the integral $\int_{V^I_\epsilon}\alpha$ for an
$(n,n-|I|)$-form $\alpha$ on $V^I_\epsilon$ with compact support. The
negative signs for $\omega_i$ ensure that we have the following Stokes'
formula
\[\int_{V^I_\epsilon}\rd''\alpha=\sum_{j\not\in I}(-1)^{(j,I\cup\{j\})}\int_{V^{I\cup\{j\}}_\epsilon}\alpha\]
for an $(n,n-|I|-1)$-form $\alpha$ on $V^I_\epsilon$ with compact support,
for $|I|\geq 1$. Here, $(j,J)$ is the position from the rear of the index $j$
when $J$ is ordered in the usual manner. However, for the initial Stokes'
formula, we have
\[\int_{U_\epsilon}\rd''\alpha=-\int_{\partial U_\epsilon}\alpha=-\sum_{|I|=1}\int_{V^I_\epsilon}\alpha\]
for an $(n,n-1)$-form $\alpha$ on $U_\epsilon$ with compact support.

In particular, we have
\begin{align}\label{eq:cycle_2}
-\int_{U_\epsilon}\rd''(\theta_{q-1}\wedge\omega)=\int_{\partial U_\epsilon}\theta_{q-1}\wedge\omega
=\sum_{|I|=1}\int_{V^I_\epsilon}\theta_{q-1,I}\wedge\omega.
\end{align}
In general, for $1\leq i\leq q-1$, we have
\begin{align*}
\sum_{|I|=i}\int_{V^I_\epsilon}\theta_{q-i,I}\wedge\omega&=\sum_{|I|=i}\int_{V^I_\epsilon}\rd''\vartheta_{q-i,I}\wedge\omega\\
&=\sum_{|I|=i}\int_{V^I_\epsilon}\rd''(\vartheta_{q-i,I}\wedge\omega)\\
&=\sum_{|I|=i}\sum_{j\not\in I}(-1)^{(j,I\cup\{j\})}\int_{V^{I\cup\{j\}}_\epsilon}\vartheta_{q-i,I}\wedge\omega\\
&=\sum_{|J|=i+1}\int_{V^J_\epsilon}\sum_{j\in J}(-1)^{(j,J)}\vartheta_{q-i,J\backslash\{j\}}\wedge\omega\\
&=\sum_{|J|=i+1}\int_{V^J_\epsilon}(\delta_{\underline{U}}\vartheta_{q-i})_J\wedge\omega\\
&=\sum_{|J|=i+1}\int_{V^J_\epsilon}\theta_{q-(i+1),J}\wedge\omega.
\end{align*}
Combining with \eqref{eq:cycle_1}, \eqref{eq:cycle_2}, we have
\begin{align}\label{eq:cycle_3}
\int_{U\backslash Z}\theta\wedge\rd''\omega=\int_{V^{\{1,\dots,q\}}_\epsilon}\theta_{0,\{1,\dots,q\}}\wedge\omega
=\int_{V^{\{1,\dots,q\}}_\epsilon}\tau^q_U(\{f_1,\dots,f_q\})\wedge\omega
\end{align}
for every $0<\epsilon<\epsilon_0$.

\textbf{Step 6.} By \eqref{eq:cycle_3}, the theorem is reduced to the formula
\begin{align}\label{eq:cycle_4}
\int_{V^{\{1,\dots,q\}}_\epsilon}\tau^q_U(\{f_1,\dots,f_q\})\wedge\omega
=\int_Z\omega
\end{align}
for sufficiently small $\epsilon>0$. We may choose a finite admissible
covering of $U$ by affinoid domains $W_k$, a tropical chart
$\trop_{W_k}\colon
W_k\to(\bG_{\r{m},K}^\an)^{N_k}\xrightarrow{-\log|\;|}\bR^{N_k}$, an
$(n-q,n-q)$-superform $\alpha_k$ on $\trop_{W_k}(W_k)$ whose support is
contained in the interior of $\trop_{W_k}(W_k)$, such that
$\omega=\sum_k\trop_{W_k}^*\alpha_k$. It suffices to check \eqref{eq:cycle_4}
on each $W_k$. Now we fix an arbitrary $k$ and suppress it from notation.
Suppose that the moment morphism $W\to (\bG_{\r{m},K}^\an)^N$ is defined by
functions $g_1,\dots,g_N\in\cO^*_U(W)$. To check \eqref{eq:cycle_4}, we may
assume that the morphism $(f_1,\dots,f_q)\colon W\to (\bA^q_K)^\an$ is purely
of relative dimension $n-q$ and $W_{\b{0}}\neq\emptyset$, where $W_{\b{0}}$
is the fiber over the origin. Put $W_\epsilon=W\cap
(V^\emptyset_\epsilon\backslash Z)$.

Applying \cite{CLD12}*{Proposition 4.6.6} successively, we know that there is
some $\delta>0$, such that $\trop_W(W_\delta)_n$ is isomorphic to
$\trop_W(W_{\b{0}})_{n-q}\times[-\log\delta,+\infty)^q$. Here, for a
polyhedral complex $\cC$ of dimension $n$, we denote by $\cC_n$ the union of
all polyhedra of dimension $n$. Therefore, \eqref{eq:cycle_4} follows for
every $0<\epsilon<\delta$, as on $\trop_W(W_\delta)$ we may take $\omega_i$
to be $-\frac{\partial}{\partial x_i}$, where $(x_1,\dots,x_q)$ is the
natural coordinate on $[-\log\delta,+\infty)^q$.
\end{proof}

\begin{corollary}\label{co:cycle}
Let $K$ be a non-Archimedean field and $\cX$ a proper smooth scheme over $K$
of dimension $n$. Then for every algebraic cycle $\cZ$ of $\cX$ of dimension
$0$, we have
\[\int_{\cX^\an}\cl_\sA(\cZ)=\deg\cZ.\]
\end{corollary}

The last result in this section establishes the relation of maps $\tau_X^q$
and $\lambda^q_X$.

\begin{theorem}\label{th:kernel}
Let $K$ be a non-Archimedean field embeddable into $\bC_p$, and $X$ a smooth
$K$-analytic space. Then $\Ker\tau^q_X=\Ker\lambda_X^q$. In other words, we
have a canonical isomorphism $\sT_X^q\simeq\sL_X^q$.
\end{theorem}

\begin{proof}
It suffices to check the equality on stalks. Thus we fix a point $x\in X$
with $s=s(x)$ and $t=t(x)$.

\textbf{Step 1.} Let $U$ be an open neighborhood of $x$. Take an element
$F=\sum_{i=1}^N c_i\{f_{i1},\dots,f_{iq}\}\in\sK^q_X(U)$ where $c_i\in\bQ$
and $f_{ij}\in\cO^*_X(U)$. By \cite{Berk07}*{Propositions 2.1.1, 2.3.1},
K\"{u}nneth formula, and (the proof of) Theorem \ref{th:1} (ii), there exist
\begin{itemize}
  \item a proper strictly semi-stable scheme $\cS$ over $k^\circ$ of
      dimension $s$, where $k$ is a finite extension of $\bQ_p$;

  \item an irreducible component $\cE$ of $\cS_s$ that is geometrically
      irreducible;

  \item an open neighborhood $W$ of $\pi^{-1}\cE_{\widetilde{L}}$ in
      $\cS_L^\an$ where $L$ is a finite extension of $K$ containing $k$;

  \item a closed subset $\cZ$ of dimension at most $s-1$ of $\cS_k$;

  \item a point $y\in V\coloneqq \bD\times\prod_{k=1}^tB(0;r_k,R_k)\times
      W$ which projects to $\sigma_\cE$ in $W$;

  \item a morphism $\alpha\colon V\to U$ that is \'{e}tale away from
      $\bD\times\prod_{k=1}^tB(0;r_k,R_k)\times(W\cap\cZ_L^\an)$, such
      that $\alpha(y)=x$;

  \item for each $i,j$, integers $d_{ij1},\dots,d_{ijt}$ and
      $g_{ij}\in\cO^*(W,\pi^{-1}\cE_{\widetilde{L}})$, such that
      \[\alpha^*\frac{\rd f_{ij}}{f_{ij}}-\frac{\rd\(\beta^*g_{ij}\prod_{k=1}^tT_k^{d_{ijk}}\)}{\beta^*g_{ij}\prod_{k=1}^tT_k^{d_{ijk}}}\]
      is an exact $1$-form on $V$. Here, $T_k$ is the coordinate function
      on $B(0;r_k,R_k)$ for $1\leq k\leq t$, which will be regarded as a
      function in $\cO^*(V)$ via the obvious pullback; and $\beta\colon
      V\to W$ is the projection morphism.
\end{itemize}
In particular, if we put $h_{ij}=\beta^*g_{ij}\prod_{k=1}^nT_k^{d_{ijk}}$,
then $|\alpha^*f_{ij}\cdot h^{-1}_{ij}|$ is equal to a constant
$c_{ij}\in\bR_{>0}$ on $V$.

\textbf{Step 2.} We define three tropical charts as follows.
\begin{itemize}
  \item The first one uses $f_{ij}$ ($1\leq i\leq N,1\leq j\leq q$),
      which induce a moment morphism $U\to(\bG_{\r{m},K}^\an)^{Nq}$, and
      thus a tropicalization map $\trop_U\colon
      U\to(\bG_{\r{m},K}^\an)^{Nq}\xrightarrow{-\log|\;|}\bR^{Nq}$.

  \item The second one uses functions $g_{ij}$ ($1\leq i\leq N,1\leq
      j\leq q$), which induce a moment morphism
      $W\to(\bG_{\r{m},L}^\an)^{Nq}$, and thus a tropicalization map
      $\trop_W\colon
      W\to(\bG_{\r{m},L}^\an)^{Nq}\xrightarrow{-\log|\;|}\bR^{Nq}$.

  \item The third one uses functions $T_k$ ($1\leq k\leq t$) and
      $\beta^*g_{ij}$ ($1\leq i\leq N,1\leq j\leq q$), which induce a
      moment morphism $V\to(\bG_{\r{m},L}^\an)^{t+Nq}$, and thus a
      tropicalization map $\trop_V\colon
      V\to(\bG_{\r{m},L}^\an)^{t+Nq}\xrightarrow{-\log|\;|}\bR^{t+Nq}$.
\end{itemize}
We have a commutative diagram
\[\xymatrix{
W \ar[rr]^-{\trop_W} && \bR^{Nq} \\
V \ar[d]_-{\alpha}\ar[u]^-{\beta}\ar[rr]^-{\trop_V} &&  \bR^{t+Nq} \ar[d]^-{\breve\alpha}\ar[u]_-{\breve\beta} \\
U \ar[rr]^-{\trop_U} && \bR^{Nq} }\] in which $\breve\alpha$ sends a point
$(x_k,x_{ij})\in\bR^t\times\bR^{Nq}$ to $(y_{ij})$ where $y_{ij}=-\log
c_{ij}+y_{ij}+\sum_{k=1}^tx_k$, and $\breve\beta$ is the projection onto the
last $Nq$ factors. Note that
\[\tau^q_X(F)=\sum_{i=1}^Nc_i\bigwedge_{j=1}^q\rd y_{ij},\]
and thus
\[\breve\alpha^*\tau^q_X(F)=\sum_{i=1}^Nc_i\bigwedge_{j=1}^q
\(\rd x_{ij}+\sum_{k=1}^td_{ijk}\rd x_k\)\] as a $q$-form on $\bR^{t+Nq}$. We
may write $\breve\alpha^*\tau^q_X(F)=\sum_{I\subset\{1,\dots,t\},|I|\leq
q}\rd x_I\wedge\breve\beta^*\zeta_I$ for some $(q-|I|)$-form $\zeta_I$ on
$\bR^{Nq}$.

\textbf{Step 3.} We show that $(\Ker\lambda_X^q)_x\subset(\Ker\tau^q_X)_x$.
Thus we assume that $\lambda^q_X(F)$ is an exact $q$-form on $U$ and we need
to show that $\tau^q_X(F)=0$ on a possibly smaller open neighborhood of $x$.
It suffices to that $\breve\alpha^*\tau^q_X(F)=0$ when restricted to
$\trop_V(V)$. This is true as, by Proposition \ref{pr:log}, we have that
$\zeta_I=0$ when restricted to $\trop_W(W)$ for every $I$.

\textbf{Step 4.} We show that $(\Ker\tau^q_X)_x\subset(\Ker\lambda_X^q)_x$.
Thus we may assume that $\tau^q_X(F)=0$ when restricted to $\trop_U(U)$ and
we need to show that $\lambda^q_X(F)$ is an exact $q$-form on a possibly
smaller open neighborhood of $x$. Then $\breve\alpha^*\tau^q_X(F)=0$ when
restricted to $\trop_V(V)$, and thus $\zeta_I=0$ when restricted to
$\trop_W(W)$ for every $I$. By Proposition \ref{pr:log}, the image of
$\alpha^*\lambda^q_X(F)$ in $H^q_\dr(V)$ is $0$ after possibly replacing $W$
by a smaller open neighborhood of $\pi^{-1}\cD_{\widetilde{L}}$, as the map
$\xi^0_q$ \eqref{eq:gamma} is injective on $H^q_\rig(\cE^\heartsuit/k)_{2q}$.
In particular, there is an open neighborhood $V'$ of $y$ in $V$ such that the
induced morphism $\alpha\colon V'\to U'$ is finite \'{e}tale where $U'\subset
U$ is the image of $\alpha\res_{V'}$, and
$\alpha^*\lambda^q_X(F)\res_{V'}=\rd\omega'$ for some $(q-1)$-form $\omega'$
on $V'$. Thus $\lambda^q_X(F)\res_{U'}=\deg(\alpha\res_{V'})^{-1}\rd\omega$
where $\omega$ is the trace of $\omega'$ along $\alpha\colon V'\to U'$. The
theorem follows.
\end{proof}

\section{Cohomological triviality}
\label{ss:triviality}

In this section, we study the relation between algebraic de Rham cycle
classes and tropical cycle classes.

In this section, sheaves like $\cO_X$, $\fc_X$, $\Upsilon_X^q$, and the de
Rham complex $(\Omega^\bullet_X,\rd)$ are understood in the analytic
topology. We fix an embedding $\bR\hookrightarrow\bC_p$ throughout this
section. Moreover, we have to use adic topology. By \cite{Sch12}*{Theorem
2.24}, we may associate to a $K$-analytic space $X$ an adic space $X^\ad$,
and we have a canonical continuous map $\gamma_X\colon X^\ad\to X$ of
topological spaces.

\begin{lem}\label{le:log}
Let $K$ be a non-Archimedean field embeddable into $\bC_p$, and $X$ a smooth
$K$-analytic space. Then the canonical map
$\sL^q_X\otimes_\bQ\fc_X\to\Upsilon_X^q$ is an isomorphism for every $q\geq
0$.
\end{lem}

\begin{proof}
By definition, it suffices to show that the map
$\sL^q_X\otimes_\bQ\fc_X\to\Upsilon_X^q$ is injective on stalks. Thus we fix
a point $x\in X$ with $s=s(x)$ and $t=t(x)$. Take an element $\sum_{l=1}^M
b_l \lambda_X^q(F^l)\in\sL^q_X(U)\otimes_\bQ\fc_X(U)$ such that $F=0$ in
$\Upsilon_X^q(U)$, where $U$ is a connected open neighborhood of $x$, and
$b_l\in\fc_X(U)$, $F^l\in\sK^q_X(U)$. It suffices to show that possibly after
shrinking $U$, the elements $\lambda_X^q(F^l)$ are linearly dependent in
$\Upsilon_X^q(U)$ over $\bQ$.

Write $F^l=\sum_{i=1}^{N_l} c^l_i\{f^l_{i1},\dots,f^l_{iq}\}$ where
$c^l_i\in\bQ$ and $f^l_{ij}\in\cO^*_X(U)$. We copy Step 1 of the proof of
Theorem \ref{th:kernel} to the element $F\coloneqq\sum_{l=1}^M b_lF^l$. Then
for every $I\subset\{1,\dots,t\}$ with $|I|\leq q$, we have that
\begin{align}\label{eq:dependence}
\sum_{l=1}^Mb_l\sum_{i=1}^{N_l}c^l_i\sum_{\jmath}\epsilon_\jmath\(\prod_{k\in I}d^l_{i\jmath(k)k}\)
\cl^\heartsuit\(\bigwedge_{j\not\in\IM\jmath}\DIV g^l_{ij}\)\in
\bigoplus_{J,|J|=q-|I|}H^0_\rig(\cE_J^\heartsuit/L)
\end{align}
vanishes, for some finite extension of non-Archimedean fields $L/\fc_X(U)$.
Here, $\jmath$ is taken over all injective maps $I\to\{1,\dots,q\}$; the
multi-wedge product $\bigwedge_{j\not\in\IM\jmath}\DIV g^l_{ij}$ is taken in
the increasing order for the index $j$; and $\epsilon_\jmath\in\{\pm1\}$ is
determined by $\jmath$. Note that $H^0_\rig(\cE_J^\heartsuit/L)$ is
canonically isomorphic to $\bQ^{\oplus\pi_0(\cE_J^\heartsuit)}\otimes_\bQ L$,
and for every $l$,
\[\sum_{i=1}^{N_l}c^l_i\sum_{\jmath}\epsilon_\jmath\(\prod_{k\in I}d^l_{i\jmath(k)k}\)
\cl^\heartsuit\(\bigwedge_{j\not\in\IM\jmath}\DIV g^l_{ij}\)\in
\bigoplus_{J,|J|=q-|I|}\bQ^{\oplus\pi_0(\cE_J^\heartsuit)}.\] Thus, there
exist $b'_l\in\bQ$, not all being zero, such that \eqref{eq:dependence}
vanishes for every $I$ if we replace $b_l$ by $b'_l$.

This implies that there is an open neighborhood $V'$ of $y$ in $V$ such that
the induced morphism $\alpha\colon V'\to U'$ is finite \'{e}tale where
$U'\subset U$ is the image of $\alpha\res_{V'}$, and
$\alpha^*\lambda^q_X(F')=\rd\omega'$ for some $\omega'\in\Omega^{q-1}(V')$
where $F'=\sum_{l=1}^M b'_l \lambda_X^q(F^l)$. Then
$\lambda^q_X(F')=\deg(\alpha)^{-1}\rd\omega$ where $\omega$ is the trace of
$\omega'$ along $\alpha\colon V'\to U'$. The lemma follows.
\end{proof}

The following theorem shows the finiteness of $H^{1,1}_\sA$ and studies the
tropical cycle class of line bundles.

\begin{theorem}\label{th:line}
Let $\cX$ be a proper smooth scheme over $\bC_p$. Then
\begin{enumerate}
  \item $H^{1,1}_\sA(\cX^\an)$ is finite dimensional;

  \item for a line bundle $\cL$ on $\cX$ whose (algebraic) de Rham Chern
      class $\cl_\dr(\cL)\in H^2_\dr(\cX)$ is trivial, we have
      $\cl_\sA(\cL)=0$.
\end{enumerate}
\end{theorem}

\begin{proof}
We put $X=\cX^\an$. By Theorem \ref{th:1} and Lemma \ref{le:log}, we know
that $H^1(X,\sL_X^1\otimes_\bQ\bC_p)\simeq H^1(X,\sL_X^1)\otimes_\bQ\bC_p$ is
a direct summand of $H^1(X,\Omega^{1,\cl}_X/\rd\cO_X)$.

For (1), it suffices to show that
$\dim_{\bC_p}H^1(X,\Omega^{1,\cl}_X/\rd\cO_X)<\infty$. In fact, we have a
spectral sequence $E^{p,q}_r$ abutting to
$H^\bullet_\dr(X)=H^\bullet(X,\Omega^\bullet_X)$ with the second page terms
$E^{p,q}_2=H^p(X,\Omega^{q,\cl}_X/\rd\Omega^{q-1}_X)$. Thus, it suffices to
show that both $H^3(X,\bC_p)$ and $H^2(X,\Omega^\bullet_X)$ are finite
dimensional. Since the homotopy type of $X$ is a finite CW complex,
$\dim_{\bC_p}H^i(X,\bC_p)<\infty$ for every $i\in\bZ$. By GAGA,
$H^i(X,\Omega^\bullet_X)$ is canonically isomorphic to the algebraic de Rham
cohomology $H^i_\dr(\cX)$ for every $i$, and thus finite dimensional.

For (2), note that the map $\cl_\sA\colon \CH^q(\cX)_\bQ\to H^{q,q}_\sA(X)$
factors through $H^q(X,\sL^q_X)$. We denote by $\cl(\cL)$ the corresponding
class in $H^1(X,\sL^1_X)$. It suffices to show that $\cl(\cL)$ is zero in
$H^1(X,\Omega^{1,\cl}_X/\rd\cO_X)$. Now we regard $\cl(\cL)$ as an element in
the latter cohomology group. Note that $\cl(\cL)$ maps to zero under the
coboundary map $\delta\colon H^1(X,\Omega^{1,\cl}_X/\rd\cO_X)\to
H^3(X,\bC_p)$, as the composite map $\CH^1(\cX)_\bQ\to H^3(X,\bC_p)$ fits
into the following commutative diagram
\[\xymatrix{
\CH^1(\cX)_\bQ \ar[r] & H^1(\cX,\Omega^{1,\cl}_\cX/\rd\cO_\cX) \ar[r]\ar[d]& H^1(X,\Omega^{1,\cl}_X/\rd\cO_X) \ar[d]^-{\delta} \\
& H^3(\cX,\bC_p) \ar[r] & H^3(X,\bC_p) }\] in which $H^3(\cX,\bC_p)$
vanishes. Thus, it suffices to show that the image of $\cl(\cL)$ vanishes
under the map $\Ker(\delta\colon H^1(X,\Omega^{1,\cl}_X/\rd\cO_X)\to
H^3(X,\bC_p))\to H^2(X,\Omega_X^\bullet)/H^2(X,\bC_p)$. However, by comparing
the definitions of two cycle class maps, we know that it is also the image of
$\cl_\dr(\cL)$ under the composite map $H^2_\dr(\cX)\simeq
H^2(X,\Omega_X^\bullet)\to H^2(X,\Omega_X^\bullet)/H^2(X,\bC_p)$, thus
vanishes.
\end{proof}

The following lemma is the analytic version of the corresponding statement in
the algebraic setting.

\begin{lem}\label{le:top}
Let $K$ be a non-Archimedean field. Let $\cX$ be a geometrically connected
proper smooth scheme over $K$ of dimension $n$. Then we have
$H^n(\cX^\an,\Omega^n_{\cX^\an}/\rd\Omega^{n-1}_{\cX^\an}) \simeq K$.
\end{lem}

\begin{proof}
Put $X=\cX^\an$. By the spectral sequence
$E^{p,q}_2=H^p(X,\Omega^{q,\cl}_X/\rd\Omega^{q-1}_X)\Rightarrow
H^{p+q}(X,\Omega^\bullet_X)$ and the GAGA comparison isomorphism
$H^\bullet_\dr(X)\simeq H^\bullet_\dr(\cX)$, it suffices to show that
$H^i(X,\sF)=0$ for $i>n$ and every abelian sheaf $\sF$ on $X$. Then we have
$H^n(X,\Omega^n_X/\rd\Omega^{n-1}_X)\simeq H^{2n}_\dr(\cX)\simeq K$.

By \cite{Berk93}*{Proposition 1.3.6 \& Lemma 1.6.2} and \cite{Sch12}*{Theorem
2.21}, we have a canonical isomorphism $H^i(X,\sF)\simeq
H^i(X^\ad,\gamma_X^{-1}\sF)$.

In fact, we will show that $H^i(X^\ad,\sF)=0$ for $i>n$ and every abelian
sheaf $\sF$ on $X^\ad$. Recall that a formal model of $X$ is a proper flat
formal $K^\circ$-scheme $\fX$ with an isomorphism $\fX_\eta\simeq X$. A
formal model $\fX$ induces a continuous map $\gamma_\fX\colon X^\ad\to\fX$.
By \cite{Sch12}*{Theorem 2.22} and \cite{SP}*{094L, 0A2Z}, we have an
isomorphism $X^\ad\simeq\varprojlim_\fX\fX$ of spectral spaces, where the
(cofiltered) limit is taken over all formal models $\fX$ of $X$. By
\cite{SP}*{0A37}, we have an isomorphism
$\varinjlim_{\fX}H^i(\fX,\gamma_{\fX*}\sF)\simeq H^i(X^\ad,\sF)$. Now as (the
underlying space of) $\fX$ is a Noetherian topological space of dimension (at
most) $n$, it follows that $H^i(\fX,\gamma_{\fX*}\sF)=0$ for $i>n$ by
Grothendieck vanishing theorem \cite{SP}*{02UZ}. The lemma then follows.
\end{proof}

\begin{definition}
Let $X$ be a compact smooth $\bC_p$-analytic space of dimension $n$. We have
a total integration map
\[\int_X\colon H^n(X,\sT^n_X\otimes_\bQ\bR)\simeq H^{n,n}_\sA(X)\to\bR.\]
By the isomorphism $\sT^n_X\simeq\sL^n_X$ in Theorem \ref{th:kernel}, Lemma
\ref{le:log}, and by extending the above map linearly over $\bC_p$, we obtain
a $\bC_p$-linear map
\[\Tr_X^\sA\colon H^n(X,\Upsilon^n_X)\simeq H^n(X,\sL^n_X\otimes_\bQ\bC_p)\to\bC_p,\]
called the \emph{trace map} for $X$.
\end{definition}

The following proposition can be regarded as certain algebraicity property of
the transcendental map of integration.

\begin{proposition}\label{pr:trace}
Let $k\subset\bC_p$ be a discrete non-Archimedean subfield, and $\cX$ a
geometrically connected proper smooth scheme over $k$ of dimension $n$. If we
put $\cX_\ra=\cX\otimes_k\bC_p$, then the map $\Tr_{\cX^\an_\ra}^\sA$ factors
through the canonical map
\[H^n(\cX_\ra^\an,\Upsilon^n_{\cX_\ra^\an})\to
H^n(\cX_\ra^\an,\Omega^n_{\cX_\ra^\an}/\rd\Omega^{n-1}_{\cX_\ra^\an})\simeq\bC_p,\]
where we have used Lemma \ref{le:top} for the last isomorphism.

In particular,
$H^n(\cX_\ra^\an,(\Omega^n_{\cX_\ra^\an}/\rd\Omega^{n-1}_{\cX_\ra^\an})_w)\simeq\bC_p$
if $w=2n$, and is trivial otherwise.
\end{proposition}

\begin{proof}
We put $X=\cX_\ra^\an$. Define $\Xi^n_X$ to be the quotient sheaf in the
following exact sequence
\[0\to \Upsilon^n_X \to (\Omega^n_X/\rd\Omega^{n-1}_X)_{2n}\to \Xi^n_X \to 0.\]
It is functorial in $X$. The best hope is that $\Xi^n_X$ is trivial; but we
do not know so far. However, one can show that $\Xi^n_X$ is supported on
$\{x\in X\res s(x)\geq 2\}$. This suggests that one should expect
$H^i(X,\Xi^n_X)=0$ for $i\geq n-1$, which suffices for the proposition. In
fact, such vanishing result can be proved if we have semi-stable resolution
instead of alteration. In the absence of semi-stable resolution, we need an
\emph{ad hoc} argument.

We may assume that $\cX$ is projective, as we will eventually take an
alteration of $\cX$. Take a cohomology class $\alpha\in H^{n-1}(X,\Xi^n_X)$.
Since $X$ is (Hausdorff and) compact, by \cite{SP}*{09V2, 01FM}, there is a
finite open covering $\underline{U}=\{U_i\res i=1,\dots,N\}$ of $X$ such that
$\alpha$ is represented by an (alternative) \v{C}ech cocycle
$\underline\alpha=\{\alpha_I\in\Xi^n_X(U_I)\res I\subset\{1,\dots,N\},
|I|=n\}$ on $\underline{U}$, where $U_I=\bigcap_{i\in I}U_i$ as always. By
refining $\underline{U}$, we may assume that $\alpha_I$ is in the image of
the map $(\Omega^n_X/\rd\Omega^{n-1}_X)(U_I)^\r{pre}_{2n}\to\Xi^n_X(U_I)$ for
every $I$ (See Definition \ref{de:weight} for the notation). By
\cite{Pay09}*{Theorem 4.2}, taking blow-ups, and possibly taking a finite
extension of $k$ inside $\bC_p$, we have a (proper flat) integral model $\cY$
of $\cX$ such that if $\cZ_1,\dots,\cZ_M$ are all reduced irreducible
components of $\cY_s$, then the covering
$\{\pi^{-1}\cZ_i\widehat\otimes_k\bC_p\res i=1,\dots,M\}$ refines
$\underline{U}$. By \cite{dJ96}*{Theorem 8.2}, possibly after taking further
finite extension of $k$ inside $\bC_p$, we have a (proper) strictly
semi-stable scheme $\cY'$ over $k^\circ$ with an alteration $\cY'\to \cY$.
For simplicity, we may also assume that every irreducible component of
$\cY'^{[t]}_s$ ($0\leq t\leq n$) is geometrically irreducible. In particular,
if we denote by $\cZ'_1,\dots,\cZ'_{M'}$ all irreducible components of
$\cY'_s$, then the covering $\{\pi^{-1}\cZ'_i\widehat\otimes_k\bC_p\res
i=1,\dots,M'\}$ refines $f^{-1}\underline{U}\coloneqq\{f^{-1}U_i\res
i=1,\dots,N\}$. We fix an index function
$\varrho\colon\{1,\dots,M'\}\to\{1,\dots,N\}$ for the refinement; in other
words, $\pi^{-1}\cZ'_i\widehat\otimes_k\bC_p$ is contained in
$f^{-1}U_{\varrho(i)}$.

We fix a uniformizer $\varpi$ of $k$, and put $X'=(\cY'\otimes_k\bC_p)^\an$.
We claim that $f^*\alpha=0$, where $f^*\alpha$ is the canonical image of
$f^{-1}\alpha$ in $H^{n-1}(X',\Xi^n_{X'})$.

For every $1\leq i\leq M'$ and $0<\epsilon<1$, we denote by $U_i(\epsilon)$
the open subset of $\pi^{-1}\cZ'_i\widehat\otimes_k\bC_p\subset X'$ as in
Step 3 in the proof of Lemma \ref{le:stalk}. Then we have that
\[\pi^{-1}\cZ'_i\widehat\otimes_k\bC_p=\bigcup_{0<\epsilon<1}U_i(\epsilon),\quad
\pi^{-1}(\cZ'_i\backslash\cX'^{[1]}_s)\widehat\otimes_k\bC_p=\bigcap_{0<\epsilon<1}U_i(\epsilon).\]
By definition, $\underline{U}(\epsilon)\coloneqq\{U_i(\epsilon)\res 1\leq
i\leq M'\}$ form an open covering if and only if $\epsilon>1/2$. Fix a real
number $1/2<\epsilon<1$. We study a typical $n$-fold intersection of
$\underline{U}(\epsilon)$. Without lost of generality, we consider
$U_{\{1,\dots,n\}}(\epsilon)\coloneqq\bigcap_{i=1}^nU_i(\epsilon)$. If
$\bigcap_{i=1}^n\cZ'_i=\emptyset$, then
$U_{\{1,\dots,n\}}(\epsilon)=\emptyset$. So we may assume that
$\bigcap_{i=1}^n\cZ'_i=\bigsqcup_{l=1}^L\cC_l$, where each $\cC_l$ is a
geometrically irreducible proper smooth curve over $\widetilde{k}$.

Take a typical member $\cC$ of $\{\cC_1,\dots,\cC_L\}$ and put
$U_\cC(\epsilon)=U_{\{1,\dots,n\}}(\epsilon)\cap\pi^{-1}\cC\widehat\otimes_k\bC_p$.
Recall the $k$-analytic space $\bE^{n-1}_\varpi$ defined in Example
\ref{ex:torus}. Let $\bE^t_\varpi(\epsilon)\subset\bE^t_\varpi$ be the
subspace such that $|T_i|<|\varpi|^{1-\epsilon}$ for every $0\leq i\leq t$.
By \cite{GK02}*{Lemma 3}, the canonical map $H^\bullet_\dr(\bE^t_\varpi)\to
H^\bullet_\dr(\bE^t_\varpi(\epsilon))$ is an isomorphism. Put
$\cC^\heartsuit=\cC\backslash\cX'^{[n]}_s$ and
$U_\cC^\heartsuit(\epsilon)=U_\cC(\epsilon)\cap\pi^{-1}\cC^\heartsuit\widehat\otimes_k\bC_p$.
By the proof of \cite{GK02}*{Theorem 2.3}, we have isomorphisms
\begin{align*}
H^\bullet_\dr(U_\cC(\epsilon),U_\cC^\heartsuit(\epsilon))&\simeq
\Tot(H^\bullet_\dr(\bE^{n-1}_\varpi(\epsilon)) \otimes_k
H^\bullet_\rig(\cC^\heartsuit/k))\otimes_k\bC_p\\
&\simeq\Tot(H^\bullet_\dr(\bE^{n-1}_\varpi) \otimes_k
H^\bullet_\rig(\cC^\heartsuit/k))\otimes_k\bC_p
\end{align*}
of graded $\bC_p$-vector spaces. By Theorem \ref{th:1} (ii), there is an open
neighborhood $U_1$ of $U_\cC^\heartsuit(\epsilon)$ in $U_\cC(\epsilon)$ such
that the image of $f^{-1}\alpha_{\{\varrho(1),\dots,\varrho(n)\}}\res_{U_1}$
in $\Xi_{X'}^\an(U_1)$ zero. Put $U_2\coloneqq
U_\cC(\epsilon)\cap\pi^{-1}(\cC\backslash\cC^\heartsuit)\widehat\otimes_k\bC_p$.
Then $H^\bullet_\dr(U_2)$ is isomorphic to a finite copy of
$H^\bullet_\dr(\bE^n_\varpi)\otimes_k\bC_p$, and in particular, the image of
$f^{-1}\alpha_{\{\varrho(1),\dots,\varrho(n)\}}\res_{U_2}$ in
$\Xi_{X'}^\an(U_2)$ is zero. Finally, note that $U_\cC(\epsilon)=U_1\cup
U_2$, which implies that $f^*\alpha=0$.

Going back to $X$, we have the following commutative diagram
\[\xymatrix{
H^{n-1}(X,\Xi^n_X) \ar[r]\ar[d]_-{f^*} & H^n(X,\Upsilon^n_X) \ar[rr]^-{\Tr_X^\sA}\ar[d]^-{f^*} && \bC_p \ar[d]^-{\deg(f)} \\
H^{n-1}(X',\Xi^n_{X'}) \ar[r] & H^n(X',\Upsilon^n_{X'})
\ar[rr]^-{\Tr_{X'}^\sA} && \bC_p, }\] where the right vertical arrow is the
multiplication by $\deg(f)$. Therefore, $\Tr_X^\sA$ factors through the map
$H^n(X,\Upsilon^n_X)\to H^n(X,\Omega^n_X/\rd\Omega^{n-1}_X)$.

The last statement follows from the combination of
\begin{itemize}
  \item $\Tr_X^\sA$ is surjective as one can write down an $(n,n)$-form
      on $X$ with nonzero total integral;

  \item
      $H^n(X,\Omega^n_X/\rd\Omega^{n-1}_X)=\bigoplus_{w\in\bZ}H^n(X,(\Omega^n_X/\rd\Omega^{n-1}_X)_w)$
      by Theorem \ref{th:1};

  \item $H^n(X,\Omega^n_X/\rd\Omega^{n-1}_X)\simeq\bC_p$ by Lemma
      \ref{le:top}; and

  \item the image of $H^n(X,\Upsilon^n_X)\to
      H^n(X,\Omega^n_X/\rd\Omega^{n-1}_X)$ is contained in
      $H^n(X,(\Omega^n_X/\rd\Omega^{n-1}_X)_{2n})$.
\end{itemize}
\end{proof}

The next theorem shows that algebraic cycles that are cohomologically trivial
in the algebraic de Rham cohomology are cohomologically trivial in of
Dolbeault cohomology of currents as well.

\begin{theorem}\label{th:trivial}
Let $k\subset\bC_p$ be a finite extension of $\bQ_p$ and $\cX$ a proper
smooth scheme over $k$ of dimension $n$. Let $\cZ$ be an algebraic cycle of
$\cX$ of codimension $q$ such that $\cl_\dr(\cZ)=0$. If we put
$\cX_\ra=\cX\otimes_k\bC_p$ and $\cZ_\ra=\cZ\otimes_k\bC_p$, then
$\cl_\sD(\cZ_\ra)=0$, that is,
\[\int_{\cZ_\ra^\an}\omega=0\]
for every $\rd''$-closed form $\omega\in\sA^{n-q,n-q}(\cX_\ra^\an)$.
\end{theorem}

We need some preparation before the proof of the theorem. We start from the
following lemma.

\begin{lem}\label{le:finite}
Let the assumption and notation be as in Theorem \ref{th:trivial}. Then for
every $i$ and $q$, the canonical map
\[\varinjlim_{k'}H^i(\cX_{k'}^\an,\sT^q_{\cX_{k'}^\an})\to H^i(\cX_\ra^\an,\sT^q_{\cX_\ra^\an})\]
is an isomorphism, where the colimit is taken over all finite extensions $k'$
of $k$ in $\bC_p$.
\end{lem}

\begin{proof}
We have an isomorphism of spectral spaces
$\cX_\ra^\ad\simeq\varprojlim_{k'}\cX_{k'}^\ad$. Thus by \cite{SP}*{0A37}, it
suffices to show that the canonical map
$\varinjlim_{k'}\varsigma_{k'}^{-1}\gamma^{-1}_{\cX_{k'}^\an}\sT^q_{\cX_{k'}^\an}\to\gamma^{-1}_{\cX_\ra^\an}\sT^q_{\cX_\ra^\an}$
is an isomorphism, where $\varsigma_{k'}\colon\cX_\ra^\ad\to\cX_{k'}^\ad$ is
the canonical map. However, this follows from the fact that for every
$f\in\cO^*(\cX_\ra^\an,V)$ where $V$ is a rational affinoid domain, there is
a function $g\in\cO^*(\cX_{k'}^\an,V')$ for some $k'$ such that
$\varsigma_{k'}^{-1}(V')=V$ and $f^{-1}\cdot\varsigma_{k'}^*g$ has norm $1$
on some open neighborhood of $V$. Here, we have used \cite{Berk07}*{Lemma
2.1.3 (ii)}.
\end{proof}

We review some facts about cup products from \cite{SP}*{01FP}. Let $X$ be a
topological space, $k$ a field, $n\geq 0$ an integer. Let $\Omega$ be a sheaf
of $k$-vector spaces on $X$. Suppose that we have two bounded complexes
$\sF^\bullet,\sG^\bullet$ of sheaves of $k$-vector spaces on $X$, with a map
of complexes of sheaves of $k$-vector spaces
\[\chi\colon\Tot(\sF^\bullet\otimes_k\sG^\bullet)\to\Omega[n].\]
Then we have a bilinear pairing
\[\cup_\chi\colon H^i(X,\sF^\bullet)\times H^{2n-i}(X,\sG^\bullet)\to H^{2n}(X,\Omega[n])=H^n(X,\Omega)\]
for every $i\in\bZ$. Now suppose that we have four bounded complexes
$\sF^\bullet_j,\sG^\bullet_j$ ($j=1,2$) of sheaves of $k$-vector spaces on
$X$, maps $\alpha_1\colon\sF_1^\bullet\to\sF_2^\bullet$,
$\alpha_2\colon\sG_2^\bullet\to\sG_1^\bullet$, and
$\chi_j\colon\Tot(\sF^\bullet_j\otimes_k\sG^\bullet_j)\to\Omega[n]$
($j=1,2$), such that $\chi_1\circ(\r{id}_{\sF_1^\bullet}\otimes
\alpha_2)=\chi_2\circ(\alpha_1\otimes\r{id}_{\sG_2^\bullet})$. Then we have
the following commutative diagram
\begin{align}\label{eq:cup}
\xymatrix{
H^i(X,\sF^\bullet_1)\!\!\!\!\!\!\!\!\!\! \ar[d]_-{H^i(X,\alpha_1)} & \times &
\!\!\!\!\!\!\!\!\!\!H^{2n-i}(X,\sG^\bullet_1) \ar[rr]^-{\cup_{\chi_1}} && H^n(X,\Omega) \ar@{=}[d] \\
H^i(X,\sF^\bullet_2)\!\!\!\!\!\!\!\!\!\! & \times &
\!\!\!\!\!\!\!\!\!\!H^{2n-i}(X,\sG^\bullet_2)
\ar[u]_-{H^{2n-i}(X,\alpha_2)}\ar[rr]^-{\cup_{\chi_2}} && H^n(X,\Omega)  }
\end{align}
for every $i\in\bZ$.

\begin{proof}[Proof of Theorem \ref{th:trivial}]
Without lost of generality, we may assume that $\cX$ is geometrically
irreducible over $k$ (of dimension $n$). Put $X=\cX^\an$.

\textbf{Step 1.} By Proposition \ref{pr:kernel} and Theorem \ref{th:kernel},
we have the following commutative diagram
\[\xymatrix{
H^{q,q}_\sA(\cX_\ra^\an)\times H^{n-q,n-q}_\sA(\cX_\ra^\an) \ar[r]^-{\cup}\ar[d]_-{\simeq}& H^{n,n}_\sA(\cX_\ra^\an) \ar[d]^-{\simeq} \\
H^q(\cX_\ra^\an,\sT_{\cX_\ra^\an}^q\otimes_\bQ\bR) \times H^{n-q}(\cX_\ra^\an,\sT_{\cX_\ra^\an}^{n-q}\otimes_\bQ\bR)
\ar[r]^-{\cup}\ar[d] & H^n(\cX_\ra^\an,\sT_{\cX_\ra^\an}^n\otimes_\bQ\bR) \ar[d]\\
H^q(\cX_\ra^\an,\Omega^{q,\cl}_{\cX_\ra^\an}/\rd\Omega^{q-1}_{\cX_\ra^\an}) \times
H^{n-q}(\cX_\ra^\an,\Omega^{n-q,\cl}_{\cX_\ra^\an}/\rd\Omega^{n-q-1}_{\cX_\ra^\an})
\ar[r]^-{\cup}& H^n(\cX_\ra^\an,\Omega^n_{\cX_\ra^\an}/\rd\Omega^{n-1}_{\cX_\ra^\an}),
}\] in which the first cup product is induced by the wedge product of real forms.

To prove the theorem, it suffices to consider an arbitrary element $\omega\in
H^{n-q}(\cX_\ra^\an,\sT_{\cX_\ra^\an}^{n-q})$. In view of Lemma
\ref{le:finite}, after replacing $k$ by a finite extension in $\bC_p$, we may
assume that $\omega\in H^{n-q}(X,\sT_X^{n-q})$.

Note that the tropical cycle class map $\cl_\sA$ (Definition
\ref{de:tropical}) factors as
\[\CH^q(\cX)_\bQ\to H^q(X,\sT_X^q)\to H^q(\cX_\ra^\an,\sT_{\cX_\ra^\an}^q)
\to H^q(\cX_\ra^\an,\sT_{\cX_\ra^\an}^q\otimes_\bQ\bR)\simeq
H^{q,q}_\sA(\cX_\ra^\an),\] in which we denote the first map by $\cl_\sT$. By
Theorem \ref{th:cycle} and Proposition \ref{pr:trace}, it suffices to show
that the image of $\cl_\sT(\cZ)\cup\omega$ in
$H^n(X,\Omega^n_X/\rd\Omega^{n-1}_X)$, which is isomorphic to $k$ by Lemma
\ref{le:top}, is zero. Denote by $\zeta$ the image of $\cl_\sT(\cZ)$ in
$H^q(X,\Omega^{q,\cl}_X/\rd\Omega^{q-1}_X)$, and regard $\omega$ as in
$H^{n-q}(X,\Omega^{n-q,\cl}_X/\rd\Omega^{n-q-1}_X)$. In fact, we can prove
that $\zeta\cup\omega=0$ if we have semi-stable resolution instead of
alteration. In the absence of semi-stable resolution, we need an \emph{ad
hoc} argument.

\textbf{Step 2.} To proceed, we need the adic topology of $X$. Recall that we
have a continuous map $\gamma_X\colon X^\ad\to X$. Let
$(\Omega^\bullet_{X^\ad},\rd)$ the de Rham complex on $X^\ad$. Then we have a
canonical map
$\gamma_X^{-1}(\Omega^\bullet_X,\rd)\to(\Omega^\bullet_{X^\ad},\rd)$ of
complexes of sheaves of $k$-vector spaces on $X^\ad$. Denote by $\zeta_\ad$
(resp.\ $\omega_\ad$) the image of $\zeta$ (resp.\ $\omega$) under the
canonical map
\[H^i(X,\Omega^{i,\cl}_X/\rd\Omega^{i-1}_X)\to H^i(X^\ad,\Omega^{i,\cl}_{X^\ad}/\rd\Omega^{i-1}_{X^\ad})\]
for $i=q$ (resp.\ $i=n-q$). Note that when $i=n$, the above map is an
isomorphism, by the same argument for Lemma \ref{le:top}.

We claim that there exists an alteration $f\colon\cX'\to\cX$ possibly after a
finite extension of $k$ in $\bC_p$, such that $f^*\omega_\ad$ is in the image
of the canonical map
\[H^{2n-2q}(X'^\ad,\tau_{\leq n-q}\Omega^\bullet_{X'^\ad})
\to H^{n-q}(X'^\ad,\Omega^{n-q,\cl}_{X'^\ad}/\rd\Omega^{n-q-1}_{X'^\ad}),\] where $X'=\cX'^\an$.

Assuming the above claim, we deduce the theorem as follows. Applying
\eqref{eq:cup} to $X'^\ad$ and the sheaf
$\Omega\coloneqq\Omega^n_{X'^\ad}/\rd\Omega^{n-1}_{X'^\ad}$, we obtain the
following commutative diagram
\[\xymatrix{
H^q(X'^\ad,\Omega^{q,\cl}_{X'^\ad}/\rd\Omega^{q-1}_{X'^\ad})\!\!\!\!\!\!\!\!\!\! \ar[d]_-{\alpha_1} & \times &
\!\!\!\!\!\!\!\!\!\!H^{n-q}(X'^\ad,\Omega^{n-q,\cl}_{X'^\ad}/\rd\Omega^{n-q-1}_{X'^\ad}) \ar[rr] && H^n(X'^\ad,\Omega) \ar@{=}[d] \\
H^{2q}(X'^\ad,\tau_{\geq q}\Omega^\bullet_{X'^\ad})\!\!\!\!\!\!\!\!\!\!  & \times &
\!\!\!\!\!\!\!\!\!\!H^{2n-2q}(X'^\ad,\tau_{\leq n-q}\Omega^\bullet_{X'^\ad})
\ar[u]_-{\alpha_2}\ar[d]^-{\beta_2}\ar[rr] && H^n(X'^\ad,\Omega) \ar@{=}[d] \\
H^{2q}(X'^\ad,\Omega^\bullet_{X'^\ad})\!\!\!\!\!\!\!\!\!\!  \ar[u]^-{\beta_1}
& \times & \!\!\!\!\!\!\!\!\!\!H^{2n-2q}(X'^\ad,\Omega^\bullet_{X'^\ad})
\ar[rr] && H^n(X'^\ad,\Omega) }\] in which the maps among various complexes
of sheaves are defined in the obvious way. By the above claim, there exists
$\omega'\in H^{2n-2q}(X'^\ad,\tau_{\leq n-q}\Omega^\bullet_{X'^\ad})$ such
that $\alpha_2(\omega')=f^*\omega_\ad$. Thus,
\[f^*\zeta_\ad\cup f^*\omega_\ad = f^*\zeta_\ad\cup\alpha_2(\omega')
=\alpha_1(f^*\zeta_\ad)\cup\omega'=\beta_1(\cl_\dr(f^*\cZ))\cup\omega'
=\cl_\dr(f^*\cZ)\cup\beta_2(\omega'),\] where we regard $\cl_\dr(f^*\cZ)$ as
an element in $H^{2q}(X'^\ad,\Omega^\bullet_{X'^\ad})$ under the comparison
map (which is in fact an isomorphism)
\[H^{2q}_\dr(\cX')=H^{2q}(\cX',\Omega^\bullet_{\cX'})\to
H^{2q}(X'^\ad,\Omega^\bullet_{X'^\ad}).\] As $\cl_\dr(\cZ)=0$, we have
$\cl_\dr(f^*\cZ)=0$ and hence $f^*\zeta_\ad\cup f^*\omega_\ad=0$. Thus,
$f^*\zeta\cup f^*\omega=0$, and in particular,
\[\int_{\cZ_\ra^\an}\omega=\deg(f)^{-1}\int_{(f^*\cZ)_\ra^\an}f^*\omega=0.\]
The theorem is proved.

\textbf{Step 3.} Now we fucus on the claim in Step 2. For an integral model
$\cY$ of $X$, define $\sK^q_{X,\cY}$ to be the sheaf on $\cY_s$ associated to
the presheaf
\[\cU\mapsto \varinjlim_{\pi^{-1}\cU\subset U} K^M_q(\cO_X(U))\otimes\bQ,\quad\cU\subset\cY_s\]
where the colimit is taken over all open neighborhoods $U$ of $\pi^{-1}\cU$
in $X$. We remark that there is a canonical morphism
$\sK^q_{X,\cY}\to\gamma_{\cY*}\gamma^{-1}_X\sK^q_X$ which is in general not
an isomorphism, where $\gamma_\cY\colon X^\ad\to\cY_s$ is the induced
continuous map. Put
$\Omega^{\dag,q}_{X,\cY}=\gamma_{\cY*}\gamma^{-1}_X\Omega^q_X$. Then we have
a complex of sheaves of $k$-vector spaces
$(\Omega^{\dag,\bullet}_{X,\cY},\rd)$ on $\cY_s$. And we have a canonical map
\[\lambda_{X,\cY}^q\colon \sK^q_{X,\cY}\to \Omega^{\dag,q,\cl}_{X,\cY}/\rd\Omega^{\dag,q-1}_{X,\cY}\]
similar to Definition \ref{de:kernel}. Denote by $\sL_{X,\cY}^q$ the image
sheaf of $\lambda_{X,\cY}^q$ in the above map. Since sheafification commutes
with pullback and taking colimit, we have a canonical isomorphism
\[\varinjlim_\cY\gamma_\cY^{-1}\sK^q_{X,\cY}\simeq\gamma^{-1}_X\sK^q_X\] of
sheaves on $X^\ad$, where the filtered colimit is taken over all integral
models $\cY$ of $X$. On the other hand, we have an obvious isomorphism
\[\varinjlim_\cY\gamma_\cY^{-1}\Omega^{\dag,\bullet}_{X,\cY}\simeq\gamma_X^{-1}\Omega^\bullet_X.\]
Passing to the quotient, we have a canonical isomorphism
\[\varinjlim_\cY\gamma_\cY^{-1}\sL_{X,\cY}^q\simeq\gamma_X^{-1}\sL_X^q.\]

Note that originally, $\omega$ belongs to $H^{n-q}(X,\sT_X^{n-q})\simeq
H^{n-q}(X,\sL_X^{n-q})$. By a similar argument for Lemma \ref{le:top}, there
is an integral model $\cY$ of $X$ such that $\omega$ is in the image of the
canonical map
\[H^{n-q}(\cY_s,\sL_{X,\cY}^{n-q})
\to H^{n-q}(X^\ad,\gamma^{-1}\sL_X^{n-q})\simeq H^{n-q}(X,\sL_X^{n-q}).\] By
\cite{dJ96}*{Theorem 8.2}, possibly after taking further finite extension of
$k$ inside $\bC_p$, we have a projective strictly semi-stable scheme $\cY'$
over $k^\circ$ with an alteration $f\colon\cY'\to\cY$. Put $X'=\cY'^\an$. If
we put $\Omega^q_{X',\cY'}=\gamma_{\cY'*}\Omega^q_{X'^\ad}$, then
$(\Omega^\bullet_{X',\cY'},\rd)$ is a complex of sheaves of $k$-vector spaces
on $\cY'_s$ and we have a canonical map
$(\Omega^{\dag,\bullet}_{X',\cY'},\rd)\to (\Omega^\bullet_{X',\cY'},\rd)$.
The claim in Step 2 will follow if we can show that the composite map
\begin{multline}\label{eq:trivial}
H^{n-q}(\cY'_s,\sL_{X',\cY'}^{n-q}) \to
H^{n-q}(\cY'_s,\Omega^{\dag,n-q,\cl}_{X',\cY'}/\rd\Omega^{\dag,n-q-1}_{X',\cY'})\\
\to H^{n-q}(\cY'_s,\Omega^{n-q,\cl}_{X',\cY'}/\rd\Omega^{n-q-1}_{X',\cY'})
\to H^{2n-2q+1}(\cY'_s,\tau_{\leq n-q-1}\Omega^\bullet_{X',\cY'})
\end{multline}
is zero. The advantage of $(\Omega^\bullet_{X',\cY'},\rd)$ is that the
\emph{entire} complex admits a canonical Frobenius action. More precisely, we
fix a uniformizer $\varpi$ of $k$; let $\cY'^\times_s$ be the log scheme
$\cY'_s$ equipped with log structure as in \cite{HK94}*{(2.13.2)}; and $\Spf
W(\widetilde{k})^\times$ be the formal log scheme $\Spf W(\widetilde{k})$
equipped with log structure $1\mapsto 0$. Here, we use Zariski topology
in the construction of log schemes and log crystal sites instead of \'etale one in \cite{HK94}.
There is a canonical morphism
$u\colon(\cY'^\times_s/\Spf W(\widetilde{k})^\times)_{\logcris}\to\cY'_s$ of
sites. Then by (the proof of) \cite{HK94}*{Theorem 5.1}, we have a canonical isomorphism
\[\rR u_*\cO_{\cY'^\times_s/\Spf W(\widetilde{k})^\times}^{\logcris}\otimes_{W(\widetilde{k})} k\simeq(\Omega^\bullet_{X',\cY'},\rd)\]
in the derived category of abelian sheaves on $\cY'_s$,
where $\cO_{\cY'^\times_s/\Spf W(\widetilde{k})^\times}^{\logcris}$ denotes the
structure sheaf in the log crystal site.
Since $\cO_{\cY'^\times_s/\Spf
W(\widetilde{k})^\times}$ admits a Frobenius action over
$\Spec\widetilde{k}$, we obtain a Frobenius action on the entire complex
$(\Omega^\bullet_{X',\cY'},\rd)$ in the derived category.

For $w\in\bZ$, denote by
$(\Omega^{q,\cl}_{X',\cY'}/\rd\Omega^{q-1}_{X',\cY'})_w$ the maximal subsheaf
of $\Omega^{q,\cl}_{X',\cY'}/\rd\Omega^{q-1}_{X',\cY'}$ generated by sections
of generalized weight $w$. We claim that
\begin{enumerate}[(a)]
  \item the image of the canonical map
      $\sK^q_{X',\cY'}\to\sL_{X',\cY'}^q\to
      \Omega^{q,\cl}_{X',\cY'}/\rd\Omega^{q-1}_{X',\cY'}$ is contained in
      the subsheaf
      $(\Omega^{q,\cl}_{X',\cY'}/\rd\Omega^{q-1}_{X',\cY'})_{2q}$ for
      every $q$;

  \item the image of the canonical map
      $\Omega^{\dag,q,\cl}_{X',\cY'}/\rd\Omega^{\dag,q-1}_{X',\cY'}\to
      \Omega^{q,\cl}_{X',\cY'}/\rd\Omega^{q-1}_{X',\cY'}$ is contained in
      the subsheaf
      $\bigoplus_{w=0}^{2q}(\Omega^{q,\cl}_{X',\cY'}/\rd\Omega^{q-1}_{X',\cY'})_w$
      for every $q$.
\end{enumerate}
Then the triviality of the map \eqref{eq:trivial} follows easily from an
argument of spectral sequences. In fact, we have a spectral sequence
$E^{p,q}_r$ abutting to $H^\bullet(\cY'_s,\Omega^\bullet_{X',\cY'})$ equipped
with a Frobenius action, such that
$E^{p,q}_2=H^p(\cY'_s,\Omega^{q,\cl}_{X',\cY'}/\rd\Omega^{q-1}_{X',\cY'})$.
By (a) and (b), the restriction of all differentials $\rd^{n-q,n-q}_r$ in the
spectral sequence with $r\geq 2$ to the image of the map
\[H^{n-q}(\cY'_s,\sL_{X',\cY'}^{n-q})\to
H^{n-q}(\cY'_s,\Omega^{n-q,\cl}_{X',\cY'}/\rd\Omega^{n-q-1}_{X',\cY'})=E^{n-q,n-q}_2\]
is zero by weight consideration. Thus, \eqref{eq:trivial} is the zero map.

\textbf{Step 4.} The last step is devoted to the proof of the two claims (a)
and (b) in Step 3. We remark that they are not formal consequences of Theorem
\ref{th:1}.

By definition, $\Omega^{q,\cl}_{X',\cY'}/\rd\Omega^{q-1}_{X',\cY'}$ is the
sheaf on $\cY'_s$ associated to the presheaf $\cU\mapsto
H^q_\dr(\pi^{-1}\cU)$, and
$\Omega^{\dag,q,\cl}_{X',\cY'}/\rd\Omega^{\dag,q-1}_{X',\cY'}$ is the sheaf
on $\cY'_s$ associated to the presheaf $\cU\mapsto H^q_\dr(X',\pi^{-1}\cU)$
by \cite{Berk07}*{Lemma 5.2.1}. We check (a) and (b) on stalks and thus fix a
point $x\in\cY'_s$.

To prove (a), it suffices to consider the case where $q=1$. Let
$\cU\subset\cY'$ be an open affine neighborhood of $x$. Take
$f\in\cO^*(X',\pi^{-1}\cU)$; we have to show that the image of $\frac{\rd
f}{f}$ in $H^1_\dr(\pi^{-1}\cU_s)$ is of generalized weight $2$ for a
possibly smaller open neighborhood $\cU$ of $x$. First, we may replace $f$ by
the restriction of an (algebraic) function $f\in\cO(\cU_k)$ without changing
the image of $\frac{\rd f}{f}$ in $H^1_\dr(\pi^{-1}\cU_s)$. Since $\cY'$ is
projective, we may choose a closed embedding
$\cY'\hookrightarrow\bP^N_{k^\circ}$ into a projective space. Choose an open
affine neighborhood $\cV$ of $x$ in $\bP^N_{k^\circ}$ such that
$\cV\cap\cY'\subset\cU$ and $f\res_{\cV\cap\cY'}=g\res_{\cV\cap\cY'}$ for
some $g\in\cO^*(\cV_k)$. Since $\bP^N_{k^\circ}$ is smooth, by Remark
\ref{re:decomp}, $\frac{\rd g}{g}$ belongs to $H^1_\rig(\cV_s/k)_2$ and thus
its image in $H^1_\dr(\pi^{-1}\cV_s)$ is of generalized weight $2$. By
functoriality of log crystal sites for the morphism $\cY'\to\bP^N_{k^\circ}$,
we conclude that the image of $\frac{\rd f}{f}$ in
$H^1_\dr(\pi^{-1}(\cV\cap\cY')_s)$ is of generalized weight $2$. Here,
$\pi^{-1}(\cV\cap\cY')_s$ is the inverse image in $X'$.

Claim (b) is a consequence of \cite{GK05}*{Theorem 0.1} and
\cite{Chi98}*{Theorem 2.3}. In fact, we have a functorial map of spectral
sequences $\pres{\prime}E^{p,q}_r\to\pres{\prime\prime}E^{p,q}_r$ abutting to
$H^\bullet_\dr(X',\pi^{-1}\cU)\to H^\bullet_\dr(\pi^{-1}\cU)$ with the first
page being
\[\pres{\prime}E^{p,q}_1=H^q_\rig(\cU_s^{(p)}/\Spf
W(\widetilde{k})^\times)\otimes_{W(\widetilde{k})[1/p]}k
\to\pres{\prime\prime}E^{p,q}_1=H^q_{\logcris}(\cU_s^{(p)}/\Spf
W(\widetilde{k})^\times)\otimes_{W(\widetilde{k})}k,\] where $\cU_s^{(p)}$ is
the disjoint union of irreducible components of $\cU_s^{[p]}$, equipped with
the induced log structure from $\cY'^\times_s$. By \cite{GK05}*{Theorem 3.1,
Lemma 4.6} and \cite{Chi98}*{Theorem 2.3}, we know that the weights of (the
finite dimensional $k$-vector space) $\pres{\prime}E^{p,q}_1$ are in the
range $[q,2q]$, and thus the weights of $H^q_\dr(X',\pi^{-1}\cU)$ are in the
range $[0,2q]$.
\end{proof}

\appendix

\end{document}